\documentclass[twoside]{amsart}

\numberwithin{equation}{section}
\usepackage{amsmath,amssymb,amsfonts,amsthm,latexsym}
\usepackage{graphicx,color,enumerate}
\usepackage[margin=1in]{geometry}
\usepackage{longtable}
\usepackage{appendix}
\usepackage{xcolor}
\usepackage{stmaryrd}
\usepackage{mleftright}
\usepackage{breqn}
\usepackage{url}

\newtheorem{theorem}{Theorem}[section]
\newtheorem{lemma}[theorem]{Lemma}
\newtheorem{proposition}[theorem]{Proposition}
\newtheorem{corollary}[theorem]{Corollary}

\theoremstyle{definition}
\newtheorem{definition}[theorem]{Definition}
\newtheorem{example}[theorem]{Example}
\theoremstyle{remark}
\newtheorem{remark}[theorem]{Remark}

\newcommand{\C}{\mathbb{C}}

\newcommand{\Q}{\mathbb{Q}}
\newcommand{\R}{\mathbb{R}}

\newcommand{\Z}{\mathbb{Z}}

\newcommand{\SL}{\operatorname{SL}}

\newcommand{\OO}{\operatorname{O}}

\newcommand{\id}{\operatorname{id}}

\newcommand{\Bild}{\operatorname{im}}

\newcommand{\diag}{\operatorname{diag}}
\newcommand{\Det}{\operatorname{Det}}

\newcommand{\git}{/\!\!/}
\newcommand{\bs}{\boldsymbol}
\newcommand{\kk} {\bs k}

\newcommand{\calC}{\mathcal{C}}

\newcommand{\dR}{\mathrm{d}}

\newcommand{\Spec}{\operatorname{Spec}}

\newcommand{\sgn}{\operatorname{sgn}}

\newcommand{\Hom}{\operatorname{Hom}}
\newcommand{\UnSh}{\operatorname{Sh}^{-1}}

\newcommand{\SA}{\operatorname{S}}
\newcommand{\Der}{\operatorname{Der}}
\newcommand{\Aut}{\operatorname{Aut}}
\newcommand{\Nambu}{\operatorname{Nambu}}
\newcommand{\dN}{\delta_{\operatorname{Nambu}}}
\newcommand{\Ham}{\operatorname{Hambu}}
\newcommand{\Alt}{\operatorname{Alt}}
\newcommand{\End}{\operatorname{End}}

\begin{document}

\title{Higher Form Brackets for even Nambu-Poisson Algebras}

\author[H.-C.~Herbig]{Hans-Christian Herbig}
\address{Departamento de Matem\'{a}tica Aplicada, Universidade Federal do Rio de Janeiro,
Av. Athos da Silveira Ramos 149, Centro de Tecnologia - Bloco C, CEP: 21941-909 - Rio de Janeiro, Brazil}
\email{herbighc@gmail.com}

\author[A. M. Chaparro Castañeda]{Ana María Chaparro Castañeda}
\address{Departamento de Matem\'{a}tica Aplicada, Universidade Federal Fluminense,
Rua Alexandre Moura, 8 Coseac - Bloco C - térreo - São Domingos, CEP:24210-200, Niterói - RJ, Brazil}
\email{amchaparroc@unal.edu.co}

\keywords{Nambu-Poisson algebras, $P_\infty$-algebra, cotangent complex, $L_\infty$-algebroids, Lie-Rinehart $m$-algebras, Nambu connections and curvature}
\subjclass[2010]{primary 17B63,	secondary 13D02, 58A50, 17B66}

\begin{abstract} 
Let $\boldsymbol{k}$ be a field of characteristic zero and $A=\boldsymbol{k}[x_{1},...,x_{n}]/I$ with $I=(f_{1},...,f_{k})$ be an affine algebra. We study Nambu-Poisson brackets on $A$ of arity $m\geq 2$, focusing on the case when $m$ is even. We construct an $L_{\infty}$-algebroid on the cotangent complex $\mathbb{L}_{A|\boldsymbol{k}}$, generalizing previous work on the case when $A$ is a Poisson algebra. This structure is referred to as the higher form brackets. The main tool is a $P_{\infty}$-structure on a resolvent $R$ of $A$. These $P_{\infty}$- and $L_{\infty}$-structures are merely $\Z_2$-graded for $m\neq 2$.
We discuss several examples and propose a method to obtain new ones that we call the outer tensor product.
We compare our higher form brackets with the form bracket of Vaisman. We introduce the notion of a Lie-Rinehart $m$-algebra, the form bracket of a Nambu-Poisson bracket of even arity being an example. We find a flat Nambu connection on the conormal module. 
\end{abstract}

\maketitle
\tableofcontents
Before talking about the mathematical content of the paper we fix our notations and conventions.
 Through-out the paper $\kk$ denotes a field of characteristic zero. 
A reader with differential geometric or physics background might think of the case $\kk=\R$, whereas an algebraic geometer might think of $\kk=\C$. Equally relevant is the likely preference $\kk=\Q$ of a computationally inclined reader. The paper is about certain algebraic aspects of singularities and lies at the intersection of those disciplines. By a $\kk$-algebra we mean a unital associative commutative algebra. For a $\kk$-algebra $A$ we put $\Der(A):=\{X\in A\to A\mid X\mbox{ is $\kk$-linear such that }X(ab)=X(a)b+aX(b) \ \forall a,b\in A\}$. We use the notions vector field and derivation interchangeably to refer to elements of $\Der(A)$.

In order to please the algebraists we decided to put indices for the coordinates downstairs. Yet in order to suit the differential geometers and physicists among the readers we use the convention of Ricci calculus: we sum over repeated pairs of indices, one being upstairs (in our case the covariant index) the other one being downstairs (the contravariant index). Sometimes we use Einstein's convention to keep formulas compact, i.e., we omit the sum signs for index sums. We are admitting that we ourselves are not able to do calculations without the Ricci calculus. We are aware of the fact that each of the decisions makes the other half of the readership unhappy\footnote{We are following here strictly the principles of Gottfried Wilhelm Leibniz's \emph{pre-established harmony} \cite{Monads}.}. Those readers who cannot live without contravariant indices upstairs are welcome to switch lower and upper indices in all the formulas. When annotating powers, we avoid using exponents if possible to not cause confusion, i.e., we write for example $xx$ instead of $x^2$. When referring to algebraic structures that are based on the Koszul sign rule, we use the prefix \emph{super} as it is customary in physics, e.g., we say supercommutative algebra or Lie superalgebra. The prefix graded that is often used in this context can be misleading.

\section{Introduction}\label{sec:intro}

Most of the geometric structures that are studied in differential geometry (such as Riemannian structures etc.) are difficult to make sense of for singular varieties. The principal reason is that the covariant and contravariant tensors do not form sections of vector bundles. Poisson geometries, however, can be defined purely in terms of function algebras. Singular Poisson algebras such as symplectic reductions (see e.g. \cite{Witten,KN90,HHScompositio} or nilpotent orbit closures (see, e.g., \cite{BaohuaFu}) play a prominent role in modern mathematics. Another structure of this type is Nambu-Poisson geometry \cite{Nambu,takhtajan1994foundation,vaisman1999survey}. The goal of this paper is to present examples of singular Nambu-Poisson algebras and to generalize previous work \cite{higherKoszul} on the higher Koszul brackets from the Poisson case to the case of even Nambu-Poisson structures. To our knowledge singular Nambu-Poisson structures have not been studied elsewhere, even though they arise quite naturally in invariant theory.

Before we can explain our main results we recall some basic notions of algebraic Nambu-Poisson geometry. 
\begin{definition}\label{def:NPbracket}
Let $A$ be a $\kk$-algebra.
A \emph{Nambu-Poisson structure} of arity $m$ on $A$ is a $\kk$-multilinear antisymmetric map $\{\:,...,\:\}$ from the direct product of $m$ copies of $A$ to $A$
\begin{align*}
A\times\cdots\times A&\rightarrow A\\
                                 (a_{1},...,a_{m})&\mapsto\{a_{1},...,a_{m}\}
\end{align*}
which satisfies the following conditions
\begin{enumerate}
\item \label{item:NLeib} Leibniz rule:
\begin{equation*}\label{eq:NLeibniz}
\{a_{1},...,a_{m-1},ab\}=a\{a_{1},...,a_{m-1},b\}+b\{a_{1},...,a_{m-1},a\}
\end{equation*}
\item \label{item:FI} Fundamental identity:
\begin{equation*}\label{eq:Nfund}
\{a_{1},...,a_{m-1},\{b_{1},...,b_{m}\}\}=\sum_{l=1}^{m}\{b_{1},...,b_{l-1},\{a_{1},...,a_{m-1},b_{l}\},...,b_{m}\}
\end{equation*}
\end{enumerate}
for any $a_{1},...,a_{m-1},b_{1},...,b_{m},a,b\in A$. Occasionally, we use the following notation $X_{a_{1},...,a_{m-1}}$ for the vector field $\{a_{1},...,a_{m-1},\:\}$. Vector fields of this type are referred to as \emph{Hamiltonian vector fields}. They form a sub Lie-Rinehart algebra of $\Der(A)$ written as $\Der^{\Ham}(A)$.
The bracket $\{\:,...,\:\}$ is called a \emph{Nambu-Poisson bracket} while $(A,\{\:,...,\:\})$ is called a \emph{Nambu-Poisson algebra}.
\end{definition}
\begin{definition}\label{def:NPideal}
Let $(S,\{\:,...,\:\})$ be a Nambu-Poisson algebra. An ideal $I$ in $S$ is called a \emph{Nambu-Poisson ideal} if $\{I,S,\dots,S\}\subseteq I$.
\end{definition}
In the above situation we have automatically a Nambu-Poisson bracket on $A$ induced from the Nambu-Poisson bracket on $S$. In fact, if $a_i=g_i+I\in A$ and in $g_i\in S$ then $\{a_{1},...,a_{m}\}:=\{g_{1},\dots,g_{m}\}+I$.

If $\kk$-algebra $A$ is affine, i.e., $A=S/I$ where $S=\kk[x_{1},...,x_{n}]$ and $I$ is a Nambu-Poisson ideal in $S$, we use the following notation 
\begin{align*}
\Pi_{i_{1}...i_{m}}=\{x_{i_1}+I,\dots,x_{i_m}+I\}\in A.
\end{align*}
Then $\{\:,...,\:\}$ is uniquely determined by $\Pi_{i_{1}...i_{m}}$ via the following formula 
\begin{equation}\label{eq:prop1}
\{a_{1},...,a_{m}\}=\sum_{i_{1},...,i_{m}=1}^{n}\Pi_{i_{1}...i_{m}}\dfrac{\partial g_{1}}{\partial x_{i_{1}}}\cdots\dfrac{\partial g_{m}}{\partial x_{i_{m}}}+I.
\end{equation}
Here $a_i=g_i+I\in A$ and in $g_i\in S$. More generally, we have the following (see \cite{takhtajan1994foundation} or \cite{vaisman1999survey}).

\begin{proposition}
If $\Pi_{i_{\sigma(1)}...i_{\sigma(m)}}=\sgn(\sigma)\Pi_{i_{1}...i_{m}}\in S$ for each $\sigma\in \Sigma_m:=\Aut(\{1,\dots,m\})$, then
\begin{equation*}
\{g_{1},...,g_{m}\}:=\sum_{i_{1},...,i_{m}=1}^{n}\Pi_{i_{1}...i_{m}}\dfrac{\partial g_{1}}{\partial x_{i_{1}}}\cdots\dfrac{\partial g_{m}}{\partial x_{i_{m}}}
\end{equation*}
 with in $g_i\in S$, defines a bracket $\{\:,\dots,\:\}$ on $S$ which satisfies all conditions of Definition \ref{def:NPbracket} except (\ref{item:FI}).
It defines a Nambu-Poisson bracket on $S$ if and only if for all $i_{1},\dots,i_{m-1},j_{1},\dots,j_{m}\in\{1,\dots n\}$ 
\begin{equation}\label{eq:FIccord}
\{x_{i_{1}},...,x_{i_{m-1}},\{x_{j_{1}},...,x_{j_{m}}\}\}=\sum_{l=1}^{m}\{x_{j_{1}},\dots,x_{j_{l-1}},\{x_{i_{1}},...,x_{i_{m-1}},x_{j_{l}}\},x_{j_{l+1}},...,x_{j_{m}}\}
\end{equation}
is satisfied.
This amounts to the conditions
\begin{equation}\label{eq:FIX}
\sum_{s=1}^{n}\Pi_{i_{1}...i_{m-1}s}\dfrac{\partial \Pi_{j_{1}\dots j_{l-1}s j_{l+1}\dots j_{m}}}{\partial x_{s}}=\sum_{s=1}^{n}\sum_{l=1}^{m}\Pi_{j_{1}...j_{l-1}sj_{l+1}...j_{m}}\dfrac{\partial \Pi_{i_{1}...i_{m-1}j_{l}}}{\partial x_{s}}.
\end{equation}
and
\begin{equation}\label{eq:eq44}
    \sum _{l=1}^{m} \Pi_{j_{1} \dotsc j_{l-1} ij_{l+1} \dotsc j_{m-1}} \Pi_{ji_{2} \dotsc i_{m-1} j_{l}} +\Pi_{j_{1} \dotsc j_{l-1} jj_{l+1} \dotsc j_{m-1}} \Pi_{ii_{2} \dotsc i_{m-1} j_{l}}=0,
\end{equation}
for all indices  $i,j\in\{1,\dots,n\}$.

If \eqref{eq:FIccord}  hold merely modulo $I$, then  \eqref{eq:prop1} defines a Nambu-Poisson bracket on $A$. This amounts to \eqref{eq:FIX} and \eqref{eq:eq44} being true modulo $I$.
\end{proposition}
Abusing notation, we use the same symbol for the Hamiltonian vector fields $X_{i_{1}...i_{m-1}}=X_{x_{i_{1}},...,x_{i_{m-1}}}\in \Der(S)$ and $X_{i_{1}...i_{m-1}}=X_{x_{i_{1}}+I,...,x_{i_{m-1}}+I}\in \Der(A)$.
Note that \eqref{eq:eq44} is equivalent to the decomposability of the antisymmetric tensor $\Pi_{i_{1} \dotsc i_{m}}$ (see \cite[Theorem]{vaisman1999survey}).

By Hilbert basis theorem, the ideal $I$ in $S$ is generated by a finite set of elements, lets say  $f_{1},\dots,f_{k}\in S$. The condition of being a Nambu-Poisson ideal can be rephrased as follows:
\begin{align}\label{eq:Zs}
\{x_{i_{1}},\dots,x_{i_{m-1}},f_{\mu}\}=\sum_{\nu}Z_{i_{1}\cdots i_{m-1}\mu}^{\nu}f_{\nu},
\end{align}
for some $Z_{i_{1}\cdots i_{m-1}\mu}^{\nu}\in S$. Observe that neither $f_{1},\dots,f_{k}\in S$ nor the coefficients $Z_{i_{1}\cdots i_{m-1}\mu}^{\nu}$ are unique. Throughout the paper we fix a choice of them.
If  $Z_{i_{1}\cdots i_{m-1}\mu}^{\nu}$ are all zero then we say that the generators $f_{1},\dots,f_{k}$ are \emph{Casimir}.

To understand the statements of our main results Theorem \ref{thm:Pinfty} and the Corollaries \ref{cor:Linftyalgeboid}, \ref{cor:formbr} the reader might want to consult Section \ref{sec:form} for the definitions of a $\Gamma$-graded $P_{\infty}$-structure and of a $\Gamma$-graded $L_{\infty}$-algebroid. If $m=2$ the abelian group of the grading is $\Gamma=\Z$, while for even $m>2$ it is merely $\Gamma=\Z_2$. 

\begin{theorem}\label{thm:Pinfty}
With the notations above let the arity $m\geq 2$ be even.
Let $I\subset S=\kk[x_{1},\dots,x_{n}]$ be a Nambu-Poisson ideal, let $A=S/I$ and let $f_{1},\dots,f_{k}$ be generators for $I$. Let $R$ be a resolvent of $S\rightarrow A$ on the generators $f_{1},\dots,f_{k}$. Then there is the structure of a $\Gamma$-graded $P_{\infty}$-algebra $(\{\:,\dots,\:\}_{l})_{l\geq 1}$ on $(R,\partial)$ such that $\partial=\{\:\}_{1}$ and the quasi-isomorphism $R\rightarrow A$ is compatible with the brackets. If the generators $f_{1},...,f_{k}$ are Casimir and form a complete intersection, the $P_{\infty}$-algebra structure is trivial in the sense that the only nonzero Nambu-Poisson brackets are $\partial=\{\:\}_{1}$ and the $m$-ary Nambu-Poisson bracket $\{\:,\dots,\:\}$ in $S$.
\end{theorem}
\begin{corollary}\label{cor:Linftyalgeboid}
Under the assumptions of Theorem \ref{thm:Pinfty} there is the structure of a $\Gamma$-graded $L_{\infty}$-algebroid with $m$-ary anchor on the cotangent complex $\mathbb{L}_{A|\kk}\rightarrow\Omega_{A|\kk}$ that is compatible with the form bracket and the anchor (cf. Section \ref{sec:form}). The brackets of this structure are uniquely defined by
\begin{align*}
    [\dR\phi_1,\dots,\dR\phi_m]_m=\dR\{\phi_1,\dots,\phi_m\}_m,
\end{align*}
for $\phi_1,\dots,\phi_m\in R$. The anchor $\rho(\dR\phi_1,\dots,\dR\phi_{m-1}, ):=\left(a\mapsto\{\phi_1,\dots,\phi_{m-1},a\}\right)\in \Der(A)$ is nonzero if and only if $\phi_1,\dots,\phi_m\in A$. In this case it reduces to the Nambu-Poisson bracket: $\rho(\dR\phi_1,\dots,\dR\phi_{m-1},a)=\{\phi_1,\dots,\phi_m,a\}$ for $a\in A$.
If the generators $f_{1},...,f_{k}$ are Casimir and form a complete intersection, then the $\Gamma$-graded  $L_{\infty}$-algebroid structure is trivial in the sense that the only nonzero Lie bracket is given by $[\dR x_{i_{1}},\dots,\dR x_{i_{m}}]=\dR \{x_{i_{1}},\dots,x_{i_{m}}\}$, for $i_{1},\dots,i_{m}\in \{1,\dots,n\}$.
\end{corollary}

The statement is already nontrivial on the level of brackets on $\Omega_{A|\kk}$. It was noted by I.Vaisman \cite{vaisman1999survey} that the form bracket of non-exact forms in general violates the Filippov identity. From the $\Gamma$-graded $L_\infty$-algebroid, however, we can deduce in a natural way a set of axioms that hold for form brackets (see the definition of a Lie-Rinehart $m$-algebra Section \ref{sec:form}). It is essential to assume $m$ to be even.

\begin{corollary}\label{cor:formbr}
Let $m\geq 2$ be an even integer.
Let $I$ be a Nambu-Poisson ideal in the Nambu-Poisson algebra $( S,\{\ ,\dotsc ,\})$ and  $A=S/I$ be the quotient Nambu-Poisson algebra. 
Assume that the zero homology of the cotangent complex $\operatorname{D}_{0}( A|\boldsymbol{k} ,A)$ is isomorphic to $\Omega_{A|\kk}$ (this is the case when $I$ is reduced).
Then
$[ \dR a_{1} ,\dotsc ,\dR a_{m}] :=\dR \{a_{1} ,\dotsc ,a_{m}\}$ and 
$\rho (\dR a_{1},\dots,\dR a_{m-1}, \ ) :=X_{a_{1} \dotsc a_{m-1}}$ 
for $a_1,\dots, a_m$
define a Lie-Rinehart $m$-algebra structure on the $A$-module of Kähler differentials $\Omega _{A|\kk}$.
\end{corollary}

The plan of the paper is as follows. In Section \ref{sec:examples} we elaborate classes of examples of singular Nambu-Poisson algebras: diagonal ones, determinantal ones and those coming from invariant theory. Moreover, we introduce the \emph{outer tensor product} that permits to concoct Nambu-Poisson algebras from smaller building blocks. In Section \ref{sec:MC} we show that the tensors of Equation \eqref{eq:Zs} satisfy a certain Maurer-Cartan equation. In Section \ref{sec:conn} we show that this Maurer-Cartan equation can be interpreted that Nambu connection on the conormal module $I/I^2$ is flat. We discuss more general Nambu connections and their curvature.
In Section \ref{sec:form} we define $\Gamma$-graded $L_\infty$- algebroids, $\Gamma$-graded $P_\infty$-algebras as well as Lie-Rinehart $m$-algebras and elementary facts relating them.
In Section \ref{sec:resolventetc} we quickly review the construction of resolvents and of the cotangent complex. In Section \ref{sec:main} we prove the main results of the paper: Theorem \ref{thm:Pinfty}, and the Corollaries \ref{cor:Linftyalgeboid} and \ref{cor:formbr}. In Section \ref{sec:complete} we discuss the special cases of complete and locally complete intersection. This also includes the case of regular Nambu-Poisson algebras. In Section \ref{sec:outlook} we discuss some topics that might be relevant for further investigations of singular Nambu-Poisson geometry. Section \ref{sec:emp} is devoted to computer calculations that we did for a couple of examples. We used Mathematica \cite{Mathematica}, Macaulay2 \cite{M2} and Magma \cite{Magma}.

\vspace{2mm}

\paragraph{\emph{Acknowledgments}}
This paper is part of the PhD thesis of AMCC. She acknowledges financial support of CAPES (Coordenação de Aperfeiçoamento de Pessoal de Nível Superior). We thank Miles Reid for recommendations concerning computer algebra.  We acknowledge the generous support of Chris Seaton whose advice was crucial when adapting his Mathematica code and who counseled us while working out Proposition \ref{prop:Nambu} and in relation to the problems in Example \ref{ex:E}. We profited from explanations of José Antonio Vallejo concerning the problems with the Filippov identity for form brackets and gratefully acknowledge advice of Theodore Voronov on higher derived brackets.

\section{Examples of Singular Nambu-Poisson algebras}
\label{sec:examples}
In this section we elaborate examples of singular Nambu-Poisson algebras, with the aim of identifying those who lead to interesting $P_\infty$-structures. This is not so much the case if the Nambu-Poisson algebras are generated by Casimirs. We also indicate a procedure, the \emph{outer tensor product} that can be used to build up Nambu-Poisson algebras from building blocks with lower arities. 

\subsection{Diagonal Bracket}
Let $S=\kk[x_{1},...,x_{n}]$ and $c_{i_{1}i_{2}\dots i_{m}}$ be the totally anti-symmetric tensor with entries in $\kk$. Consider the \emph{diagonal bracket}
\begin{equation*}
\Pi_{i_{1}...i_{m}}=\{x_{i_{1}},...,x_{i_{m}}\}:=c_{i_{1}i_{2}\dots i_{m}}x_{i_{1}}x_{i_{2}}\cdots x_{i_{m}}.
\end{equation*}
\begin{lemma}
If $m$ is even, $\{\:,\dots,\:\}$ is a Nambu-Poisson bracket.
\end{lemma}
\begin{proof}
For fixed $i$, we have $\dfrac{\partial \Pi_{j_{1}...j_{m}}}{\partial x_{i}}= c_{j_{1}\dots j_{m}}x_{j_{1}}\cdots \widehat{x_{i}}\cdots x_{j_{m}}$ if $i=j_{l}$ for some $l=1,...,m$ or zero otherwise. Then 
\begin{equation*}
\sum_i\Pi_{i_{1}...i_{m-1}i}\dfrac{\partial \Pi_{j_{1}...j_{m}}}{\partial x_{i}}=c_{i_{1}\dots i_{m-1}j_l}c_{j_{1}\dots j_{m}}x_{i_{1}}\cdots x_{i_{m-1}}x_{j_{1}}\cdots x_{j_{m}}.
\end{equation*}
So equation \eqref{eq:FIX} can be written as
\begin{equation*}
\biggl(\sum_{k=1}^{m}-c_{i_{1}\dots i_{m-1}j_{k}}c_{j_{1}\dots j_{m}}+\sum_{i_{m}=1}^{n}\sum_{l=1}^{m}c_{i_{1}...i_{l-1}i_{m}i_{l}...i_{m-1}}
c_{j_{1}...j_{m-1}i_{l}}\biggr)x_{j_{1}}\cdots x_{j_{m}}x_{i_{1}}\cdots x_{i_{m-1}}=0
\end{equation*}
which implies that the coefficients must satisfy
\begin{equation}\label{eq3}
\sum_{k=1}^{m}-c_{i_{1}\dots i_{m-1}j_{k}}c_{j_{1}\dots j_{m}}+\sum_{i_{m}=1}^{n}\sum_{l=1}^{m}c_{i_{1}...i_{l-1}i_{m}i_{l}...i_{m-1}}
c_{j_{1}...j_{m-1}i_{l}}=0.
\end{equation}
According to \cite[Lemma 2]{takhtajan1994foundation} a totally antisymmetric scalar tensor is Nambu-Poisson and \eqref{eq3} follows.  
\end{proof}
\begin{lemma}
Let $\boldsymbol m=(m_1,\dots,m_n)\in\Z^n_{\geq 0}$. 
Consider $f(x_1,\dots,x_n)=\boldsymbol{x}^
{\boldsymbol{m}}=\prod_{i=1}^{n}x_{i}^{m_{i}}$. Then $\{x_{i_{1}},...,x_{i_{m-1}},f\}$ is proportional to $f$. In fact,  $g=\sum_{i=1}^{n}m_{i}c_{i_{1}...i_{m-1},i}x_{i_{1}}\cdots x_{i_{m-1}}$ satisfies $\{x_{i_{1}},...,x_{i_{m-1}},f\}=gf$.
\end{lemma}
\begin{proof}
We find that 
\begin{align*}
\{x_{i_{1}},...,x_{i_{m-1}},f\}&=\sum_{i=1}^{n} \Pi_{i_1\dots i_{m-1}i}{\partial\boldsymbol x^{\boldsymbol m}\over \partial x_i}=\sum_{i=1}^{n} c_{i_{1}...i_{m-1}i}x_{j_{1}}\cdots x_{j_{m-1}}x_i\: m_ix_i^{-1}\boldsymbol x^{\boldsymbol m}
=gf
\end{align*}
\end{proof}

A monomial ideal $I=(f_1,\dots,f_k)$ is characterized by vectors of exponents  $\boldsymbol n_\mu=(n_{1,\mu},\dots,m_{n,\mu})\in\Z^n_{\geq 0}$ for $\mu=1,\dots,k$ with $f_{\mu}=\boldsymbol x^{\boldsymbol n_{\mu}}$. From the lemma we see that the $Z$s in \eqref{eq:Zs} can be chosen by
\begin{align*}
    Z_{i_{1},...,i_{m-1},\mu}^{\nu}=\delta_{\mu}^{\nu}\sum_{i=1}^{n}n_{i,\mu}c_{i_{1}...i_{m-1},i}x_{i_{1}}\cdots x_{i_{m-1}}
\end{align*}
If all $Z_{i_{1},...,i_{m-1},\mu}^{\nu}$  proportional to $\delta_{\mu}^{\nu}$ we say that the $Z$s are \emph{diagonal}.
Even for monomial ideals we can choose non-diagonal $Z$ if this seems convenient.

\subsection{Determinantal Brackets}\label{subsec:detbrack}

The aim of this subsection is to prove Theorem \ref{thm:detbracket} providing determinantal Nambu Poisson brackets.
We say that a smooth  manifold   $P$ is \emph{Nambu-Poisson} if $\mathcal C^\infty(P)$ is a Nambu-Poisson algebra. Similarly, a complex manifold $P$ is called \emph{Nambu-Poisson} if the sheaf of holomorphic functions $\mathcal O(P)$ is a sheaf of Nambu-Poisson algebras.

By an \emph{$m$-vectorfield} on a smooth manifold $P$ we mean a section of the $m$th exterior power $\bigwedge^m TP$ of the tangent bundle. The $\mathcal C^\infty(P)$-module of smooth sections $\Gamma(P,\bigwedge^m TP)$   is denoted by $\mathfrak X^m(P)$ and we write $\mathfrak X(P):=\bigoplus_{m\ge 0}\mathfrak X^m(P)$. Elements of $\mathfrak X^m(P)$ are called $m$-vector fields, while, more generally,  those of $\mathfrak X(P)$ are called multivector fields.
Similarly, for a complex manifold $P$ an \emph{$m$-vectorfield} is a section of the sheaf $\mathfrak X^m(P):=\bigwedge^m \mathcal T P$.

For the sake of clarity let us put the proof of the following.
\begin{lemma}[\cite{hagiwara2002nambu}]
Let $P$ be a smooth manifold.
Let $\Pi\in\mathfrak{X}^m(P)$ and define 
$\{f_{1},...,f_{m}\}:=\Pi(df_{1},...,df_{m})$.
The bracket $\{\:,\dots,\:\}$ satisfies the fundamental identity if and only if 
\begin{equation}\label{eqnp1}
\mathcal{L}_{X_{f_{1}...f_{m-1}}}\Pi=0
\end{equation}
for all $f_{1},...,f_{m-1}\in \calC^{\infty}(P)$.
Here $X_{f_{1}...f_{m-1}}$ is understood to be the vector field $\{f_{1},...,f_{m-1},\:\}$ and $\mathcal{L}_{X_{f_{1}...f_{m-1}}}$ the Lie derivative along $X_{f_{1}...f_{m-1}}$.
\end{lemma}
\begin{proof} 
Indeed, $\mathcal{L}_{X_{f_{1}...f_{m-1}}}\Pi=0$ implies $(\mathcal{L}_{X_{f_{1}...f_{m-1}}}\Pi)(dg_{i_{1}},...,dg_{i_{m}})=0$ for all $g_{i_{1}},...,g_{i_{m}}\in C^{\infty}(P)$. 
From the general properties of the Lie derivative we obtain:
\begin{multline}\label{eq1}
\mathcal{L}_{X_{f_{1}...f_{m-1}}}(\Pi(dg_{1},...,dg_{m}))=(\mathcal{L}_{X_{f_{1}...f_{m-1}}}\Pi)(dg_{1},...,dg_{m})+\Pi(\mathcal{L}_{X_{f_{1}...f_{m-1}}}dg_{1},...,dg_{m})+\cdots\\
+\Pi(dg_{1},...,\mathcal{L}_{X_{f_{1}...f_{m-1}}}dg_{m}).
\end{multline}
Now applying Cartan's formula and the fact that $d^{2}=0$, we get
\begin{align*}
\mathcal{L}_{X_{f_{1}...f_{m-1}}}dg_{j}&=i_{X_{f_{1}...f_{m-1}}}d^{2}g_{j}+d(i_{X_{f_{1}...f_{m-1}}}dg_{j})                         =d(i_{X_{f_{1}...f_{m-1}}}dg_{j})
=d(\{f_{1},...,f_{m-1},g_{j}\})
\end{align*}
for $j=1,...,m$. Replacing this equation in equation \eqref{eq1} we obtain
\begin{multline*}
\mathcal{L}_{X_{f_{1}...f_{m-1}}}(\Pi(dg_{1},...,dg_{m}))=(\mathcal{L}_{X_{f_{1}...f_{m-1}}}\Pi)(dg_{1},...,dg_{m})+\{\{f_1,\dots,f_{m-1},g_1\},\dots,g_m\}+\cdots\\
+\{g_{1},\dots,\{f_{1},...,f_{m-1},g_{m}\}\}.
\end{multline*}
The left hand side simplifies to
\begin{equation*}
X_{f_{1}...f_{m-1}}(\{g_{1},...,g_{m}\})=\{f_{1}...f_{m-1},\{g_{1},...,g_{m}\}\}.
\end{equation*}
This means that $\mathcal{L}_{X_{f_{1}...f_{m-1}}}\Pi=0$ is equivalent to the fundamental identity.
\end{proof}
We get a similar statement for complex manifolds. 
If $A$ is the coordinate algebra of a smooth affine Nambu-Poisson variety then $\{f_{1},...,f_{m}\}=\Pi(df_{1},...,df_{m})$ can be used to define $\Pi\in \bigwedge_A\Der(A)$. As Lie derivative can be defined via Cartan's magic formula, equation \eqref{eqnp1} holds in this situation as well.

\begin{lemma}
Let $h_1,\dots,h_n$ be a two times continuously differentiable functions in the variables $t_1,\dots,t_n$. Then $\sum _{i}\frac{\partial }{\partial t_{i}}\frac{\partial ( h_{1} ,\dotsc ,h_{n-1} ,t_{i})}{\partial ( t_{1} ,\dotsc ,t_{n})}=0$.
\end{lemma}
\begin{proof}
\begin{align*}
 \sum _{i}\frac{\partial }{\partial t_{i}}\frac{\partial ( h_{1} ,\dotsc ,h_{n-1} ,t_{i})}{\partial ( t_{1} ,\dotsc ,t_{n})} &=\sum _{i}\frac{\partial }{\partial t_{i}}\sum _{\sigma }( -1)^{\sigma }\frac{\partial h_{1}}{\partial t_{\sigma ( 1)}} \cdots \frac{\partial h_{n-1}}{\partial t_{\sigma ( n-1)}} \delta _{\sigma ( n)}^{i}\\
 &=\sum _{i}\frac{\partial }{\partial t_{i}} \ \sum _{\sigma ,\ \sigma ( n) =i}( -1)^{\sigma }\frac{\partial h_{1}}{\partial t_{\sigma ( 1)}} \cdots \frac{\partial h_{n-1}}{\partial t_{\sigma ( n-1)}}\\
 &=\sum _{\sigma }( -1)^{\sigma }\frac{\partial }{\partial t_{\sigma ( n)}}\left(\frac{\partial h_{1}}{\partial t_{\sigma ( 1)}} \cdots \frac{\partial h_{n-1}}{\partial t_{\sigma ( n-1)}}\right) \\
 &=\sum _{j=1}^{n-1}\sum _{\sigma }( -1)^{\sigma }\frac{\partial h_{1}}{\partial t_{\sigma ( 1)}} \cdots \widehat{\frac{\partial h_{j}}{\partial t_{\sigma ( j)}}} \cdots \frac{\partial h_{n-1}}{\partial t_{\sigma ( n-1)}}\frac{\partial ^{2} h_{j}}{\partial t_{\sigma ( n)} \partial t_{\sigma ( j)}} =0.
\end{align*}
\end{proof}

\begin{theorem}\label{thm:detbracket}
Let $\kk=\R$ or $\C$ and $S=\kk[x_{1},...,x_{n}]$. Assume that $X^{1},X^{2},...,X^{k+m}\in \Der(S)$ be pairwise commuting derivations, $g\in S$ and $f_{1},...,f_{k}\in S$. For
$a_{1}=f_{k+1},...,a_{m}=f_{k+m}\in S$,
let
\begin{equation*}
\{a_{1},...,a_{m}\}:=g\Det(X^{\nu}(f_{\mu}))_{\mu,\nu=1,2,...,k+m}.
\end{equation*}
Then $\{\:,...,\:\}$ defines a Nambu-Poisson bracket on $S$.
\end{theorem}
\begin{proof}
If $\kk=\R$ let $U\subseteq \mathbb{R}^{n}$ be the open subset where $X^{1},...,X^{k+m}$ are linearly independent. On $\mathbb{R}^{n}\backslash U$ the bracket is identically zero and there is nothing to show. Let $x\in U$ and $\phi$ be a diffeomorphism in a neighbourhood $V$ of $0\in\{(t_{1},...,t_{n})\in\mathbb{R}^{n}\}$ to a neighbourhood $U_{x}$ of $x$ in $U$ such that $T\phi$ sends $\dfrac{\partial}{\partial t_{\nu}}$ to $X^{\nu}$ for $\nu=1,...,k+m$. Consider $F_{\mu}:=f_{\mu}\circ\phi$ for $\mu=1,...,k$ and $G:=g\circ\phi$.\\
For $A_{1}=F_{k+1},...,A_{l}=F_{k+m}\in C^{\infty}(V)$,
\begin{equation*}
\{A_{1},...,A_{m}\}:= G\Det\biggl(\biggl(\dfrac{\partial F_{\mu}}{\partial t_{\nu}}\biggr)_{\mu,\nu=1,...,k+m}\biggr)
\end{equation*}
It is enough to show that this is a Nambu-Poisson bracket since $\{a_{1},\dots,a_{m}\}\circ\phi=\{A_1,\dots,A_{m}\}$. A similar argument can be presented for the case $\kk=\C$.
Let $\Pi\in\mathfrak X(V)$ such that $\Pi(dA_{1},...,dA_{m})=\{A_{1},...,A_{m}\}$ and let $X=X_{g_{1},...,g_{m-1}}\in\mathfrak{X}(V)$ be such that $X_{g_{1}...g_{m-1}}(g)=\{g_{1},...,g_{m-1},g\}$ for all $g\in C^{\infty}(V)$. We want to see that $\mathcal{L}_{X_{g_{1}...g_{m-1}}}\Pi=0$. Note that
\begin{align*}
\{A_{1},...,A_{m}\}&= G\Det\biggl(\biggl(\dfrac{\partial F_{\mu}}{\partial t_{\nu}}\biggr)_{\mu,\nu=1,...,k+m}\biggr)\\
                   &= G dF_{1}\wedge\cdots\wedge dF_{k}\wedge dA_{1}\wedge\cdots\wedge dA_{m}\biggr(\dfrac{\partial}{\partial t_{1}}\wedge\cdots\wedge\dfrac{\partial}{\partial t_{k+m}}\biggl)
\end{align*}
Note that the $F_{\mu}$ are Casimir by the construction of $\Pi$. We need to prove
\begin{align*}
&(\mathcal{L}_{X} G) dF_{1} \land \dotsc \land dF_{k} \land dA_{1} \land \dotsc \land dA_{l}\left(\frac{\partial }{\partial t_{1}} \land \dotsc \land \frac{\partial }{\partial t_{k+l}}\right)\\
&+GdF_{1} \land \dotsc \land dF_{k} \land dA_{1} \land \dotsc \land dA_{l}\mathcal{L}_{X}\left(\frac{\partial }{\partial t_{1}} \land \dotsc \land \frac{\partial }{\partial t_{k+l}}\right).
\end{align*}
We have that $\mathcal{L}_{X}\frac{\partial }{\partial t_{j}} =-\sum _{i}\frac{\partial (\{g_{1} ,\dotsc ,g_{m-1} ,t_{i}\})}{\partial t_{j}}\frac{\partial }{\partial t_{i}}$ and it follows that
\begin{align*}
    \mathcal{L}_{X}\left(\frac{\partial }{\partial t_{1}} \land \dotsc \land \frac{\partial }{\partial t_{k+l}}\right) =-\left(\sum _{i}\frac{\partial (\{g_{1} ,\dotsc ,g_{m-1} ,t_{i}\})}{\partial t_{i}}\right)\frac{\partial }{\partial t_{1}} \land \dotsc \land \frac{\partial }{\partial t_{k+l}}.
\end{align*}
Using the shorthand $X_{i} :=\{g_{1} ,\dotsc ,g_{m-1} ,t_{i}\}$ we obtain
\begin{align*}
 \sum _{i}\frac{\partial (\{g_{1} ,\dotsc ,g_{m-1} ,t_{i}\})}{\partial t_{i}} =\sum _{i} X_{i}\frac{\partial G}{\partial t_{i}} +G\sum _{i}\frac{\partial }{\partial t_{i}}\frac{\partial ( F_{1} ,\dotsc ,F_{k} ,g_{1} ,\dotsc ,g_{m-1} ,t_{i})}{\partial ( t_{1} ,\dotsc ,t_{k+l})},
\end{align*}
which is $\mathcal{L}_{X} G$ by the Lemma.
\end{proof}
The special case when $g=1$ and $X^i=\partial/\partial x_i$ for $i=1,\dots,n$ is referred to as the \emph{Nambu bracket} (Nambu considered the case $n=3$ in \cite{Nambu}). We are pretty sure that
there is an algebraic proof of the statement, avoiding the coordinate change, and expect the theorem to hold arbitrary fields $\kk$ of characteristic zero.

\begin{corollary}\label{cor:detNP}
With the notation of Theorem \ref{thm:detbracket} consider the ideal $I=( f_{1},...,f_{k})$ in $S$. As the $f_{1},...,f_{k}$ are Casimir we have that  $A=S/I$ is a Nambu-Poisson algebra.
\end{corollary}

\subsection{Invariant Rings}\label{subsec:invrings}
Let $G$ be a complex reductive Lie group and let $V$ be a finite-dimensional representation of $G$. Let $\C[V]$ be the algebra of regular functions on $V$ and $\C[V]^{G}$ be the subalgebra of $G$-invariants. Note that $\C[V]^{G}$ can be seen as the algebra of regular functions on the categorical quotient $V\git G$.
A Nambu-Poisson bracket $\{\ ,...,\ \}$ in $\C[V]$ is called $G$-invariant if  for each $g\in G$
\begin{align*}
    g^{*}\{f_{1},\dots,f_{m}\}=\{g^{*}f_{1},\dots,g^{*}f_{m}\},
\end{align*}
where $f_{1},\dots,f_{m}\in\C[V]$ and $(g^{*}f_{i})(v)=f_{i}(gv)$, $v\in V$. In this case $\C[V]^{G}$ forms a Nambu-Poisson subalgebra of $\C[V]$. By a theorem of D. Hilbert and H. Weyl there exists a complete system of homogeneous polynomial invariants $\varphi_{1},\dots,\varphi_{n}\in\C[V]^{G}$. Accordingly, the substitution $x_{i}\mapsto \varphi_{i}$ is a surjective homomorphism $S=\C[x_{1},\dots,x_{n}]\to \C[V]^{G}$. Let $I$ be the kernel of this substitution homomorphism so that $A=S/I$ is isomorphic to the invariant ring $\C[V]^{G}$ as a $\C$-algebra. Via this isomorphism $A$ becomes a Nambu-Poisson algebra.

A natural class of examples to consider are subgroups $G$ of $\SL(n,\C)$. In this case $n=m$ and the Nambu-Poisson bracket is given by the Nambu-Poisson bracket (cf. \ref{subsec:detbrack}), which is clearly $G$-invariant. In fact, for all $g\in \SL(n,\mathbb{C})$
\begin{align*}
    \{g^{*}f_{1},\dots,g^{*}f_{n}\}(v)=&\Det\left({\partial (g^*f_i)\over \partial x_j}(v)\right)_{ij}
     =\Det\left(\dfrac{\partial f_{i}}{\partial x_{j}}(gv)\dfrac{\partial (gv)_{i}}{\partial x_{j}}\right)_{ij}
     =\Det\left(\dfrac{\partial f_{i}}{\partial x_{j}}(gv)\right)_{ij}\Det\left(g\right)\\
     =& (g^{*}\{f_{1},\dots,f_{n}\})(v).
\end{align*}

\subsubsection{Complex torus of dimension $d$}\label{subsubsec:torus} 

With a given $n\times d$-matrix $\mathcal{A}$ we associate a group of diagonal $n\times n$-matrices:
\begin{equation*}
    \Gamma_{\mathcal{A}}:=\left\{\diag\left(\prod_{i=1}^{d}t_{i}^{a_{1i}},\prod_{i=1}^{d}t_{i}^{a_{2i}},\dots,\prod_{i=1}^{d}t_{i}^{a_{ni}}\right)\biggr| \ \ t_{1},...,t_{d}\in\mathbb{C}^{*}\right\}
\end{equation*}
The matrix group $\Gamma_{\mathcal{A}}$ is isomorphic to the group $(\mathbb{C^\times})^{d}$ of invertible diagonal $d\times d$-matrices, which is called the $d$-dimensional complex torus. We call $\Gamma_{\mathcal{A}}$ the torus defined by $\mathcal{A}$. In (\cite{sturmfels2008algorithms}), Bernd Sturmfels described an algorithm for computing the generators of its invariant ring $\mathbb{C}[x_{1},...,x_{n}]^{\Gamma_{\mathcal{A}}}$. 
If row sum of each row of  $\mathcal{A}$ is zero $\Gamma_{\mathcal{A}}$ is a subgroup of $\SL(n,\mathbb{C})$.

\begin{example}
Consider, for instance, the vector $\mathcal{A}=(1,1,-1,-1)$. A fundamental system of invariants is
\begin{equation*}
    u_{1}=x_{1}x_{3},\ \ \ u_{2}=x_{1}x_{4},\ \ \ u_{3}=x_{2}x_{3},\ \ \ u_{4}=x_{2}x_{4}.
\end{equation*}
With the help of \emph{Macaulay2}, we find this polynomials satisfy the relation
\begin{equation*}
    u_{1}u_{3}-u_{2}u_{4}=0.
\end{equation*}
As described above, the invariant ring $\mathbb{C}(V)^{\Gamma_{\mathcal{A}}}$ becomes a Nambu-Poisson algebra with induced bracket completely determined by the equation $\{u_{1},u_{2},u_{3},u_{4}\}=0$
\end{example}
This can be explained as follows.
\subsubsection{Invariant rings with vanishing brackets}\label{subsubsec:vanishingbrackets}
The \emph{categorical quotient} $V\git G$ is understood to be the spectrum of the algebra $\C[V]^G$. Let $H$ be the principal isotropy group of the representation $V$ (for details see, e.g., \cite{HHScompositio}).
\begin{proposition}\label{prop:Nambu}
If $\dim(G)>\dim(H)$ then bracket on $\C[V]^G$ induced from the Nambu-Poisson bracket vanishes. If the induced bracket on $\C[V]^G$ is nonvanishing then there is a finite group $\Gamma$ acting linearly on $V$ such that  $\C[V]^G=\C[V]^\Gamma$.
\end{proposition}
\begin{proof}
Let $u_{1} ,\dotsc ,u_{k}$ be the fundamental $ G$-invariants and put $ \boldsymbol{u} =( u_{1} ,\dotsc ,u_{k}) :\mathbb{C}^{n}\rightarrow \mathbb{C}^{k}$.The Nambu-Poisson bracket
\begin{align*}
    \{u_{j_{1}} ,\dotsc ,u_{j_{n}}\} =\frac{\partial ( u_{j_{1}} ,\dotsc ,u_{j_{n}})}{\partial ( x_{1} ,\dotsc ,x_{n})}
\end{align*}
can be understood as an $n\times n$-minor of the Jacobi matrix $T_{\boldsymbol{x}}\boldsymbol{u} =\left(\frac{\partial u_{i}}{\partial x_{j}}\right)_{i,j}$. Recall that for the dimension of the categorical quotient $V\git G$ we have the formula
\begin{align*}
\dim( V) -\dim( G) +\dim( H) =\dim( V\git G)\overset{( *)}{=}\mathrm{rank}(T_{\boldsymbol{x}}\boldsymbol{u}),
\end{align*}
where the equality (*) follows from the fact that the invariant ring is Cohen-Macaulay \cite{HochsterRoberts}.
\end{proof}

By a \emph{linear bracket in the sense of Vaisman} (see \cite[Equation (2.26)]{vaisman1999survey}) on $ V=\kk^{n}$, with $\kk=\R$ or $\C$, we mean a Nambu-Poisson bracket of the following form (a more general definition has been suggested in \cite{takhtajan1994foundation}). We take the standard Riemannian metric and consider the Hodge star operator $*:\bigwedge{}^{n-1} V^{*}\rightarrow V^{*}$. Interpreting the canonical linear coordinates $ x_{1} ,\dotsc ,x_{n}$ on $ V$ as covectors the bracket is defined as $\{x_{i_{1}} ,\dotsc ,x_{i_{n-1}}\} =*( x_{i_{1}} \land \dotsc \land x_{i_{n-1}})$. The bracket is obviously $\OO(n,\kk)$-invariant.

\begin{proposition}\label{prop:nambu} Let $\kk=\R$ (respectively =$\C$) and $G$ be a closed Lie subgroup (respectively reductive subgroup) of $\OO(n,\kk)$ with principal isotropy group $H$. If $\dim(G)>\dim(H)$ the bracket on $\kk[x_1,\cdots,x_n]^G$ induced from the linear Nambu-Poisson bracket in the sense of Vaisman vanishes.
In particular, if such a Nambu-Poisson bracket is non-vanishing we can find a finite group $\Gamma$ acting linearly on $V$ such that $\kk[x_1,\cdots,x_n]^G=\kk[x_1,\cdots,x_n]^\Gamma$.
\end{proposition}
\begin{proof} Firstly, we observe that the bracket can be understood as
\begin{align*}
 \{x_{1} ,\dotsc ,x_{n-1}\} &=x_{n},\\
 \{x_{n} ,x_{2} ,\dotsc ,x_{n-2}\} &=x_{n-1},\\
 \{x_{n-1} ,x_{n} ,x_{2} ,\dotsc ,x_{n-3}\} &=x_{n-2},\\
 \dots&
\end{align*}
Let $u_1 ,\dotsc ,u_k$'s be the fundamental $\OO(n,\kk)$-invariants. Their brackets can be written accordingly as
\begin{align*}
\{u_{j_{1}} ,\dotsc ,u_{j_{n-1}}\} =\begin{vmatrix}
x_{1} & x_{2} & x_{3} & \cdots  & x_{n}\\
\frac{\partial u_{j_{1}}}{\partial x_{1}} & \frac{\partial u_{j_{1}}}{\partial x_{2}} & \frac{\partial u_{j_{1}}}{\partial x_{3}} &  & \frac{\partial u_{j_{1}}}{\partial x_{n}}\\
\frac{\partial u_{j_{2}}}{\partial x_{1}} & \frac{\partial u_{j_{2}}}{\partial x_{2}} & \frac{\partial u_{j_{2}}}{\partial x_{3}} &  & \frac{\partial u_{j_{2}}}{\partial x_{n}}\\
\cdots  & \cdots  & \cdots  &  & \cdots \\
\frac{\partial u_{j_{n-1}}}{\partial x_{1}} & \frac{\partial u_{j_{n-1}}}{\partial x_{2}} & \frac{\partial u_{j_{n-1}}}{\partial x_{3}} & \cdots  & \frac{\partial u_{j_{n-1}}}{\partial x_{n}}
\end{vmatrix}.
\end{align*}
But $ \frac{1}{2}\sum _{i=1}^{n} x_{i}^{2} =:f( u_{1} ,\dotsc ,u_{k})$  because it is an $\OO(n,\kk)$-invariant. So
\begin{align*}
\{u_{j_{1}} ,\dotsc ,u_{j_{n-1}}\} =\begin{vmatrix}
\frac{\partial f}{\partial x_{1}} & \frac{\partial f}{\partial x_{2}} & \frac{\partial f}{\partial x_{3}} &  & \frac{\partial f}{\partial x_{n}}\\
\frac{\partial u_{j_{1}}}{\partial x_{1}} & \frac{\partial u_{j_{1}}}{\partial x_{2}} & \frac{\partial u_{j_{1}}}{\partial x_{3}} &  & \frac{\partial u_{j_{1}}}{\partial x_{n}}\\
\frac{\partial u_{j_{2}}}{\partial x_{1}} & \frac{\partial u_{j_{2}}}{\partial x_{2}} & \frac{\partial u_{j_{2}}}{\partial x_{3}} &  & \frac{\partial u_{j_{2}}}{\partial x_{n}}\\
\cdots  & \cdots  & \cdots  &  & \cdots \\
\frac{\partial u_{j_{n-1}}}{\partial x_{1}} & \frac{\partial u_{j_{n-1}}}{\partial x_{2}} & \frac{\partial u_{j_{n-1}}}{\partial x_{3}} & \cdots  & \frac{\partial u_{j_{n-1}}}{\partial x_{n}}
\end{vmatrix} =\sum _{l=1}^{k}\frac{\partial f}{\partial u_{l}}\frac{\partial ( u_{l} ,u_{j_{1}} ,\dotsc ,u_{j_{n-1}})}{\partial ( x_{1} ,\dotsc ,x_{n})}\overset{\text{Prop. \ref{prop:Nambu}}}{=} 0.
\end{align*}
Hence if the bracket $\{\ ,\dots, \ \}$ is nonvanishig we must have $\dim(G)=\dim(H)$.
But if $H\subseteq G$  is a subgroup of $\dim H = \dim G$, then $H$ consists of connected components of $G$. This means the quotient is trivial unless $G$ is not connected. Let us assume $G$ is not connected. Then $\dim G = \dim H$ means that $H$ contains the connected component $G^0$ of the identity, which is normal. If $G^0$ acts trivially on a point with principal isotropy, then it acts trivially on every point with trivial isotropy. So by continuity, it acts trivially on the entire representation. More generally, if the principal isotropy group is non-trivial the action factors through $G':=G/G^0$. But $\dim G' = \dim (H/G^0)$ and the principal isotropy group $H':=H/G^0$ of $G'$ acts trivially on every point with trivial isotropy.
Put $\Gamma=G'/H'$.
\end{proof}

For each finite subgroup $G$ of $\SL(2,\C)$ there are three invariants satisfying one relation $\varphi_G\in S=\C[x_1,x_2,x_3]$. The corresponding hypersurface rings $A=S/I$ are known as \emph{Kleinian singularities} \cite{Klein,PoissonBook}.
\begin{theorem}[see, e.g., \cite{PoissonBook} or \cite{higherKoszul}]
For every finite subgroup $G$ of $\SL(2,\C)$ the polynomial $\varphi_G$ is a Casimir.
\end{theorem}

\begin{example}\label{ex:abelian}
Let $G$ be a diagonal subgroup of $\SL(3, \mathbb{C})$ of order $24$ generated by
\begin{equation*}
    \begin{pmatrix}
    \varepsilon_{4}^2 && 0 && 0\\
    0 && \varepsilon_{4} && 0\\
    0 && 0 && \varepsilon_{4}
    \end{pmatrix},
    \begin{pmatrix}
    \varepsilon_{6} && 0 && 0\\
    0 && \varepsilon_{6}^3 && 0\\
    0 && 0 && \varepsilon_{6}^2
    \end{pmatrix},
\end{equation*}
where $\varepsilon_{4}$ and $\varepsilon_{6}$ are fourth and sixth primitive roots of unity, respectively. Using Macaulay2 \cite{M2}, we found a fundamental system of polynomial invariants: 
\begin{align*}
    u_{1}= x^6, \ u_{2}= y^4, \ u_{3}= z^{12}, \ u_{4}= xyz, \  u_{5}= y^2z^6, \ u_{6}= x^4z^4,    \ u_{7}= x^2z^8.
\end{align*}
The generators appear in  \cite[p.45--46]{yau1993gorenstein}. These polynomials satisfy the relations
\begin{align*}
    u_4^4-u_2u_6=0, \ u_4^2u_5-u_2u_7=0, \ 
    u_4^2u_6-u_1u_5=0, \ u_5^2-u_2u_3=0, \
    u_5u_6-u_4^2u_7=0, \\
    u_5u_7-u_3u_4^2=0,\ 
    u_6^2-u_1u_7=0, \
    u_6u_7-u_1u_3=0, \ 
    u_7^2-u_3u_6=0.
\end{align*}
The Hilbert series of $\C[\C^3]^G$ is 
\begin{multline*}
{1-t^{12}-2t^{14}-3t^{16}-2t^{18}+t^{20}+4t^{22}+4t^{24}+4t^{26}+t^{28}-2t^{30}-3t^{32}-2t^{34}-t^{36}+t^{48} \over (1-t^{12})(1-t^{10})(1-t^8)^2(1-t^6)(1-t^4)(1-t^3)}\\
={1\over 24} (1 - t)^{-3} + {53\over 288} (1 - t)^{-1} + {53\over 288} +{2017\over 3456} (1 - t) + {1699\over 1728} (1 - t)^2 + O((t - 1)^3).
\end{multline*}
The Nambu-Poisson brackets between the invariants are given by
\begin{alignat*}{3}
\{u_{1},u_{2},u_{3}\}&=288u_{4}^3u_{7},&\qquad
\{u_{1},u_{2},u_{4}\}&=24u_{1}u_{2},&\qquad \{u_{1},u_{2},u_{5}\}&=144u_{2}u_{4}u_{6},\\
\{u_{1},u_{2},u_{6}\}&=96u_{1}u_{4}^3,&\qquad
\{u_{1},u_{2},u_{7}\}&=192u_{4}^3 u_{6},&\qquad
\{u_{1},u_{3},u_{4}\}&=-72u_{1}u_{3},\\
\{u_{1},u_{3},u_{5}\}&=-144u_{4}u_{6}u_{3},&\qquad
\{u_{1},u_{3},u_{6}\}&=0,&\qquad
\{u_{1},u_{3},u_{7}\}&=0,\\
\{u_{1},u_{4},u_{5}\}&=24u_{1}u_{5},&\qquad
\{u_{1},u_{4},u_{6}\}&=54u_{1}u_{6},&\qquad
\{u_{1},u_{4},u_{7}\}&=48u_{6}^2,\\
\{u_{1},u_{4},u_{5}\}&=24u_{1}u_{5},&\qquad
\{u_{1},u_{4},u_{6}\}&=54u_{1}u_{6},&\qquad
\{u_{1},u_{4},u_{7}\}&=48u_{6}^2,\\
\{u_{1},u_{5},u_{6}\}&=48u_{1}u_{4}u_{7},&\qquad
\{u_{1},u_{5},u_{7}\}&=96u_{1}u_{3}u_{4},&\qquad
\{u_{1},u_{6},u_{7}\}&=0,\\
\{u_{2},u_{3},u_{4}\}&=48u_{2}u_{3},&\qquad
\{u_{2},u_{3},u_{5}\}&=0,&\qquad
\{u_{2},u_{3},u_{6}\}&=192u_{4}^3u_{3},\\
\{u_{2},u_{3},u_{7}\}&=96u_{4}u_{5}u_{3},&\qquad
\{u_{2},u_{4},u_{5}\}&=-24u_{2}u_{5},&\qquad
\{u_{2},u_{4},u_{6}\}&=0,\\
\{u_{2},u_{4},u_{7}\}&=-24u_{2}u_{7},&\qquad
\{u_{2},u_{3},u_{4}\}&=48u_{2}u_{3},&\qquad
\{u_{2},u_{3},u_{5}\}&=0,\\
\{u_{2},u_{3},u_{6}\}&=192u_{4}^3u_{3},&\qquad
\{u_{2},u_{3},u_{7}\}&=96u_{4}u_{5}u_{3},&\qquad
\{u_{2},u_{4},u_{5}\}&=0,\\
\{u_{2},u_{4},u_{6}\}&=0,&\qquad
\{u_{2},u_{4},u_{7}\}&=-24u_{4}^2u_{5},&\qquad
\{u_{2},u_{5},u_{6}\}&=96u_{4}^3u_{5},\\
\{u_{2},u_{5},u_{7}\}&=48u_{4}u_{2}u_{3},&\qquad
\{u_{2},u_{6},u_{7}\}&=-96u_{4}^3u_{7},&\qquad
\{u_{3},u_{4},u_{5}\}&=24u_{3}u_{5},\\
\{u_{3},u_{4},u_{6}\}&=-48u_{3}u_{6},&\qquad
\{u_{3},u_{4},u_{7}\}&=-24u_{7}u_{3},&\qquad
\{u_{3},u_{5},u_{6}\}&=-96u_{4}u_{7}u_{3},\\
\{u_{3},u_{5},u_{7}\}&=-48u_{4}u_{3}^2,&\qquad 
\{u_{3},u_{6},u_{7}\}&=0,&\qquad
\{u_{4},u_{5},u_{6}\}&=24u_{5}u_{6},\\
\{u_{4},u_{5},u_{7}\}&=24u_{7}u_{5},&\qquad
\{u_{4},u_{6},u_{7}\}&=-24u_{1}u_{3},&\qquad
\{u_{5},u_{6},u_{7}\}&=-48u_{4}u_{7}^2.
\end{alignat*}
The list of nonzero $Z_{i_{1}\cdots i_{m-1}\mu}^{\nu}\in S$ is to long to put here (they will be made available in AM's forthcoming PhD thesis.)
\end{example}
\begin{example}\label{ex:E} 
Let $\zeta$ be a third primitive root of unity and $\kappa:=({-2\over 3}\zeta-{1\over 3})\zeta$. The finite subgroup $E$ of $\SL(3,\C)$ of order $108$ (see, e.g., \cite{Blichfeldt, yau1993gorenstein}) is generated by the matrices
\begin{align*}
   \begin{pmatrix}
1 & 0 & 0\\
0 & \zeta  & 0\\
0 & 0 & \zeta ^{2}
\end{pmatrix} ,\begin{pmatrix}
0 & 1 & 0\\
0 & 0 & 1\\
1 & 0 & 0
\end{pmatrix} \ ,\kappa \begin{pmatrix}
1 & 1 & 1\\
1 & \zeta  & \zeta ^{2}\\
1 & \zeta ^{2} & \zeta 
\end{pmatrix}
\end{align*}
To calculate the fundamental invariants of $E$ we used Magma \cite{Magma}:
\begin{align*}
u_1&=x_1^6 - \dfrac{15}{4}x_1^4x_2x_3 - \dfrac{5}{2}x_1^3x_2^3 - \dfrac{5}{2}x_1^3x_3^3 - \dfrac{45}{4}x_1^2x_2^2x_3^2
        - \frac{15}{4}x_1x_2^4x_3 - \dfrac{15}{4}x_1x_2x_3^4 + x_2^6 - \dfrac{5}{2}x_2^3x_3^3 + x_3^6,\\
 u_2&=x_1^4x_2x_3 - 2x_1^3x_2^3 - 2x_1^3x_3^3 + 3x_1^2x_2^2x_3^2 + x_1x_2^4x_3 +
        x_1x_2x_3^4 - 2x_2^3x_3^3,\\
u_3&=x_1^{12} + \dfrac{33}{61}x_1^{10}x_2x_3 + \dfrac{55}{61}x_1^9x_2^3 + \dfrac{55}{61}x_1^9x_3^3 + \dfrac{1485}{122}x_1^8x_2^2x_3^2 + \dfrac{990}{61}x_1^7x_2^4x_3 + \dfrac{990}{61}x_1^7x_2x_3^4 \\&
+\dfrac{231}{61}x_1^6x_2^6+ \dfrac{4620}{61}x_1^6x_2^3x_3^3 + \dfrac{231}{61}x_1^6x_3^6 +\dfrac{4158}{61}x_1^5x_2^5x_3^2 + \dfrac{4158}{61}x_1^5x_2^2x_3^5 + \dfrac{990}{61}x_1^4x_2^7x_3\\& 
+\dfrac{17325}{122}x_1^4x_2^4x_3^4 + \dfrac{990}{61}x_1^4x_2x_3^7 + \dfrac{55}{61}x_1^3x_2^9 + \dfrac{4620}{61}x_1^3x_2^6x_3^3 + \dfrac{4620}{61}x_1^3x_2^3x_3^6 + \dfrac{55}{61}x_1^3x_3^9 \\& 
+\dfrac{1485}{122}x_1^2x_2^8x_3^2 + \dfrac{4158}{61}x_1^2x_2^5x_3^5 + \dfrac{1485}{122}x_1^2x_2^2x_3^8 + \dfrac{33}{61}x_1x_2^{10}x_3 
\\& + \dfrac{990}{61}x_1x_2^7x_3^4 + \dfrac{990}{61}x_1x_2^4x_3^7 + \dfrac{33}{61}x_1x_2x_3^{10} + x_2^{12} + \dfrac{55}{61}x_2^9x_3^3 +
        \dfrac{231}{61}x_2^6x_3^6 + \dfrac{55}{61}x_2^3x_3^9 + x_3^{12}
\\
        \end{align*}
        \begin{align*}
u_4&= x_1^6x_2^3 - x_1^6x_3^3 - x_1^3x_2^6 + x_1^3x_3^6 + x_2^6x_3^3 - x_2^3x_3^6,\\
  u_5&=  x_1^{12} + 15x_1^{10}x_2x_3 - \dfrac{19}{2}x_1^9x_2^3 - \dfrac{19}{2}x_1^9x_3^3 - 27x_1^8x_2^2x_3^2 + 126x_1^7x_2^4x_3 + 126x_1^7x_2x_3^4 - 21x_1^6x_2^6\\
  &   - 42x_1^6x_2^3x_3^3
        - 21x_1^6x_3^6 + 189x_1^5x_2^5x_3^2 + 189x_1^5x_2^2x_3^5 +
        126x_1^4x_2^7x_3 - 315x_1^4x_2^4x_3^4 + 126x_1^4x_2x_3^7 - \dfrac{19}{2}x_1^3x_2^9 \\
  &  - 42x_1^3x_2^6x_3^3 - 42x_1^3x_2^3x_3^6 - \dfrac{19}{2}x_1^3x_3^9- 27x_1^2x_2^8x_3^2 + 189x_1^2x_2^5x_3^5 - 27x_1^2x_2^2x_3^8 +
        15x_1 x_2^{10}x_3 + 126x_1x_2^7x_3^4 
        \\&+ 126x_1x_2^4x_3^7 + 15x_1x_2x_3^{10} +
        x_2^{12} - \dfrac{19}{2}x_2^9x_3^3 - 21x_2^6x_3^6 - \dfrac{19}{2}x_2^3x_3^9 + x_3^{12}
\end{align*}
Here $u_1,u_2,u_3$ are the primary invariants and $u_4,u_5$ the secondary invariants. We emphasize that using the formulas in the well-known book \cite{yau1993gorenstein}
we ran into inconsistencies for the Nambu-Poisson brackets. It turned out that the suggested polynomials are not $E$-invariant.
The invariants computed by Magma look less logical than those in \cite{yau1993gorenstein} and we propose to try to identify and correct the mistakes in \cite{yau1993gorenstein}. 

There are $2$ relations among the $u_1,u_2,u_3,u_4,u_5$.
\begin{align*}
    -&\dfrac{251}{42795}u_{1}^3 + \dfrac{1813}{19020}u_{1}^2u_{2} + \dfrac{233}{5072}u_{1}u_{2}^2 - \dfrac{215}{20288}u_{2}^3 +
        \dfrac{61}{4755}u_{1}u_{3}- \dfrac{427}{11412}u_{2}u_{3}
        - u_{4}^2 \\
        &- \dfrac{298}{42795}u_{1}u_{5} - \dfrac{25}{5706}u_{2}u_{5}=0,\\
    &\dfrac{624939}{38884}u_{1}^4 - \dfrac{4063797}{19442}u_{1}^3u_{2} + \dfrac{173361303}{311072}u_{1}^2u_{2}^2 -\dfrac{24352245}{155536}u_{1}u_{2}^3
        - \dfrac{3859344225}{9954304}u_{2}^4 -
        \dfrac{3652497}{38884}u_{1}^2u_{3}\\
        &+ \dfrac{39974337}{77768}u_{1}u_{2}u_{3} -
        \dfrac{564089265}{622144}u_{2}^2u_{3}+ \dfrac{1101416}{9721}u_{3}^2 + \dfrac{604557}{19442}u_{1}^2u_{5} -
        \frac{968247}{38884}u_{1}u_{2}u_{5}\\
        &- \dfrac{11237535}{311072}u_{2}^2u_{5} - \dfrac{637084}{9721}u_{3}u_{5} -
        u_{5}^2 =0.
\end{align*}
The Hilbert series of $\C[\C^3]^E$ 
\begin{multline*}
{-t^{18} + t^{15} - t^{12} - t^6 + t^3 - 1\over t^{21} - t^{18} - t^{15} + t^{12} - t^9 + t^6 + t^3 - 1}\\
={1\over 108} (1 - t)^{-3} + {77\over 432} (1 - t)^{-1} + {77\over 432} + {3305 \over 5184}(1 - t) + {2843\over 2592} (t - 1)^2+ O((1 - t)^3).
\end{multline*}
The Nambu-Poisson brackets between the invariants are given by
\begin{align*}
\{u_{1}, u_{2}, u_{3}\}&=\dfrac{32646078}{96685}u_{1}^2u_{4} -\dfrac{60682689}{154696}u_{2}^2u_{4} -\dfrac{26142669}{96685}u_{1}u_{2}u_{4} -\dfrac{1127952}{1585}u_{4}u_{3} + -\dfrac{2099736}{96685}u_{4}u_{5},\\
\{u_{1}, u_{2}, u_{4}\}&=-\dfrac{5318784}{397835}u_{2}^3-\dfrac{5318784}{397835}u_{1}^2u_{2}-\dfrac{12078558}{397835}u_{2}^2u_{1}
+\dfrac{244}{251}u_{1}u_{3}+\dfrac{10957491}{397835}u_{3}u_{2}+\dfrac{59292}{251}u_{4}^2-\dfrac{244}{251}u_{5}u_{1}\\
&-\dfrac{46827}{397835}u_{5}u_{2},\\
\{u_{1}, u_{2}, u_{5}\}&=\dfrac{1616679}{1585}u_{1}^2u_{4}+\dfrac{49935771}{5072}u_{2}^2u_{4}-\dfrac{17693559}{3170}u_{1}u_{2}u_{4}  -\dfrac{3900096}{1585}u_{4}u_{3}+\dfrac{1127952}{1585}u_{4}u_{5},\\
\{u_{1}, u_{3}, u_{4}\}&=\dfrac{5704653954544839}{34647166589696}u_{2}^4 + \dfrac{583393415540661}{222009643960}u_{4}^2u_{1}
+\dfrac{3370844600762229}{888038575840}u_{4}^2u_{2}-\dfrac{46313098056846}{721076405075}u_{3}^2\\
&+\dfrac{35345775859857}{87971321419150}u_{5}^2- \dfrac{38462950550026107}{175942642838300}u_{1}^3u_{2}+ \dfrac{40098727953459423}{112603291416512}u_{2}^3u_{1} +\dfrac{996228877085973}{54136197796400}u_{1}^2u_{2}^2 \\
&+\dfrac{42254978197383}{1442152810150}u_{1}^2u_{3} + \dfrac{309721889613717}{43985660709575} u_{1}^2u_{5}+\dfrac{1337803306752501}{4614888992480}u_{2}^2u_{3} + \dfrac{16407160638831}{140754114270640}u_{2}^2u_{5}\\
&+\dfrac{19818481457439}{721076405075}u_{3}u_{5},
\end{align*}
\begin{align*}
\{u_{1}, u_{3}, u_{5}\}=&  \dfrac{4168091763}{2623}u_{4}^3 -\dfrac{231764431599}{1662982}u_{4}u_{1}^2u_{2}+\dfrac{5862335877}{831491}u_{1}^3u_{4} - \dfrac{1039020476247}{13303856}u_{1}u_{2}^2u_{4} \\
&-\dfrac{403580961}{27262}u_{1}u_{4}u_{3} +\dfrac{8975849679}{831491} u_{1}u_{5}u_{4}+\dfrac{13679685}{344} u_{2}u_{4}u_{3}+\dfrac{64816119}{10492} u_{2}u_{5}u_{4},
\end{align*}
\begin{align*}
\{u_1, u_4, u_5\}=&-\dfrac{78709258596411}{283993168768}u_{2}^4-\dfrac{33344932297473}{3639502360}u_{4}^2u_{1}
+\dfrac{53709428414097}{14558009440}u_{4}^2u_{2} + \dfrac{133245637761888}{721076405075}u_{3}^2 \\&
-\dfrac{2007185541093}{721076405075}u_{5}^2 + \dfrac{778650978726243}{1442152810150}u_{1}^3u_{2} 
+\dfrac{64698475070529}{922977798496}u_{2}^3u_{1} +\dfrac{305810277167223}{443739326200}u_{1}^2u_{2}^2\\&
-\dfrac{70703199320337}{721076405075}u_{1}^2u_{3}+\dfrac{3049960785084}{721076405075}u_{1}^2u_{5}
-\dfrac{1728268774206939}{2307444496240}u_{2}^2u_{3} - \dfrac{26696092221207}{576861124060}u_{2}^2u_{4}\\
&-\dfrac{63585213685542}{721076405075}u_{3}u_{5},
\end{align*}
\begin{align*}
    \{u_{2}, u_{3},u_{4}\}=&-\dfrac{538417725748893}{8661791647424}u_{2}^4 -\dfrac{17693030982951}{55502410990}u_{4}^2u_{1}
    +\dfrac{182398791605319}{222009643960}u_{4}^2u_{2} + \dfrac{8958486260744}{721076405075}u_{3}^2 \\ &
    -\dfrac{3269241260774}{43985660709575} u_{5}^2 - \dfrac{714805005662937}{43985660709575}u_{1}^3u_{2}
    -\dfrac{2466119654894565}{28150822854128}u_{2}^3u_{1} + \dfrac{440504461460457}{13534049449100}u_{1}^2u_{2}^2\\& -\dfrac{6252201734106}{721076405075}u_{1}^2u_{3}
    +\dfrac{167553754203012}{43985660709575}u_{1}^2u_{5} -\dfrac{129033380340711}{1153722248120}u_{2}^2u_{3}
    -\dfrac{54271892078541}{35188528567660}u_{2}^2u_{5}\\
    &-\dfrac{5399473263396}{721076405075}u_{3}u_{5},
    \end{align*}
    \begin{align*}
    \{u_{2}, u_{3}, u_{5}\}=& -\dfrac{2104033968}{13115}u_{4}^3 + \dfrac{459426514842}{20787275}u_{4}u_{1}^2u_{2}
   -\dfrac{34070806764}{20787275} u_{1}^3u_{4} -\dfrac{78712494147}{16629820}u_{1}u_{2}^2u_{4}\\
    &+\dfrac{1396516626}{340775}u_{1}u_{4}u_{3}-\dfrac{37260114372}{20787275}u_{1}u_{5}u_{4} -\dfrac{4981257}{430}u_{2}u_{4}u_{3} -\dfrac{5672349}{13115} u_{2}u_{5}u_{4},
    \end{align*}
    \begin{align*}
    \{u_{2}, u_{4}, u_{5}\}=& -\dfrac{8209314781929}{70998292192} u_{2}^4-\dfrac{499929903477}{909875590}u_{4}^2u_{1}
    +\dfrac{38328621975333}{3639502360}u_{4}^2u_{2} +\dfrac{4908094733888}{721076405075}u_{3}^2\\&
    +\dfrac{1843763564532}{721076405075}u_{5}^2 + \dfrac{29640453273234}{721076405075}u_{1}^3u_{2}
    -\dfrac{1211568848961921}{1153722248120}u_{2}^3u_{1}  -\dfrac{70172225519283}{110934831550}u_{1}^2u_{2}^2\\
    &-\dfrac{443352666012}{721076405075}u_{1}^2u_{3}
   -\dfrac{583426926816}{721076405075}u_{1}^2u_{5}+ \dfrac{65031807400263}{115372224812}u_{2}^2u_{3} + \dfrac{8923712085594}{144215281015}u_{2}^2u_{5}\\
    &-\dfrac{5725078705592}{721076405075}u_{3}u_{5}
\end{align*}
and
\begin{align*}
    \{u_{3}, u_{4}, u_{5}\}=& \dfrac{485779978321461}{2027634505120}u_{1}^4u_{2}
   \dfrac{22900926901007187}{4446255106048}u_{2}^4u_{1} - \frac{309958591431435057}{348753134880640}u_{1}^3u_{2}^2\\
  &-\dfrac{121278451779}{4154988740}u_{1}^3u_{3} +  \dfrac{441349590708}{63363578285}u_{1}^3u_{5}
  \dfrac{10414025397434504997}{1395012539522560} u_{2}^3u_{1}^2\\
  &-\dfrac{13268693985093}{9111178496}u_{2}^3u_{3}
  -\dfrac{23455143533331}{69472736032}u_{2}^3u_{5} -\dfrac{256921386165}{91376048}u_{1}^2u_{4}^2 \\
  &-\dfrac{137654404947315}{1753255168}u_{2}^2u_{4}^2 -\dfrac{12800990749149}{33239909920}u_{1}^2u_{2}u_{3}
  +\dfrac{43859242350099}{1013817252560}u_{1}^2u_{2}u_{5}\\ &+\dfrac{9642883594977}{4555589248}u_{2}^2u_{1}u_{3} -\dfrac{50959799895969}{99246765760}u_{2}^2u_{1}u_{5} +  \dfrac{102809034}{4138435}u_{3}^2u_{1}\\
  & -\dfrac{830488221}{8276870}u_{3}^2u_{2}
  -\dfrac{2644817859}{1009778140}u_{5}^2u_{1}-\dfrac{4283553213}{4039112560}u_{5}^2u_{2}.
\end{align*}

\end{example}

\subsection{Outer tensor product}\label{subsec:outer}

We give here a simple construction how one can assemble a Nambu-Poisson algebra from smaller pieces. In principle, this outer tensor product is not more complicated than the pieces it is made of. However, it provides plenty of examples that are generated by Casimirs, which is quite helpful when one is searching for nontrivial examples of Nambu-Poisson algebras. 

\begin{theorem}
Let $ ( A,\{\ ,\dotsc ,\ \}_{k})$ and $( B,\{\ ,\dotsc ,\ \}_{l})$ be Nambu-Poisson algebras of order $k$ and $l$, respectively. Then
$\{a_{1} \otimes 1,\dotsc ,a_{k} \otimes 1,1\otimes b_{1} ,\dotsc ,1\otimes b_{l}\}_{k+l} :=\{a_{1} ,\dotsc ,a_{k}\}_{k} \otimes \{b_{1} ,\dotsc ,b_{l}\}_{l}$
uniquely defines a Nambu-Poisson bracket on $A\otimes B$. Here it is understood that 
$\{a_{1} \otimes 1,\dotsc ,a_{i} \otimes 1,1\otimes b_{1} ,\dotsc ,1\otimes b_{j}\}_{k+l} =0$ \ if $i\neq k$ and $j\neq l$.
\end{theorem}

\begin{proof} 
Assume $ a_{1} ,\dotsc ,a_{k} \in A$ and $ b_{1} ,\dotsc ,b_{l} \in B$, $ n:=k+l$.
Writing  $a_{i} \otimes b_{i} =( a_{i} \otimes 1)( 1\otimes b_{i})$ and applying Leibniz rule repeatedly we derive
\begin{align*}
    &\{a_{1} \otimes b_{1} ,\dotsc ,a_{n} \otimes b_{n}\}_{n}\\
&=\sum _{\sigma \in \UnSh_{k,l}}( -1)^{\sigma } a_{\sigma ( k+1)} \cdots a_{\sigma ( k+l)}\{a_{\sigma ( 1)} ,\dotsc ,a_{\sigma ( k)}\}_{k} \otimes b_{\sigma ( 1)} \cdots b_{\sigma ( k)}\{b_{\sigma ( k+1)} ,\dotsc ,b_{\sigma ( n)}\}_{l}.
\end{align*}
We note that $ \{a_{1} \otimes b_{1} ,\dotsc ,a_{n} \otimes b_{n}\}_{n}$ is totally antisymmetric.
Now we assume $a_{n} =a'a''$ and $ b_{n} =b'b''$ so that
 $ a_{n} \otimes b_{n} =a'a''\otimes b'b''=( a'\otimes b')( a''\otimes b'')$. To verify the general Leibniz rule we calculate
 \begin{align*}
    & \{a_{1} \otimes b_{1} ,\dotsc ,a_{n-1} \otimes b_{n-1} ,a'a'' \otimes b'b''\}_{n}\\
&=\sum _{\sigma \in \UnSh_{k,l} ,\sigma ( k) =n}( -1)^{\sigma } a_{\sigma ( k+1)} \cdots a_{\sigma ( k+l)}\{a_{\sigma ( 1)} ,\dotsc ,a_{\sigma ( k-1)} ,a'a''\}_{k} \otimes b_{\sigma ( 1)} \cdots b_{\sigma ( k-1)} b'b''\{b_{\sigma ( k+1)} ,\dotsc ,b_{\sigma ( n)}\}_{l}\\
&+\sum _{\sigma \in \UnSh_{k,l} ,\sigma ( n) =n}( -1)^{\sigma } a_{\sigma ( k+1)} \cdots a_{\sigma ( n-1)} a'a''\{a_{\sigma ( 1)} ,\dotsc ,a_{\sigma ( k)}\}_{k} \otimes b_{\sigma ( 1)} \cdots b_{\sigma ( k)}\{b_{\sigma ( k+1)} ,\dotsc ,b_{\sigma ( n-1)} b'b''\}_l\\
&=a'\otimes b'\sum _{\sigma \in \UnSh_{k,l} ,\sigma ( k) =n}( -1)^{\sigma } a_{\sigma ( k+1)} \cdots a_{\sigma ( k+l)}\{a_{\sigma ( 1)} ,\dotsc ,a_{\sigma ( k-1)} ,a''\}_{k} \otimes b_{\sigma ( 1)} \cdots b_{\sigma ( k-1)} b''\{b_{\sigma ( k+1)} ,\dotsc ,b_{\sigma ( n)}\}_{l}\\
&+a''\otimes b''\sum _{\sigma \in \UnSh_{k,l} ,\sigma ( k) =n}( -1)^{\sigma } a_{\sigma ( k+1)} \cdots a_{\sigma ( k+l)}\{a_{\sigma ( 1)} ,\dotsc ,a_{\sigma ( k-1)} ,a'\}_{k} \otimes b_{\sigma ( 1)} \cdots b_{\sigma ( k-1)} b'\{b_{\sigma ( k+1)} ,\dotsc ,b_{\sigma ( n)}\}_{l}\\
&+a''\otimes b''\sum _{\sigma \in \UnSh_{k,l} ,\sigma ( n) =n}( -1)^{\sigma } a_{\sigma ( k+1)} \cdots a_{\sigma ( n-1)} a'\{a_{\sigma ( 1)} ,\dotsc ,a_{\sigma ( k)}\}_{k} \otimes b_{\sigma ( 1)} \cdots b_{\sigma ( k)}\{b_{\sigma ( k+1)} ,\dotsc ,b_{\sigma ( n-1)} b'\}_l\\
&+a'\otimes b'\sum _{\sigma \in \UnSh_{k,l} ,\sigma ( n) =n}( -1)^{\sigma } a_{\sigma ( k+1)} \cdots a_{\sigma ( n-1)} a''\{a_{\sigma ( 1)} ,\dotsc ,a_{\sigma ( k)}\}_{k} \otimes b_{\sigma ( 1)} \cdots b_{\sigma ( k)}\{b_{\sigma ( k+1)} ,\dotsc ,b_{\sigma ( n-1)} b''\}_l\\
&=( a'\otimes b')\{a_{1} \otimes b_{1} ,\dotsc ,a_{n-1} \otimes b_{n-1} ,a'' \otimes b''\}_{n} +( a''\otimes b'')\{a_{1} \otimes b_{1} ,\dotsc ,a_{n-1} \otimes b_{n-1} ,a' \otimes b'\}_{n}.
 \end{align*}
 
It remains to prove the fundamental identity for generators of $A\otimes B$. We consider elements of the form $a_{i}\otimes 1$, $1\otimes b_{j}$, $\tilde{a}_{i}\otimes 1$, $1\otimes \tilde{b}_{m}\in A\otimes B$ with $i\in\{1,\dots,k\}$, $j\in\{1,\dots,l\}$ and $m\in\{1,\dots,l-1\}$.
\begin{align*}
\{a_{1}\otimes 1,&\dots,a_{k}\otimes1,1\otimes b_{1},\dots,1\otimes b_{l-1},\{\tilde{a}_{1}\otimes 1,\dots,\tilde{a}_{k}\otimes1,1\otimes\tilde{b}_{1},\dots,1\otimes \tilde{b}_{l}\}_n\}_n\\
=&\{a_{1}\otimes 1,\dots,a_{k}\otimes1,1\otimes b_{1},\dots,1\otimes b_{l-1},\{\tilde{a}_{1},\dots,\tilde{a}_{k}\}_k\otimes\{\tilde{b}_{1},\dots,\tilde{b}_{l}\}_l\}_n\\
=&\{a_{1}\otimes 1,\dots, a_{k}\otimes 1,1\otimes b_{1},\dots,1\otimes b_{l-1},\{\tilde{a}_{1},\dots,\tilde{a}_{k}\}_k\otimes{1}\}_n \left(1\otimes\{\tilde{b}_{1},\dots,\tilde{b}_{l}\}_l\right)\\
&+\left(\{\tilde{a}_{1},\dots,\tilde{a}_{k}\}_k\otimes 1\right)\{a_{1}\otimes 1,\dots,a_{k}\otimes 1,1\otimes b_{1},\dots, 1\otimes b_{l-1},1\otimes \{\tilde{b}_{1},\dots,\tilde{b}_{l}\}_l\}_n\\
=&\{\tilde{a}_{1},\dots,\tilde{a}_{k}\}_k\{a_{1},\dots,a_{k}\}_k\otimes \{b_{1},\dots, b_{l-1},\{\tilde{b}_{1},\dots,\tilde{b}_{l}\}_l\}_l,
\end{align*}
where the first term of the second equality vanishes because the arity is wrong. Now, using the fundamental identity for $\{\:,\dots,\:\}_{l}$ this is
\begin{align*}
=&\sum_{i=1}^{l}\{\tilde{a}_{1},\dots,\tilde{a}_{k}\}_k\{a_{1},\dots,a_{k}\}_k\otimes \{\tilde{b}_{1},\dots,\tilde{b}_{i-1},\{b_{1},\dots,b_{l-1},\tilde{b}_{i}\}_l,\tilde{b}_{i+1},...,\tilde{b}_{l}\}_l\\
=&\left(\{a_{1},\dots,a_{k}\}_k\otimes 1\right)\sum_{i=1}^{l}\{\tilde{a}_{1}\otimes 1,\dots,\tilde{a}_{k}\otimes 1,1\otimes\tilde{b}_{1},\dots,1\otimes\tilde{b}_{i-1},1\otimes\{b_{1},\dots,b_{l-1},\tilde{b}_{i}\}_l,1\otimes\tilde{b}_{i+1},...,1\otimes\tilde{b}_{l}\}_n\\
=&\sum_{i=1}^{l}\{\tilde{a}_{1}\otimes 1,\dots,\tilde{a}_{k}\otimes 1,1\otimes\tilde{b}_{1},\dots,1\otimes\tilde{b}_{i-1},\{a_{1},\dots,a_{k}\}_k\otimes\{b_{1},\dots,b_{l-1},\tilde{b}_{i}\}_l,1\otimes\tilde{b}_{i+1},...,1\otimes\tilde{b}_{l}\}_n\\
=&\sum_{i=1}^{l}\{\tilde{a}_{1}\otimes 1,\dots,\tilde{a}_{k}\otimes 1,1\otimes\tilde{b}_{1},\dots,1\otimes\tilde{b}_{i-1},\{a_{1}\otimes 1,\dots,a_{k}\otimes 1,1\otimes b_{1},\dots,1\otimes b_{l-1},1\otimes \tilde{b}_{i}\}_n,\\
&\ \ \ \ \ \ \ \ \ \ \ \ \ \ \ \ \ \ \ \ \ \ \ \ \ \ \ \ \ \ \ \ \ \ \ \ \ \ \ \ \ \ \ \ \ \ \ \ \ \ \ \ \ \ \ \ \ \ \ \ \ \ \ \ \ \ \ \ \ \ \ \ \ \ \ \ \ \ \ \ \ \ \ \ \ \ \ \ \ \ \ \ \ \ \ \ \ \ \ \ \ \ \ \ \ \ \ \ 
1\otimes\tilde{b}_{i+1},...,1\otimes\tilde{b}_{l}\}_n\\
&+\sum_{i=1}^{k}\{\tilde{a}_{1}\otimes 1,\dots,\tilde{a}_{i-1}\otimes 1,\{a_{1}\otimes 1,\dots,a_{k}\otimes 1,1\otimes b_{1},\dots,1\otimes b_{l-1},1\otimes\tilde{b}_{i}\}_n,a_{i+1}\otimes 1,\dots,\tilde{a}_{k}\otimes 1,\\
&\ \ \ \ \ \ \ \ \ \ \ \ \ \ \ \ \ \ \ \ \ \ \ \ \ \ \ \ \ \ \ \ \ \ \ \ \ \ \ \ \ \ \ \ \ \ \ \ \ \ \ \ \ \ \ \ \ \ \ \ \ \ \ \ \ \ \ \ \ \ \ \ \ \ \ \ \ \ \ \ \ \ \ \ \ \ \ \ \ \ \ \ \ \ \ \ \ \ \ \ \ \ \ \ \ \ \ \ 1\otimes\tilde{b}_{1},\dots,1\otimes\tilde{b}_{l}\}_n,
\end{align*}
where the fourth and the last equalities hold because the arity of $\{\tilde{a}_{1}\otimes 1,\dots,\tilde{a}_{k}\otimes 1,1\otimes\tilde{b}_{1},\dots,1\otimes\tilde{b}_{i-1},\{a_{1},\dots,a_{k}\}\otimes 1,1\otimes\tilde{b}_{i+1},...,1\otimes\tilde{b}_{l}\}1\otimes\{b_{1},\dots,b_{l-1},\tilde{b}_{i}\}$ and $\{\tilde{a}_{1}\otimes 1,\dots,\tilde{a}_{i-1}\otimes 1,\{a_{1}\otimes 1,\dots,a_{k}\otimes 1,1\otimes b_{1},\dots,1\otimes b_{l-1},1\otimes\tilde{b}_{i}\},a_{i+1}\otimes 1,\dots,\tilde{a}_{k}\otimes 1,1\otimes\tilde{b}_{1},\dots,1\otimes\tilde{b}_{l}\}$ are wrong for all $i$.

In the same way we can prove the identity for $a_{i}\otimes 1$, $1\otimes b_{j}$, $\tilde{a}_{j}\otimes 1$, $1\otimes \tilde{b}_{m}\in A\otimes B$ with $i\in\{1,\dots,k-1\}$, $j\in\{1,\dots,k\}$ and $m\in\{1,\dots,l\}$.
\end{proof}
We refer to $(A\otimes B,\{\ ,\dotsc ,\ \}_{k+l})$ as the \emph{outer tensor } product of  $ ( A,\{\ ,\dotsc ,\ \}_{k})$ and $( B,\{\ ,\dotsc ,\ \}_{l})$.
\begin{corollary}
If the affine Nambu-Poisson algebras $(A,\{\ ,\dotsc ,\ \}_{k})$ and $(B,\{\ ,\dotsc ,\ \}_{l})$ are generated by Casimirs the same is true for their outer tensor product $(A\otimes B,\{\ ,\dotsc ,\ \}_{k+l})$.
\end{corollary}
It follows that the  $Z$'s vanish for the invariant ring of any finite subgroup $G\subseteq\SL(2,\C)\times\dots\times\SL(2,\C)\subseteq\SL(2l,\C)$, where we have taken $l$ copies of $\SL(2,\C)$. The same holds true for outer tensor products of those with Nambu-Poisson algebras from Corollary \ref{cor:detNP}. 

\section{Maurer-Cartan equation}
\label{sec:MC}
Throughout this section $m$ is an even positive integer, $S=\boldsymbol{k}[x_{1},x_{2},...,x_{n}]$ is the polynomial algebra and $\{\ ,...,\ \}$ be a Nambu-Poisson bracket on $S$ of arity $m$.

The $S$-module of Kähler differentials $\Omega_{S|\kk}$ is a free $S$-module dual to the $S$-module of $\kk$-linear derivations $\Der(S):=\{X:S\rightarrow S|X(fg)=fX(g)+gX(f)\}$. This extends to an isomorphism $\bigwedge^j \Der(S)\cong \operatorname{Alt}_{S}^{j}(\Omega_{S|\kk},S)$ between the $j$-th $S$-linear exterior power $\Der(S)$  and alternating $j$-multilinear maps from $\Omega_{S|\boldsymbol{k}}$ to $S$. 

Recall that $\bigwedge \Der(S)$ is a Gerstenhaber algebra with respect to the so-called \emph{Schouten} bracket $\llbracket \ ,\ \rrbracket$ (confer, e.g., \cite{xu1999gerstenhaber}). This means that 
\begin{enumerate}
\item $XY=(-1)^{|X||Y|}YX$,
\item $\llbracket X,Y\rrbracket=-(-1)^{(|X|-1)(|Y|-1)}\llbracket Y,X\rrbracket$,
\item $\llbracket X,YZ\rrbracket=\llbracket X,Y\rrbracket Z+(-1)^{|Y|(|X|-1)}Y\llbracket X,Z\rrbracket $,
\item $\llbracket X,\llbracket Y,Z\rrbracket\rrbracket=\llbracket \llbracket X,Y\rrbracket,Z\rrbracket+(-1)^{(|X|-1)(|Y|-1)}\llbracket Y,\llbracket X,Z\rrbracket \rrbracket$,
\end{enumerate}
where $X\in \bigwedge ^{|X|}\Der(S)$, $Y\in \bigwedge ^{|Y|}\Der(S)$
and $Z\in \bigwedge ^{|Z|}\Der(S)$. In fact, 
$\llbracket \ ,\ \rrbracket$ is the unique bracket satisfying those axioms whose restriction to $\Der(S)=\bigwedge ^{1}\Der(S)$ is the commutator. There is a coordinate formula for Schouten bracket (see \cite[p.541]{cattaneo2007relative}) that we find to be convenient for our purposes. Here put $\xi^i:=\partial/\partial x_i[-1]$ and identify $\bigwedge^n\Der(S)=S[\xi^1,\xi^2,\dots,\xi^n]$ (the latter being interpreted in the tensor category of $\Z$-graded  vector spaces) and notice that this is an algebra isomorphism since $\Der(S)$ concentrated in degree $0$.
\begin{equation*}\label{eq:leftrightSchouten}
\llbracket X,Y\rrbracket
=\sum_{l=1}^{m}\biggl(X\overleftarrow{\partial\over \partial\xi^{l}}\overrightarrow{\partial\over \partial{x_{l}}}Y-X\overleftarrow{\partial\over \partial x_{l}}\overrightarrow{\partial\over\partial{\xi^{l}}}Y\biggr),
\end{equation*}
where $\overleftarrow{\partial_{x_{l}}}:=\overleftarrow{\partial\over \partial{x_{l}}}$, $\overleftarrow{ \partial_{\xi^{l}}}:=\overleftarrow{\partial\over \partial{\xi^{l}}}$ are the left 
and $\overrightarrow{\partial_{x_{l}}}:=\overrightarrow{\partial \over \partial{x_{l}}}$, $\overrightarrow{ \partial_{\xi^{l}}}:=\overrightarrow{\partial\over \partial {\xi^{l}}}$ are the right derivatives with respect to the coordinates $x_{l}$ and $\xi^{l}$.

According to \cite{vaisman1999survey} the \emph{Nambu-Poisson tensor} $\Pi\in\bigwedge ^{m}\Der(S)$, i.e., the unique tensor such that $\Pi(\dR f_1,\dots,\dR f_m)=\{f_1,\dots,f_m\}$ for $f_1,\dots,f_m\in S$, satisfies $\llbracket \Pi,\Pi\rrbracket=0$. 
Typically this is untrue if $m$ is odd. The differential of Nambu-Poisson cohomology $\dN:\bigwedge ^{\bullet}\Der(S)\to \bigwedge ^{\bullet+m-1}\Der(S)$ is defined by the formula $\dN:=\llbracket \Pi,\ \rrbracket$. 
\begin{proposition}\label{prop:dnambu}
For $Y\in\bigwedge ^{|Y|}\Der(S)$ we have
\begin{multline}\label{eq:dnambu}
( \dN Y)( \dR f_{1} ,\dotsc ,\dR f_{m+| Y| -1})\\
=\sum _{\sigma\in\UnSh_{m-1,|Y|}} (-1)^{|\sigma|} X_{f_{\sigma(1)} ,\dotsc ,f_{\sigma(m-1)}}\left (Y\left( \dR f_{\sigma(m)} ,\dotsc, \dR f_{\sigma(m+| Y|-1)}\right)\right)\\
-\sum _{\sigma\in\UnSh_{m,|Y|-1}}(-1)^{|\sigma|} Y\left( \dR\{f_{\sigma(1)} ,\dotsc ,f_{\sigma(m)}\} ,\dR f_{\sigma(m+1)},\dotsc ,\dR f_{\sigma(m+| Y|-1)}\right)
\end{multline}
where $\dR f_{1} ,\dotsc ,\dR f_{m+| Y| -1}\in\Omega_{S|\kk}$. This formula extends to arbitrary elements in $\Omega_{S|\kk}$ by $A$-linearity.
\end{proposition}
\begin{proof}
We use Einstein's summation convention and the following shorthand notation. In the $\Sigma _{m}$-orbit $\Sigma _{m}( i_{1} ,\dotsc ,i_{m})$ of $(i_{1} ,\dotsc ,i_{m}) \in \{1,\dotsc ,n\}^{m}$ with pairwise distinct $i_{1} ,\dotsc ,i_{m}$ there is a unique totally ordered $(i'_{1} ,\dotsc ,i'_{m})$. Let us write $I=\{i'_{1} ,\dotsc ,i'_{m}\}$. In other words, $I$ is a totally ordered subset of $\{1,\dotsc ,n\}$ of cardinality $m$. In this way we write for example
\begin{align*}
\frac{\Pi _{i_{1} \dotsc i_{m}}}{m!} \xi ^{i_{1}} \cdots \xi ^{i_{m}} =\Pi _{I} \xi ^{I},
\end{align*}
where the right hand side is also to interpreted with the summation convention, i.e., we sum over repeated indices $I$. 

It is enough to show the identity for $f_{s} =x_{j_{s}}$, where $s=1,\dotsc ,m+|Y|-1$ and $J=\{j_{1} ,\dotsc ,j_{m+|Y|-1}\}$ is an ordered set.
 \begin{align*}
&\frac{\Pi _{i_{1} \dotsc i_{m}}}{m!} \xi ^{i_{1}} \cdots \xi ^{i_{m}}\overleftarrow{\frac{\partial }{\partial \xi ^{j}}}\overrightarrow{\frac{\partial }{\partial x_{j}}} Y_{L} \xi ^{L}( \dR x_{j_{1}} ,\dotsc ,\dR x_{j_{m+| X| -1}})\\
&=\sum _{r=1}^{m}( -1)^{^{m-r}}\frac{\Pi _{i_{1} \dotsc i_{m}}}{m!} \xi ^{i_{1}} \cdots \widehat{\xi ^{i_{r}}} \cdots \xi ^{i_{m}}\frac{\partial Y_{L}}{\partial x_{i_{r}}} \xi ^{L}( \dR x_{j_{1}} ,\dotsc ,\dR x_{j_{m+| X| -1}})=\sum _{ \substack{
I\sqcup L=J,\\
|I| =m-1}}(-1)^{|\{i\in I,l\in L\mid i >l\}|} X_{I}( Y_{L}), 
\end{align*}
which corresponds to the first term in Eqn. \eqref{eq:dnambu}. On the other hand,
\begin{align*}
    &\frac{\Pi _{i_{1} \dotsc i_{m}}}{m!} \xi ^{i_{1}} \cdots \xi ^{i_{m}}\overleftarrow{\frac{\partial }{\partial x_{j}}}\overrightarrow{\frac{\partial }{\partial \xi ^{j}}} Y_{L} \xi ^{L}( \dR x_{j_{1}} ,\dotsc ,\dR x_{j_{m+|X|-1}})\\
    &=\sum _{r=1}^{|Y|}( -1)^{r-1}\frac{\partial \Pi _{I}}{\partial x_{l_{r}}}\frac{Y_{l_{1} \dotsc l_{|Y|}}}{|Y|!} \xi ^{I} \xi ^{l_{1}} \cdots \widehat{\xi ^{l_{r}}} \cdots \xi ^{l_{|Y|}}(\dR x_{j_{1}} ,\dotsc ,\dR x_{j_{m+|X|-1}})\\
&=\sum _{ \substack{
I\sqcup K=J,\ L=\{l\} \sqcup K,\\
|I|=m}}( -1)^{|\{i\in I, l\in L| i >l\} |}( -1)^{|\{k\in K\mid k< l\} |} \ \frac{\partial \Pi _{I}}{\partial x_{l}} Y_{L},
\end{align*}
which lines up with the second term in Eqn. \eqref{eq:dnambu}.
\end{proof}

The purpose of this section is to interpret the tensors $Z_{i_{1}\cdots i_{m-1}\mu}^{\nu}\in S$ (see \eqref{eq:Zs}) as solutions of an appropriate Maurer-Cartan equation.
The underlying space of the requisite dg Lie algebra is defined to be
$\mathfrak{N} :=\bigoplus _{i\geq 0} \bigwedge^{i( m-1)} \Der( S)\otimes_{\kk} \mathfrak{gl}_k$ where we put $\mathfrak{N}^{i} :=\bigwedge^{i( m-1)} \Der(S)\otimes_{\kk} \mathfrak{gl}_k$. Note that $\bigoplus _{i\geq 0} \mathfrak{N}^{i}$ forms a subalgebra of $\bigwedge\Der( S)$ with respect to the supercommutative product. We can tensor this supercommutative algebra $\bigoplus _{i\geq 0} \bigwedge^{i( m-1)} \Der( S)$ with the Lie algebra $\mathfrak{gl}_k$ to obtain a super Lie bracket on $\mathfrak{N}$. Moreover, we define $\dN:=\llbracket \Pi,\: \rrbracket\otimes \id:\mathfrak{N}^{\bullet}\to \mathfrak{N}^{\bullet+1}$, where $\Pi$ is the even Nambu-Poisson tensor and $\llbracket \:,\: \rrbracket$ is the Schouten bracket. The result is the dg Lie algebra $\left({\mathfrak{N}}=\oplus_{i\geq 0} \mathfrak{N}^{i} ,\dN ,[ \ ,\ ]\right)$. To refer to the $Z$ tensors (see \eqref{eq:Zs}) we also use the notations
\begin{align*}
Z_{j_{1}\cdots j_{m-1}}&=\left(Z_{j_{1}\cdots j_{m-1} \mu}^{\nu}\right)_{\mu,\nu=1,\dots k}\in\mathfrak{gl}_k(S)=S\otimes_{\kk}\mathfrak{gl}_k,\\
Z&=\left(\sum_{j_1,\dots,j_{m-1}}\dfrac{Z_{j_{1}\cdots j_{m-1} \mu}^{\nu}}{(m-1)!}\xi^{j_{1}}\cdots \xi^{j_{m-1}}\right)_{\mu,\nu=1,\dots k}\in \mathfrak{N}^{1},\\
Z_\mu^\nu&=\sum_{j_1,\dots,j_{m-1}}\dfrac{Z_{j_{1}\cdots j_{m-1} \mu}^{\nu}}{(m-1)!}\xi^{j_{1}}\cdots \xi^{j_{m-1}}\in\bigwedge \ \!\!\!^{m-1}\Der(S).
\end{align*}

The $A$-module of Hamiltonian derivations $\Der^\Ham(A)$ is an $A$-submodule of the $A$-module of  \emph{Nambu derivations}
\begin{align*}
    \Der^\Nambu &:=\{X\in\Der(A)\ \mid \ X(\{a_1,a_2,\dots,a_m\})=\{X(a_1),a_2,\dots,a_m\}+\{a_1,X(a_2),\dots,a_m\}+\\
    & \ \ \ \ \ \ \  \ \ \ \ \ \ \  \ \ \ \ \ \ \  \ \ \ \ \ \ \  \ \ \ \ \ \ \  \ \ \ \ \ \ \   \ \ \ \ \ \ \  \ \ \ \ \ \ \ \cdots+\{a_1,a_2,\dots,X(a_m)\} \ \forall a_1,a_2,\dots,a_m\in A\}.
\end{align*}
It is straightforward to see that $(\Der^\Ham(A),A)$ forms a sub-Lie-Rinehart algebra of $(\Der^\Nambu(A),A)$.
Obviously, in contrast to the Poisson case $m=2$, for $m>2$ the $A$-modules $\Der^\Ham(A)$ and $\Der^\Nambu(A)$ do not live in $\bigoplus_{i\geq 0}\bigwedge^{i(m-1)}\Der(A)$ and cannot be directly related to Nambu cohomology.

\begin{theorem}\label{thm:MC}
Let $m$ be even.
Writing the generators as a column vector $\vec f=(f_1,\dots,f_k)^T$, we have that $(\dN Z-[Z,Z])\vec f\in I^2\mathfrak N$.
\end{theorem}
\begin{proof}
The proof is a consequence of the Theorems \ref{thm:Nmod=flat} and \ref{thm:Ncurv} of the next section.
\end{proof}

\section{The flat Nambu connection on the conormal module}
\label{sec:conn}

Let $(A,\{\ ,\dotsc ,\ \})$ an Nambu-Poisson algebra and $V$ be an $A$-module. Let $B:=\SA_{A}( V) /\ker(\epsilon)^{2}$, where $\epsilon :\operatorname{S}_{A}( V)\rightarrow A$ is the augmentation. A \emph{Nambu module} structure on $V$ is a Nambu-Poisson algebra structure 
on $B$ extending that of $A\subseteq B$ such that for $V\subseteq B$ we have $\{B,\dotsc ,B,V,V\} =0$. This boils down to an operation
$\{\ ,\dotsc ,\ \} :A\times \cdots \times A\times V\rightarrow V$ satisfying the obvious axioms. 

\begin{definition}\label{def:Nconn}
Let $(A=S/I,\{\ ,\dotsc ,\ \})$ an affine Nambu-Poisson algebra of arity $m$.
By a \emph{Nambu connection} on an $A$-module $ V$ we mean a map 
\begin{align*}
    \nabla :V\rightarrow \Hom_{A}\left( \Der^{\Ham}( A) ,A\right) \otimes _{A} V, \ \ \ \ v\mapsto (X \mapsto \nabla _{X} v)
\end{align*} satisfying 
\begin{enumerate}
    \item $ \nabla _{X}( av) =a\nabla _{X}( v) +X( a) v$,
    \item $ [ \nabla _{X_{i_{1} \dotsc i_{m-1}}} ,\nabla _{X_{j_{1} \dotsc j_{m-1}}}] v=\sum_{r=1}^n\sum _{l=1}^{m-1}\frac{\partial \Pi_{i_{1} \dotsc i_{m-1}j_l}}{\partial x_r} \nabla _{X_{j_{1} \dotsc r \dotsc j_{m-1}}}v$
\end{enumerate}
for all $ a\in A,v\in V$ and $X\in\Der^\Ham(A)$. By its \emph{curvature} we mean the map
\begin{align*}
    ( X,Y) \mapsto \mathcal R_{X,Y} :=[ \nabla _{X} ,\nabla _{Y}] -\nabla _{[ X,Y]}, 
\end{align*}
where $ X,Y\in \Der^{\Ham}( A)$. A standard calculation shows that this map is actually an element $\mathcal R$ in $ \Alt_{A}^{2}\left( \Der^{\Ham}( A) ,A\right) \otimes _{A} \End_{A}( V)$. If $\mathcal R=0$ the connection $ \nabla $ is called \emph{flat}.
\end{definition}

\begin{theorem}\label{thm:Nmod=flat}
Let $(A=S/I,\{\ ,\dotsc ,\ \})$ an affine Nambu-Poisson algebra.
A Nambu module structure on $ V$ over $ ( A,\{\ ,\dotsc ,\ \})$ is nothing but a flat
Nambu connection on the $ A$-module $ V$.
\end{theorem}

\begin{proof}
Let $V$ be a Nambu module over $( A,\{\ ,\dotsc ,\ \})$. We define $\nabla _{X_{i_{1} \dotsc i_{m-1}}} v:=\{x_{i_{1}} ,\dotsc ,x_{i_{m-1}} ,v\}$. The Leibniz rule of Definition \ref{def:NPbracket}(\ref{item:NLeib}) for the Nambu-Poisson algebra $ B=\SA_{A}( V) /\ker( \epsilon )^{2}$ shows identity (1). By the fundamental identity for $B=\SA_{A}( V) /\ker( \epsilon )^{2}$ we have 
\begin{align*}
\nabla _{X_{i_{1} \dotsc i_{m-1}}}\{x_{j_{1}} ,\dotsc ,x_{j_{m-1}} ,v\} =\{x_{j_{1}} ,\dotsc ,x_{j_{m-1}} ,\nabla _{X_{i_{1} \dotsc i_{m-1}}} v\} +\sum _{l}\{x_{j_{1}} ,\dotsc ,X_{i_{1} \dotsc i_{m-1}}( x_{j_{l}}) ,\dotsc x_{j_{m-1}} ,v\},
\end{align*}
which proves (2). From the calculation 
\begin{align*}
&[ \nabla _{X_{i_{1} \dotsc i_{m-1}}} ,\nabla _{X_{j_{1} \dotsc j_{m-1}}}] v=\{x_{i_{1}} ,\dotsc ,x_{i_{m-1}} ,\{x_{j_{1}} ,\dotsc ,x_{j_{m-1}} ,v\}\} -\{x_{j_{1}} ,\dotsc ,x_{j_{m-1}} ,\{x_{i_{1}} ,\dotsc ,x_{i_{m-1}} ,v\}\}\\
&=[ X_{i_{1} \dotsc i_{m-1}} ,X_{j_{1} \dotsc j_{m-1}}] (v) =\nabla _{[ X_{i_{1} \dotsc i_{m-1}} ,X_{j_{1} \dotsc j_{m-1}}]} v
\end{align*}
we conclude that $ \nabla $ is flat. 

Conversely, assume $ \nabla $ to be a flat Nambu connection and put $\{x_{i_{1}} ,\dotsc ,x_{i_{m-1}} ,v\}:=\nabla _{X_{i_{1} \dotsc i_{m-1}}} v$. From the flatness we conclude that the bracket on $ B$ extends the one on $ A\subseteq B$. Identity (1) of Definition \ref{def:Nconn} shows the Leibniz rule for the bracket $\{\ ,\dotsc ,\ \}$, while (2) shows the fundamental identity.
\end{proof}

An example of a Nambu module over $A=S/I$ is the conormal module $I/I^{2}$, where $I$ is a Nambu-Poisson ideal.
If $I=( f_{1} ,\dotsc ,f_{k}) \subseteq S=\kk\left[ x_{1} ,\dotsc ,x_{n}\right]$ is a Nambu-Poisson ideal, then $I$ and $I^{2}$ are stable under the action of $\Der^\Ham(A)$. The flat Nambu connection associated to the Nambu module is
 \begin{align*}
     \nabla_{X_{i_1\dots i_{m-1}}}(f_\mu+I^2)=X_{i_1\dots i_{m-1}}(f_\mu)+I^2=\sum_{\nu}Z_{i_1\dots i_{m-1}\mu}^\nu f_\nu+I^2.
\end{align*} 
 We refer to $\nabla$ as the \emph{flat Nambu connection defined by the} $Z_{i_1\dots i_{m-1}\mu}^\nu$'s. The experts call the $Z_{i_1\dots i_{m-1}\mu}^\nu$ also the \emph{Nambuffel symbols} of the connection. The connection $\nabla$ can be also understood as the restriction of the Bott connection $\nabla^B:I/I^2\to\Omega_{A|\kk}\otimes_A I/I^2$ to $\Der^\Ham(A)$ (see \cite{PoissonHomology}).  
 \begin{theorem}\label{thm:Ncurv}
 Let the arity $m$ of the Nambu-Poisson bracket  $\{\ ,\dotsc ,\ \}$ on $A=S/I$ be even and let $\nabla$ be the flat Nambu connection defined by the $Z_{i_1\dots i_{m-1}\mu}^\nu$'s.
  For the  curvature $\mathcal R$ of $\nabla$ we have $\mathcal Rv=(\dN Z -[Z ,Z])v$, where $v=\sum_\mu v^\mu f_\mu+I^2\in I/I^2$. 
\end{theorem}

\begin{proof}
We evaluate
\begin{align*}
\nabla _{X_{i_{1} \dotsc i_{m-1}}} \nabla _{X_{j_{1} \dotsc j_{m-1}}}\left( f_{\mu } +I^{2}\right) &=X_{i_{1} \dotsc i_{m-1}}\left( Z_{j_{1} \dotsc j_{m-1} \mu }^{\nu } f_{\nu }\right) +I^{2}\\
&=X_{i_{1} \dotsc i_{m-1}}\left( Z_{j_{1} \dotsc j_{m-1} \mu }^{\nu }\right) f_{\nu } +Z_{j_{1} \dotsc j_{m-1} \mu }^{\nu } Z_{i_{1} \dotsc i_{m-1} \nu }^{\lambda } f_{\lambda } +I^{2}
\end{align*}
and $\nabla _{[ X_{i_{1} \dotsc i_{m-1}} ,X_{j_{1} \dotsc j_{m-1}}]}\left( f_{\mu } +I^{2}\right) =\sum _{r=1}^{n}\sum _{l=1}^{m-1}\frac{\partial \Pi _{i_{1} \dotsc i_{m-1} j_{l}}}{\partial x_{r}} Z_{j_{1} \dotsc r\dotsc j_{m-1} \mu }^{\nu } f_{\nu } +I^{2}$ using property (2) of Definition \ref{def:Nconn}. Combining these terms we obtain
\begin{align*}
&([ \nabla _{X_{i_{1} \dotsc i_{m-1}}} ,\nabla _{X_{j_{1} \dotsc j_{m-1}}}] -\nabla _{[ X_{i_{1} \dotsc i_{m-1}} ,X_{j_{1} \dotsc j_{m-1}}]})\left( f_{\mu } +I^{2}\right)\\
&=\biggr( X_{i_{1} \dotsc i_{m-1}}\left( Z_{j_{1} \dotsc j_{m-1} \mu }^{\nu }\right) -X_{j_{1} \dotsc j_{m-1}}\left( Z_{i_{1} \dotsc i_{m-1} \mu }^{\nu }\right) -\sum _{r=1}^{n}\sum _{l=1}^{m-1}\frac{\partial \Pi _{i_{1} \dotsc i_{m-1} j_{l}}}{\partial x_{r}} Z_{j_{1} \dotsc r\dotsc j_{m-1} \mu }^{\nu } \\
&\ \ \ \ \ \ \ \ \ \ \ \ \ \ \ \ \ \ \ \ \ \ \ \ \ \ \ \ \ \ \ \ \ \ \ \ \ \ \ \ \ \ \ \ \ \ \ \ \ \ \ \ \ \ \ \ \ \ \ \ \ \ \ \ \ \ \ \ \ \ \ \ \ \ \ \ \ \ \ \ -[ Z_{i_{1} \dotsc i_{m-1}} ,Z_{j_{1} \dotsc j_{m-1}}]_{\mu }^{\nu }\biggl) f_{\nu } +I^{2}.
\end{align*}
In view of Proposition \ref{prop:dnambu} this is nothing but $(\dN Z -[Z ,Z])v$.
\end{proof}

More  generally, a Nambu connection is unique up to addition of an $\omega \in \Hom_{A}\left(\Der^{\Ham}( A) ,A\right)$. That is, any other Nambu connection is of the form
\begin{align*}
\nabla _{Y}^{\omega } =\nabla _{Y} +\omega ( Y)\mbox{ for } Y\in \Der^{\Ham}( A).
\end{align*}
This $\omega $ could be, for example, just the restriction of a form $\Omega \in \Omega _{A|\boldsymbol{k}}$. The curvature 
\begin{align*}
    \mathcal{R}^{\omega } \in \Alt_{A}^{2}\left(\Der^{\Ham}( A) ,A\right) \otimes _{A} I/I^{2}
\end{align*}
of $\nabla ^{\omega }$ is given by the formula
\begin{align}\label{eq:omegacurv}
( Y,Y') \mapsto \mathcal{R}_{Y,Y'}^{\omega } =[ \nabla _{Y} +\omega ( Y) ,\nabla _{Y'} +\omega ( Y')] -\nabla _{[ Y,Y']} +\omega ([ Y,Y']) =( d_{LR} \omega )( Y,Y') .
\end{align}
where $d_{LR} :\Alt^{\bullet }\left( \Der^{\Ham}( A) ,A\right)\rightarrow \Alt^{\bullet +1}\left(\Der^{\Ham}( A) ,A\right)$ is the de Rham-differential of the Lie-Rinehart algebra $\left( \Der^{\Ham}( A) ,A\right)$ (see \cite{Rinehart}).

If $\omega$ is the restriction of $\Omega \in \Omega _{A|\boldsymbol{k}}$ then
\begin{align}\label{eq:lift}
    ( d_{LR} \omega )( Y,Y') =\dR\tilde{\Omega }\left(\tilde{Y} ,\widetilde{Y'}\right) +I,
\end{align}
where $\tilde{\Omega } \in \Omega _{S|\boldsymbol{k}}$ is a lift of $\omega \in \Omega _{A|\boldsymbol{k}}$ in the cotangent sequence  and $\tilde{Y} ,\widetilde{Y'} \in \Der( S)_{I} :=\{X\in \Der( S) \mid X( I) \subseteq I\}$ are lifts of $Y,Y'\in \Der^{\Ham}( A) \subseteq \Der( A) \simeq \Der( S)_{I} /I\Der( S)_{I}$ (see \cite[Lemma (2.1.2)]{Simis}). In \eqref{eq:lift} $\dR:\Omega _{S|\boldsymbol{k}}\rightarrow \land _{S}^{2} \Omega _{S|\boldsymbol{k}}$ is just the usual exterior differential.

\section{$\Gamma$-graded $L_\infty$-algebroids and Lie-Rinehart $m$-algebras}
\label{sec:form}

In this section we do not assume that $m$ is even.

For our applications to Nambu-Poisson brackets we use a slight generalization of the $ L_{\infty }$-algebroids studied in \cite{higherKoszul}. That gives us also the opportunity to clarify a certain point about higher anchors that the authors of \cite{higherKoszul} missed (see remark \ref{rem:binanch}).

Let $ \Gamma $ be an abelian group with a homomorphism $ \Gamma \rightarrow \Z_{2} ,\ \gamma \rightarrow |\gamma |$. We will also assume a bijection $  \Pi _{\Gamma } :\Gamma \rightarrow \Gamma $ , $ \gamma \mapsto \gamma +\delta $, $ \delta  \in \Gamma $ with fixed, compatible with the parity change $  \Pi :\mathbb{Z}_{2}\rightarrow \mathbb{Z}_{2} ,  \Pi (\overline{i}) :=\overline{i+1}$, i.e., $ | \Pi _{\Gamma }( \gamma ) |= \Pi ( |\gamma |)$. 
Thoughout the paper we use the notation $ \overline{i} :=i\bmod 2$. Our principal examples will be:
\begin{enumerate}
    \item $ \Gamma =\mathbb{Z}_{2} ,\ \delta =\overline{1}$,
    \item $ \Gamma =\mathbb{Z\rightarrow Z}_{2} ,\ i\mapsto \overline{i} ,$ and $ \delta =1$.
\end{enumerate}
By a $ \Gamma $-graded $ A$-module $ V=\bigoplus _{\gamma \in \Gamma } V_{\gamma }$ we mean a direct sum of $ A$-modules $ V_{\gamma }$. If we have two such $ \Gamma $-graded $ A$-modules $ V=\bigoplus _{\gamma \in \Gamma } V_{\gamma }$ and \ $ W=\bigoplus _{\gamma \in \Gamma } W_{\gamma }$ their tensor product is also a 
$ \Gamma $-graded $ A$-module $ V\otimes W=\bigoplus _{\gamma \in \Gamma }( V\otimes W)_{\gamma }$ with $ ( V\otimes W)_{\gamma } =\bigoplus _{\gamma ',\gamma ''\in \Gamma :\ \gamma =\gamma '+\gamma ''}( V_{\gamma '} \otimes W_{\gamma ''})$. A homomorphism of $ A$-modules $ \phi :V=\bigoplus _{\gamma \in \Gamma } V_{\gamma }\rightarrow W=\bigoplus _{\gamma \in \Gamma } W_{\gamma }$ is said to be \emph{$\Gamma $-graded} if $ \phi ( V_{\gamma }) \subseteq W_{\gamma }$ for each $ \gamma \in \Gamma $. The category of $ \Gamma $-graded $ A$-modules will be denoted by $ \operatorname{Mod}{_{A}^{\Gamma }}$. 
We have a forgetful functor
$ \operatorname{Mod}_{A}^{\Z}\rightarrow \operatorname{Mod}_{A}^{\Z_{2}}$. We say that $ v\in V$ is  \emph{homogeneous} if there is a $ \gamma $ such that \ $ v\in V_{\gamma }$. In this case we write $|v|=|\gamma |$. The braiding $ V\otimes W\rightarrow W\otimes V$ is given by the Koszul sign, i.e., $ v\otimes w\mapsto ( -1)^{|v||w|} w\otimes v$. With regards to this braiding one can then construct the tensor algebra $ \operatorname{T} V$, the symmetric algebra $ \operatorname{S} V$ and the exterior algebra $ \land V$. They form $ \Z \times \Gamma $-graded $ A$-algebras. Sometimes we regard $ v\in V_{\gamma }$ to have degree $  \Pi _{\Gamma }( \gamma ) =\gamma +\delta $. In this case we write $  \Pi v$ instead of $ v$. This gives rise to an endofunctor $ V\rightarrow  \Pi V$ in the category $ \operatorname{Mod}_{A}^{\Gamma }$. There is the décalage isomorphism $  \Pi ^{n}\left( \land ^{n} V\right) \simeq \operatorname{S}^{n}(  \Pi V)$ of $ \Gamma $-graded $ A$-modules.
We also consider the subcategory $ \operatorname{pMod}_{A}^{\Gamma }$ of $\Gamma$-graded $A$-projective modules.

\begin{definition}
By a \emph{$\Gamma$-graded $L_{\infty}$-algebra} of degree $\nu\in \Gamma$ with $|\nu|=\overline{0}$ we mean an object  $L=\oplus_{\gamma\in\Gamma}L_{\gamma}$ in $ \operatorname{Mod}_{\kk}^{\Gamma }$ with a sequence $([\:,\dots,\:]_{n})_{n\geq 1}$ of $\kk$-linear operations $[\:,\dots,\:]_{n}:\bigwedge^{n}L\rightarrow L$ of degree $|[\:,\dots,\:]_{n}|=\nu-n\delta$  such that $n\geq 1$
\begin{equation}\label{eq:LinftJacobi}
    \sum_{p+q=n+1}\sum_{\sigma\in \UnSh_{q,p-1}}(-1)^{\sigma}\varepsilon(\sigma,\boldsymbol{x})(-1)^{q(p-1)}[[x_{\sigma(1)},\dots,x_{\sigma(q)}]_{q},x_{\sigma(q+1)},\dots,x_{\sigma(n)}]_{p}=0
\end{equation}
for homogeneous $x_{1},\dots,x_{n}\in L$. Here $\varepsilon(\sigma,\boldsymbol{x})=(-1)^{\sum_{i<j,\sigma(i)>\sigma(j)}|x_{i}||x_{j}|}$ is the \emph{Koszul sign} of the permutation $\sigma\in\Sigma_n$ and $(-1)^{\sigma}$ its sign. Here $\UnSh_{q,p-1}$ stands for the $(q,p-1)$-unshuffle permutations, i.e., the set of permutations $\sigma$ of $\{1,2,\dots,n\}$ such that $\sigma(1)<\sigma(2)<\cdots<\sigma(q)$ and $\sigma(q+1)<\sigma(q+2)<\cdots<\sigma(n)$. 
\end{definition}

Note that a $\Z$-graded $L_\infty$-algebra with $\delta=1$ of degree $2$ is what is usually understood as an $L_\infty$-algebra. Applying the forgetful functor $\operatorname{Mod}_{\kk}^{\Z}\rightarrow \operatorname{Mod}_{\kk}^{\Z_{2}}$ every $\Z$-graded $L_\infty$-algebra can be also understood as a $\Z_2$-graded $L_\infty$-algebra.

\begin{definition}
By a \emph{$\Gamma$-graded $\Pi L_{\infty}$-algebra} structure of degree $\mu$, $|\mu|=\overline{1}$, on an object $E=\oplus_{n\in\Z}E^{n}$ in  $\operatorname{Mod}_{\kk}^{\Gamma }$ we mean a sequence $(l_{n})_{n\geq 1}$ of $\kk$-linear operations $l_{n}:\SA^{n}E\rightarrow E$ of degree $|l_{n}|=\mu$ such that for all $n\geq 1$
\begin{equation}\label{eq:decalage}
    \sum_{p+q=n+1}\sum_{\sigma\in \UnSh_{q,p-1}}\varepsilon(\sigma, \boldsymbol{e})l_{p}(l_{q}(e_{\sigma(1)},\dots,e_{\sigma(q)}),e_{\sigma(q+1)},\dots,e_{\sigma(n)})=0.
\end{equation}
for homogeneous $e_{1},\dots,e_{n}\in E$. 
\end{definition}
Note that a $\Z$-graded $\Pi L_\infty$-algebra with $\delta=1$ of degree $1$ is what is usually called an $L_\infty[1]$-algebra. Applying the forgetful functor $\operatorname{Mod}_{\kk}^{\Z}\rightarrow \operatorname{Mod}_{\kk}^{\Z_{2}}$ every $\Z$-graded $\Pi L_\infty$-algebra can be also understood as a $\Z_2$-graded $\Pi L_\infty$-algebra.

The décalage isomorphism also works in our slightly generalized setup.
A $\Gamma$-graded $\Pi L_{\infty}$-algebra structure of degree $\mu$ on $\Pi L$ is equivalent to a $\Gamma$-graded $L_{\infty}$-algebra structure of degree $\mu+\delta$ on $L$ by putting
\begin{equation}
    [x_{1},\dots,x_{n}]_{n}=(-1)^{\sum_{i=1}^{n}(n-i)|x_{i}|}\Pi^{-1}(l_{n}(\Pi x_{1},\dots,\Pi x_{n}))
\end{equation} 
for homogeneous $x_{1},\dots,x_{n}\in L$.
The reason is that the algebraic identities that are relevant only depend on the $\Z_2$-grading while the $\Gamma$-degree is taken care of.

\begin{theorem}[Higher derived brackets of \cite{VoronovHigherDerived}]\label{thm:higher}
Let $\mathfrak{G}$ be a $ \Gamma $-graded Lie algebra over $ \boldsymbol{k}$ with an abelian subalgebra $ \mathfrak{A} \subseteq \mathfrak{G}$ and let $ \varepsilon \in \End_{\boldsymbol{k}}( \mathfrak{G})$, $\varepsilon^{2} =\varepsilon$, and $\varepsilon(\mathfrak{G}) \subseteq \mathfrak{A}$ satisfying $\varepsilon[ X,Y]  =\varepsilon[ \varepsilon X, Y] + \varepsilon[ X,\varepsilon Y]$  for all $X,Y\in \mathfrak{G}$. 
Moreover, let $ \Delta \in \mathfrak{G}$ be of degree $ \mu $ such that $ |\mu |=\overline{1}$ and $ [ \Delta ,\Delta ] =0$. Then 
\begin{align}\label{eq:higherderived}
l_{n}( X_{1} ,\dotsc ,X_{n}) :=\varepsilon [ \dotsc [[ \Delta ,X_{1}] ,X_{2}] ,\dotsc ,X_{n}]
\end{align}
for $X_{1} ,\dotsc ,X_{n}\in\mathfrak{G}$ defines a $\Gamma $-graded $\Pi L_{\infty }$-algebra structure of degree $\mu$ on $\mathfrak{A}$.
 \end{theorem}  
 \begin{proof}
 The relevant calculations in \cite{VoronovHigherDerived} depend only on the $\Z_2$-grading. Obviously, the degree $\mu$ of $\Delta$ coincides with the degree of the  $l_n$'s.
 \end{proof}

\begin{definition}
By a \emph{$\Gamma$-graded left $L_{\infty}$-module} over the $\Gamma$-graded $L_{\infty}$-algebra $L$ of degree $\nu$ with brackets $[\:,\dots,\:]_{n}:\bigwedge^{n}L\rightarrow L$, $n\geq 1$ we mean an object $M=\oplus_{\gamma\in\Gamma}M_{\gamma}$ in $\operatorname{Mod}_{\kk}^{\Gamma}$ and a sequence of operations $\rho_{n}:\bigwedge^{n-1}L\otimes M\rightarrow M$, $n\geq 2$, of degree $|\rho_{n}|=\nu-n\delta$ such that for each $n\geq 1$
\begin{equation}\label{eq:repjac}
    \sum_{p+q=n+1}\sum_{\sigma\in \UnSh_{q,p-1}}(-1)^{\sigma}\varepsilon(\sigma,\boldsymbol{x})(-1)^{q(p-1)}k_{p}(k_{q}(x_{\sigma(1)},\dots,x_{\sigma(q)}),x_{\sigma(q+1)},\dots,x_{\sigma(n)})=0
\end{equation}
for $x_{1},\dots,x_{n-1}\in L$ and $x_{n}\in M$, where $k_{n}:\bigwedge^{n}(L\oplus M)\rightarrow M$, $n\geq 1$, is the unique extension of the operations $[\:,\dots,\:]_n$ and $\rho_{n}$ such that $k_{n}$ vanishes when two or more arguments are from $M$. 
\end{definition}

Note that a $\Gamma$-graded left $L_{\infty}$-module $\Z$-graded $\Pi L_\infty$-algebra $L$ with $\delta=1$ of degree $1$ is what is nothing but an $L_\infty$-module over $L$.

\begin{definition}
By a \emph{$\Gamma$-graded $L_{\infty}$-algebroid} over $\Spec(A)$ we mean a $\Gamma$-graded $L_{\infty}$-algebra $L=\oplus_{\gamma\in\Gamma}L_{\gamma}$ with brackets $([\;,\dots,\:]_{n})_{n\geq 1}$, such that each $L_{\gamma}$ is a projective $A$-module, together with operations $\rho_{n}\in\bigwedge^{n-1}_A L\otimes_\kk A\rightarrow A$, $n\geq 2$, that make $A$ a left $L_{\infty}$-module over $L$ satisfying the following properties
\begin{enumerate}
    \item $\partial:=[\:]_1$ is $A$-linear,
    \item for all $n\geq 2$ and homogeneous $x_{1},\dots,x_{s}\in L$ and $a,b\in A$ we have
    \begin{align*}
        [x_{1},\dots,x_{n-1},ax_{n}]_{n}&=\rho_{n}(x_{1},\dots,x_{n-1},a)x_{n}+a[x_{1},\dots,x_{n-1},x_{n}]_{n},\\
        \rho_{n}(x_{1},\dots,x_{n-1},ab)&=\rho_{n}(x_{1},\dots,x_{n-1},a)b+a\rho_{n}(x_{1},\cdots,x_{n-1},b),\\
        \rho_{n}(ax_{1},\dots,x_{n-1},b)&=a\rho_{n}(x_{1},\dots,x_{n-1},b).
    \end{align*}
    We refer to the collection of operations $(\rho_{n})_{n\geq 2}$ as the \emph{higher anchors}. 
\end{enumerate} 
In the special case when $\rho_{n}=0$ for all $n\neq m$ we say that $L$ is a \emph{$\Gamma$-graded $L_{\infty}$-algebroid with $m$-ary anchor} over $\Spec(A)$.
\end{definition}

In contrast to \cite{higherKoszul} we do not assume $L_{\gamma}$ to be finitely generated. This is because our main examples do not have this property when $\Gamma=\Z_2$. 

\begin{remark}\label{rem:binanch}
If $\Gamma=\Z$ and the grading on $L$ is non-positive a \emph{$\Gamma$-graded $L_{\infty}$-algebroid} with $\delta=1$ of degree $2$ over $\Spec(A)$ automatically has a binary anchor. This observation was missed in \cite{higherKoszul}. The $L_\infty$-algebroid on the cotangent complex that was constructed there has no higher anchors. The reason is that the degree of $\rho_n(x_1,\dots,x_{n-1},a)$, being equal to $2-n+\sum_i |x_i|$, cannot be zero when $n>2$.  
\end{remark}

\begin{definition}
A \emph{$\Gamma$-graded $P_{\infty}$-algebra} of degree $\nu$ is a commutative algebra $R=\oplus_{\gamma\in\Gamma}R_\gamma$ in the monoidal category $\operatorname{Mod}_{A}^{\Gamma}$ such that each $R_\gamma$ is a projective $A$-module which is
also a $\Gamma$-graded  $L_{\infty}$-algebra of degree $\nu$ with brackets $\{\:,\dots,\:\}_{n}:\bigwedge^{n}R\rightarrow R$ such that the Leibniz rule
\begin{equation}
    \{ab,a_{2},\dots,a_{n}\}_{n}=a\{b,a_{2},\dots,a_{n}\}_{n}+(-1)^{|a||b|}b\{a,a_{2},\dots,a_{n}\}_{n}
\end{equation}
holds for $a,b,a_{2},\dots,a_{n}\in R$ with homogeneous $a,b$. 
\end{definition}
 
\begin{definition}
Let $m\geq 2$ be an integer.
A \emph{Lie-Rinehart $m$-algebra} is an $A$-module $V$ with a $\kk$-linear operation $[\ ,\dots, \ ]:\bigwedge^m V\to V$ and an anchor $\rho:\bigwedge_A^{m-1}V\otimes_{\kk} A$ that satisfy 
\begin{align}
\label{eq:mLRJac}&\sum_{\sigma \in \UnSh_{m,m-1}}( -1)^{\sigma }[ [v_{\sigma (1)} ,\dotsc ,v_{\sigma (m)} ] ,v_{\sigma (m+1)} ,\dotsc ,v_{\sigma (2m-1)}]=0,\\
&\label{eq:mLeib1} [v_{1} ,\dotsc ,v_{m-1},av_{m} ]=a[v_{1} ,\dotsc ,v_{m} ]+\rho (v_{1} ,\dotsc ,v_{m-1},a)v_m,\\
\label{eq:manchor}    &\sum _{\sigma \in \UnSh_{m,m-2}}( -1)^{\sigma }\Big( \rho ([ v_{\sigma (1)} ,\dotsc ,v_{\sigma (m)}] ,v_{\sigma (m+1)} ,\dotsc ,v_{\sigma (2m-2)} ,a) \\
 \nonumber   &\ \ \ \ \ \ \ \ \ \ \ \ \ \ \ \ \ \ \ \ \ \ \ \ \ \ \ \ \ +\rho ( v_{\sigma (m)} ,\dotsc ,v_{\sigma (2m-2)} ,\rho ( v_{\sigma (1)} ,\dotsc ,v_{\sigma (m-1)} ,a))\Big) =0,\\
 &\label{eq:mLeib2} \rho (v_{1} ,\dotsc ,v_{m-1},ab)=a\rho (v_{1} ,\dotsc ,v_{m-1},b)+b\rho (v_{1} ,\dotsc ,v_{m-1},a)
 \end{align}
for all $v_1,\dots,v_{2m-1}\in V$, $a,b\in A$.
\end{definition}

\begin{proposition}\label{prop:minimal}
 Let $\left(L,([\ ,\dots,\ ]_{n})_{n\geq 1},\rho_{m}\right)$ be a $\Gamma$-graded $L_{\infty}$-algebroid over $\Spec(A)$ with $m$-ary anchor such that $[\ ,\dots, \ ]_j=0$ for all $j\in\{2,\dots,m-1\}$ and let $W\subseteq L_0$ be an $A$-submodule that is preserved under $[\ ,\dots, \ ]_m$ and such that \begin{enumerate}
     \item $\Bild[\ ]_1$ is an ideal, i.e., $[\Bild[\ ]_1,W\dots,W \ ]_m\subseteq \Bild[\ ]_1$ and
     \item $\rho _{m}\left(\Bild[ \ ]_{1} ,W,\dotsc ,W,A\right) =0$.
 \end{enumerate} 
 Then $(\ker[\ ]_1\cap W)/(\Bild[\ ]_1\cap W)$ is a Lie-Rinehart $m$-algebra where the bracket $[\ ,\dots,\ ]$ is induced by $[\ ,\dots,\ ]_{m}$ and the anchor $\rho$ is induced by $\rho_{m}$.
\end{proposition}
If $m=2$ the conditions (1) and (2) are mute. The fact that a dg Nambu-Poisson structure does not automatically produce a Nambu-Poisson structure in cohomology is a bit surprising to those accustomed to more conventional homological algebra.
\begin{proof} Evaluating Eqn. \eqref{eq:repjac} for $p+q=2m-1$ for $w_{1} ,w_{2} ,\dotsc ,w_{2m-1}\in W$ we obtain
\begin{align*}
    0&=\sum _{\sigma \in \UnSh_{m,m-1}}(-1)^{\sigma } [ [w_{\sigma (1)} ,\dotsc ,w_{\sigma (m)} ]_{m} ,w_{\sigma (m+1)} ,\dotsc ,w_{\sigma (2m-1)}]_{m}  \\
    &\ \ \ \ \ \ \ \ \ \ \ \ \ \ \ \ \ \ + [ [w_{1} ,\dotsc ,w_{2m-1} ]_{2m-1}]_{1}+\sum_l(-1)^{l-1}[ [w_{l} ]_{1} ,w_{2} ,\dotsc,w_l, \dotsc,w_{2m-1}]_{2m-1},
\end{align*}
which means that we find Eqn. \eqref{eq:mLRJac} in homology. Evaluating Eqn. \eqref{eq:repjac} for $p+q=2m-1$ for $w_{1} ,w_{2} ,\dotsc ,w_{2m-1}\in W$ and the last argument $a\in A$ yields
\begin{align*}
    0&=\sum _{\sigma \in \UnSh_{m,m-2}}( -1)^{\sigma } \rho _{m}([ x_{\sigma (1)} ,\dotsc ,x_{\sigma (m)}]_{m} ,x_{\sigma (m+1)} ,\dotsc ,x_{\sigma (2m-2)} ,a) \\
    &\ \ \ \ \ \ \ \ \ \ \ \ \ \ \ \ \ \  -\sum _{\sigma \in \UnSh_{m,m-2}}( -1)^{\sigma } k_{m}( \rho _{m}( x_{\sigma (1)} ,\dotsc ,x_{\sigma (m-1)} ,a)_{m} ,x_{\sigma (m)} ,\dotsc ,x_{\sigma (2m-2)})\\
    &=\sum _{\sigma \in \UnSh_{m,m-2}}( -1)^{\sigma } \rho _{m}([ x_{\sigma (1)} ,\dotsc ,x_{\sigma (m)}]_{m} ,x_{\sigma (m+1)} ,\dotsc ,x_{\sigma (2m-2)} ,a)\\
    &\ \ \ \ \ \ \ \ \ \ \ \ \ \ \ \ \ \ +\sum _{\sigma \in \UnSh_{m,m-2}}( -1)^{\sigma } \rho _{m}( x_{\sigma (m)} ,\dotsc ,x_{\sigma (2m-2)} ,\rho _{m}( x_{\sigma (1)} ,\dotsc ,x_{\sigma (m-1)} ,a))\\
    &=\sum _{\sigma \in \UnSh_{m,m-2}}( -1)^{\sigma } \rho _{m}([ x_{\sigma (1)} ,\dotsc ,x_{\sigma (m)}]_{m} ,x_{\sigma (m+1)} ,\dotsc ,x_{\sigma (2m-2)} ,a) \\
    &\ \ \ \ \ \ \ \ \ \ \ \ \ \ \ \ \ \ +\sum _{\sigma \in \UnSh_{m,m-2}}( -1)^{\sigma } \rho _{m}( x_{\sigma (m)} ,\dotsc ,x_{\sigma (2m-2)} ,\rho _{m}( x_{\sigma (1)} ,\dotsc ,x_{\sigma (m-1)} ,a))\\
    &=\sum _{\sigma \in \UnSh_{m,m-2}}( -1)^{\sigma }( \rho_m ([ v_{\sigma (1)} ,\dotsc ,v_{\sigma (m)}] ,v_{\sigma (m+1)} ,\dotsc ,v_{\sigma (2m-2)} ,a) \\
    &\ \ \ \ \ \ \ \ \ \ \ \ \ \ \ \ \ \ +\rho_m ( v_{\sigma (m)} ,\dotsc ,v_{\sigma (2m-2)} ,\rho ( v_{\sigma (1)} ,\dotsc ,v_{\sigma (m-1)} ,a))) =0.
\end{align*}
This shows Eqn. \eqref{eq:manchor}. The Leibniz rules Eqns. \eqref{eq:mLeib1} and \eqref{eq:mLeib2} are inherited from the corresponding relations of $\rho_m$. 
\end{proof}

\section{Resolvent, cotangent complex and the complete Schouten algebra}
\label{sec:resolventetc}

In this section we recall basic definitions and notations from \cite{higherKoszul}. As a reference on the resolvent and the cotangent complex we use \cite{AvramovInfFree,Manetti}.

By a \emph{graded set} we mean a countable set $\mathcal{I}$ with a function $\phi :\mathcal{I}\rightarrow \mathbb{N} =\{1,2,\dotsc \}$ such that for each $l\in \mathbb{N}$ the cardinality of $\mathcal{I}_{l} :=\phi ^{-1} (l)$ is finite. 
To each $i\in \mathcal{I}_{l}$ we attach a variable $x_{i}^{(l)}$ whose parity coincides with the parity of $l$ and introduce the graded polynomial algebra
$S[\boldsymbol x]:=S[x_{i}^{(l)} \  | \ l\in \mathbb{N} ,i\in \mathcal{I}_{l} ]$
with the relation $x_{i}^{(l)} x_{j}^{(n)} =(-1)^{ln} x_{j}^{(n)} x_{i}^{(l)}$. That is, the variable $x_{i}^{(l)}$ is even if $l$ is even and odd otherwise.
We assign the cohomological degree $|x_{i}^{(l)} |:=-l$ to the variables $x_{i}^{(l)}$. By considering only variables up to level $r\geq 1$ we also have the graded polynomial ring in finitely many variables
$S[\boldsymbol x_{\leq r} ]:=S[x_{i}^{(l)} \  | \ l\leq r ,i\in \mathcal{I}_{l} ]$.
By convention $S[\boldsymbol x_{\leq 0} ]:=S$.

A dg $S$-algebra $(R,\partial )$ is called \emph{semifree} if 
\begin{enumerate}
    \item As an $S$-algebra $R$ is a graded polynomial algebra $S[\boldsymbol{x} ]$ over the graded set $\phi :\mathcal{I}\rightarrow \{1,2,\dotsc \}$.
\item For each $l >0$ and $i\in \mathcal{I}_{l}$ we have $\partial (x_{i}^{(l)} )\in S[x_{\leq l-1} ]$.
\end{enumerate}
Clearly, for each $r\geq 0$ $S[\boldsymbol{x}_{\leq r} ]$ forms a semifree dg subalgebra $(R_{\leq r} ,\partial _{\leq r} )$ of $(R ,\partial)$. 
There is $\kappa :R\rightarrow S$ canonical algebra map that sends each variable $x_{i}^{(l)}$ to zero. Denoting the image of $x_{i}^{(l)}$ under $\partial $ by $F_{i} (\boldsymbol{x}_{\leq l-1} )$ we will find it convenient to write
\begin{align*}
\partial _{\leq r} =\sum _{l=1}^{r}\sum _{j\in \mathcal{I}_{l}} F_{j} (\boldsymbol{x}_{\leq l-1} )\frac{\partial }{\partial x_{j}^{(l)}} \ \ \text{and } \ \ \partial =\sum _{l=1}^{\infty }\sum _{j\in \mathcal{I}_{l}} F_{j} (\boldsymbol{x}_{\leq l-1} )\frac{\partial }{\partial x_{j}^{(l)}} .
\end{align*}

Let $I$ be an ideal in $S$. We say that the semifree $S$-algebra $(R,\partial )$ is a  \emph{resolvent }of $A=S/I$ if the composition of the algebra morphisms $\kappa :R\rightarrow S$ and $S\rightarrow A$ is a quasi-isomorphism. It is well-known that such a resolvent always exists.
Notice that $R_{\leq 1}$ is nothing but the Koszul complex seen as a cochain complex. If $I$ is a complete intersection $R_{\leq 1}$ is a resolvent.
If $I$ is not locally a complete intersection the resolvent has infinitely many generators\footnote{In physics people say: the constraint is \emph{infinitely reducible}.} (see \cite{AvramovHerzog}).

In the situation when $I=(f_{1} ,\dotsc ,f_{k} )$ is a homogeneous ideal in $S=\boldsymbol{k} [x_{1} ,\dotsc ,x_{n} ]$ with $\deg (x_{i} )\geq 1$ for $i=1,\dotsc ,n$  such that $I\subseteq \mathfrak{m} =(x_{1} ,\dotsc ,x_{n} )$, we assign to the variable $x_{j}^{(l)}$s internal degrees such that $\deg (\partial )=0$. In this way the resolvent $(R,\partial )$ of $S/I$ becomes a bigraded dg algebra and can  be assumed to be a  \emph{minimal model}. The first terms of the minimal model can be calculated using the Macaulay2 package \texttt{dgalgebras}.

Let $(R,\partial )$ be a resolvent of the $S$-algebra $A=S/I$. The \emph{cotangent complex} of $A$ over $\boldsymbol{k}$ is 
 $\mathbb{L}_{A|\boldsymbol{k}} =A\otimes _{R} \Omega_{R|\boldsymbol{k}}$, where $\Omega_{R|\boldsymbol{k}}$ are the Kähler differentials of $R$ and $\otimes _{R}$ is the tensor product in the category of complexes of $R$-modules. If the ideal $I$ is homogeneous the cotangent complex $\mathbb{L}_{A|\boldsymbol{k}}$ is bigraded in the obvious way. If the generators $f_{1} ,\dotsc ,f_{k}$ of $I$ form a reduced complete intersection then $\mathbb{L}_{A|\boldsymbol{k}}$ forms a projective resolution 
of the $A$-module $\Omega_{R|\boldsymbol{k}}$ (see, e.g., \cite{Kunz}).

A bit more involved is the construction of the \emph{complete Schouten algebra} $\mathfrak g$ of the proaffine superscheme $\operatorname{Spec}(R)$. Recall that for $l\in\Z$ we have the degree shifted $V[l]=\oplus _{k} V[l]^{k}$, $V[l]^{k}:=V^{l+k}$ of the $\Z$-graded module $V=\oplus _{k} V^{k}$.
The map induced from the identity $\downarrow :V\rightarrow V[1]$ is of degree $-1$. Its inverse is $\uparrow$, which is of degree $1$.
We put $\mathfrak{h}_{\leq r} :=\operatorname{S}_{R_{\leq r}} (\operatorname{Der}_{R_{\leq r}} [-1])$, which is an algebra generated by
\begin{align*}
\xi _{(l)}^{i} :=\frac{\partial }{\partial x_{i}^{(l)}} [-1]\in \operatorname{Der}_{R_{\leq r}} [-1]\subset \mathfrak{h}_{\leq r} ,\ \ r\geq m\geq 0,i\in \mathcal{I}_{l},
\end{align*}
which are of degree  $|\xi _{(k_{1} )}^{i_{1}} \cdots \xi _{(k_{\ell } )}^{i_{\ell }} |=k_{1} +\dotsc +k_{\ell } +\ell $. Here our convention is to write
the original variables $x_i$ as $x_i^{(0)}$.
The algebra $\mathfrak{h}_{\leq r}$ can be also understood as the graded polynomial algebra 
\begin{align*}
\boldsymbol{k}\left[ x_{i}^{(l)} ,\xi _{(l)}^{i} |r\geq l\geq 0,i\in \mathcal{I}_{l}\right].
\end{align*} We extend the differential $\partial _{\leq r}$
to $\mathfrak{h}_{\leq r}$ by declaring $\partial _{\leq r} \xi _{(l)}^{i} =0$. We introduce a filtration degree on the algebra $\mathfrak{h}_{\leq r}$ by declaring it on generators
\begin{align*}
\mathrm{fd} (\xi _{(l)}^{i} ):=|\xi _{(l)}^{i} |=l,\ \ \mathrm{fd} (x_{i}^{(l)} ):=0.
\end{align*}
Let $\mathcal{F}^{p}\mathfrak{h}_{\leq r}$ be the $R_{\leq r}$-span of $\{X\in \mathfrak{h}_{\leq r} \mid \mathrm{fd} (X)\geq p\}$. The collection $(\mathcal{F}^{p}\mathfrak{h}_{\leq r} )_{p\geq 0}$ forms a descending Hausdorff filtration such that $\partial _{\leq r} (\mathcal{F}^{p}\mathfrak{h}_{\leq r} )\subseteq \mathcal{F}^{p}\mathfrak{h}_{\leq r}$, i.e, $\mathfrak{h}_{\leq r}$ is a filtered complex. We use the convention that if $p< 0$ then $\mathcal{F}^{p}\mathfrak{h}_{\leq r} =\mathcal{F}^{0}\mathfrak{h}_{\leq r}$.
There is a unique bracket $\llbracket \ ,\ \rrbracket $ of cohomological degree $-1$ extending the supercommutator that makes $\mathfrak{h}_{\leq r}$ into a Gerstenhaber algebra. It is referred to as the \emph{Schouten bracket}.
A convenient formula for calculating the Schouten bracket is (see \cite{cattaneo2007relative})
\begin{align*}
\llbracket X,Y\rrbracket =\sum _{l=0}^{r}\sum _{i\in \mathcal{I}_{l}} X\frac{\overleftarrow{\partial }}{\partial \xi _{(l)}^{i}}\frac{\overrightarrow{\partial }}{\partial x_{i}^{(l)}} Y-X\frac{\overleftarrow{\partial }}{\partial x_{i}^{(l)}}\frac{\overrightarrow{\partial }}{\partial \xi _{(l)}^{i}} Y.
\end{align*}
Putting $\mathfrak{h} :=\cup _{r\geq 0}\mathfrak{h}_{\leq r}$, we see that $(\mathfrak{h}_{\leq r} )_{r\geq 0}$ forms a directed system of  Gerstenhaber algebras. The Schouten bracket on the direct limit $\mathfrak{h} :=\cup _{r\geq 0}\mathfrak{h}_{\leq r}$ is given by
\begin{align*}
\llbracket X,Y\rrbracket =\sum _{m=0}^{\infty }\sum _{i\in \mathcal{I}_{m}} X\frac{\overleftarrow{\partial }}{\partial \xi _{(l)}^{i}}\frac{\overrightarrow{\partial }}{\partial x_{i}^{(l)}} Y-X\frac{\overleftarrow{\partial }}{\partial x_{i}^{(l)}}\frac{\overrightarrow{\partial }}{\partial \xi _{(l)}^{i}} Y.
\end{align*}
The canonical isomorphism $\left(\Der_{R_{\leq r}}( R_{\leq r})\right) [-1]\cong \mathrm{\Hom}_{R_{\leq r}} (\Omega_{R_{\leq r} |\boldsymbol{k}} [1],R_{\leq r} )$ extents to an injective morphism of $R$-modules $\mathfrak{h} =\cup _{r\geq 0}\mathfrak{h}_{\leq r}\rightarrow \mathfrak{g} :=\operatorname{Sym}_{R} (\Omega_{R|\boldsymbol{k}} [1],R)$. The $\kk$-vector space $\mathfrak{g}$ is the completion of $\mathfrak{h}$ in the $\mathcal{F}$-adic topology. There is a unique structure of a Gerstenhaber algebra on $\mathfrak{g}$ such that $\mathfrak{h}\rightarrow \mathfrak{g}$ is a morphism of Gerstenhaber algebras (this is \cite[Proposition 6.1]{higherKoszul}). Moreover, the $\mathcal{F}$-adic completion of $\mathfrak{h}_{\leq r}$ is $\mathfrak{g}_{\leq r} :=\operatorname{Sym}_{R_{\leq r}} (\Omega_{R_{\leq r} |\kk} [1],R_{\leq r} )$ and the collection $(\mathfrak{g}_{\leq r} )_{r\geq 0}$ forms a directed system of $\mathcal{F}$-adically complete Gerstenhaber algebras.

\section{Proofs of the main results}
\label{sec:main}

The main idea is now to construct for $j\geq 0$ elements $\pi _{j } \in \mathcal{F}^{j +m-1}\mathfrak{h}_{\leq j -1}\subseteq\mathfrak g$. We put 
\begin{align*}
    \pi_0=\sum_{l=1}^{\infty}\sum_{i\in\mathcal{I}_{l}}F_{i}(\boldsymbol{x}_{\leq l-1})\dfrac{\partial}{\partial x_{i}^{(l)}}\in \mathfrak g\ \ \ \mbox{and}\ \ \ \pi_1=\Pi.
\end{align*}
Note that the series above converges. For the sake of determining $\pi _{\ell }$ it is enough to use the truncation
\begin{align*}\pi_0^{\leq j-1}=\sum_{l=1}^{j-1}\sum_{i\in\mathcal{I}_{l}}F_{i}(\boldsymbol{x}_{\leq l-1})\dfrac{\partial}{\partial x_{i}^{(l)}}.
\end{align*}
The main step of the construction will be done in Lemma \ref{lem:main}. The method employed is called homological perturbation theory and is widely used nowadays (e.g., in Fedosov's deformation quantization, calculation of the BFV charge or of the BV action). We adjusted the procedure of \cite{higherKoszul} to the Nambu-Poisson situation. 
\begin{lemma}[\cite{higherKoszul}]\label{lem:filtration}
Let $p,q,r,s$ be integers $\ge 0$. Then
\begin{enumerate}
\item for each $X\in \mathcal F^p\mathfrak h_{\le r}$ we have $\llbracket\pi_0^{\le r},X\rrbracket  \in\partial_{\le r}X+\mathcal F^{p+1}\mathfrak h_{\le r}$, and
\item $\left\llbracket\mathcal F^p\mathfrak h_{\le r},\mathcal F^q\mathfrak h_{\le s}\right\rrbracket\subseteq \mathcal F^{p+q-1-\min(r,s)}\mathfrak h_{\le \max(r,s)}$. \qedhere
\end{enumerate}
\end{lemma}

\begin{lemma}\label{lem:main}
For $\ell \geq 2$ we can define recursively $\pi _{\ell } \in \mathcal{F}^{\ell +m-1}\mathfrak{h}_{\leq \ell -1}$  such that
\begin{align}\label{eq:mainlemma}
 \left\llbracket \pi _{0}^{\leq \ell -1} +\sum _{i=1}^{\ell -1} \pi _{i} ,\pi _{0}^{\leq \ell -1} +\sum _{i=1}^{\ell -1} \pi _{i} \right\rrbracket +\mathcal{F}^{\ell +m-1}\mathfrak{h}_{\leq \ell -1} =-2\partial _{\leq \ell -1} \pi _{\ell } +\mathcal{F}^{\ell +m-1}\mathfrak{h}_{\leq \ell -1}.
 \end{align}
 \end{lemma}

\begin{proof} We proof the claim by induction on $ \ell \geq 2$. Recall (see, e.g., \cite{takhtajan1994foundation,vaisman1999survey}) that  $\llbracket\pi_{1},\pi_{1}\rrbracket=0$.
If  $ \ell =2$ we have
\begin{align*}
 \llbracket \pi _{0}^{\leq 1} +\pi _{1} ,\pi _{0}^{\leq 1} +\pi _{1} \rrbracket &=\llbracket \pi _{0}^{\leq 1} ,\pi _{0}^{\leq 1} \rrbracket +2\llbracket \pi _{1} ,\pi _{0}^{\leq 1} \rrbracket +\llbracket \pi _{1} ,\pi _{1} \rrbracket \\
&=\frac{2}{( m-1) !} 
\sum_{i_{1} \dotsc i_{m-1},\mu,\nu}
\{x_{i_{1}} ,x_{i_{2}} ,\dotsc ,x_{i_{m-1}} ,f_{\mu }\} \xi ^{i_{1}} \xi ^{i_{2}} \cdots \xi ^{i_{m-1}} \xi _{( 1)}^{\mu } +\mathcal{F}^{2m-1}\mathfrak{h}_{\leq 1}
\\
 &=\frac{2}{( m-1) !}
 \sum_{i_{1} \dotsc i_{m-1},\mu,\nu}
 Z_{i_{1} \dotsc i_{m-1} \mu }^{\nu } f_{\nu } \xi ^{i_{1}} \xi ^{i_{2}} \cdots \xi ^{i_{m-1}} \xi _{( 1)}^{\mu } +\mathcal{F}^{2m-1}\mathfrak{h}_{\leq 1}\\
 &=-2\partial _{\leq 1} \pi _{2} +\mathcal{F}^{2m-1}\mathfrak{h}_{\leq 1}
 \end{align*}
putting  $\pi _{2} :=-\frac{1}{( m-1) !} \sum_{i_{1} \dotsc i_{m-1},\mu,\nu}Z_{i_{1} \dotsc i_{m-1} \mu }^{\nu } x_{\nu }^{( 1)} \xi ^{i_{1}} \xi ^{i_{2}} \cdots \xi ^{i_{m-1}} \xi _{( 1)}^{\mu } \in \mathcal{F}^{m+1}\mathfrak{h}_{\leq 1}$.

Let us assume that the claim holds for $\ell \geq 2$. We have to find $\pi _{\ell +1}$ such that \eqref{eq:mainlemma} holds after substituting $\ell \mapsto \ell +1$. Let us write $X_{\ell } :=\pi _{0}^{\leq \ell } -\pi _{0}^{\leq \ell -1}$ and decompose
\begin{align*}
 & \left\llbracket \pi _{0}^{\leq \ell } +\sum _{i=1}^{\ell } \pi _{i} ,\pi _{0}^{\leq \ell } +\sum _{j=1}^{\ell } \pi _{j} \right\rrbracket =\left\llbracket X_{\ell } +\pi _{\ell } +\pi _{0}^{\leq \ell -1} +\sum _{i=1}^{\ell -1} \pi _{i} ,X_{\ell } +\pi _{\ell } +\pi _{0}^{\leq \ell -1} +\sum _{j=1}^{\ell -1} \pi _{j} \right\rrbracket \\
 & =\left\llbracket \pi _{0}^{\leq \ell -1} +\sum _{i=1}^{\ell -1} \pi _{i} ,\pi _{0}^{\leq \ell -1} +\sum _{j=1}^{\ell -1} \pi _{j} \right\rrbracket +\llbracket X_{\ell } ,X_{\ell } \rrbracket  +\llbracket \pi _{\ell } ,\pi _{\ell }\rrbracket \\
 & +2\left\llbracket X_{\ell } ,\pi _{0}^{\leq \ell -1} +\sum _{i=1}^{\ell -1} \pi _{i} \right\rrbracket +2\left\llbracket \pi _{\ell } ,\pi _{0}^{\leq \ell} +\sum _{i=1}^{\ell -1} \pi _{i} \right\rrbracket .
\end{align*}
We need to verify that this is in $\mathcal{F}^{\ell +m}\mathfrak{h}_{\leq \ell }$. We have $\llbracket X_{\ell } ,X_{\ell } \rrbracket =0$ and using Lemma \ref{lem:filtration} we see that for $\ell \geq j$ we get $\llbracket \pi _{\ell } ,\pi _{j} \rrbracket \in \llbracket \mathcal{F}^{\ell +m-1}\mathfrak{h}_{\leq \ell -1} ,\mathcal{F}^{j+m-1}\mathfrak{h}_{\leq j-1} \rrbracket \in \mathcal{F}^{\ell +j+2m-3-j+1}\mathfrak{h}_{\leq \ell -1} \subset \mathcal{F}^{\ell +m}\mathfrak{h}_{\leq \ell -1}$.
Moreover, $\llbracket X_{\ell } ,\pi _{0}^{\leq \ell -1} \rrbracket =0$ since
\begin{align*}
0=2( \partial _{\leq \ell })^{2} =\llbracket \pi _{0}^{\leq \ell } ,\pi _{0}^{\leq \ell } \rrbracket =\llbracket X_{\ell } ,X_{\ell } \rrbracket +2\llbracket X_{\ell } ,\pi _{0}^{\leq \ell -1} \rrbracket +2( \partial _{\leq \ell -1})^{2} .
\end{align*}
Also, if $\ell>j$ it follows $\llbracket X_{\ell } ,\pi _{j} \rrbracket \in \llbracket \mathcal F^{\ell+1}\mathfrak{h}_{\leq \ell } ,\mathcal F^{j+m-1}\mathfrak{h}_{\leq j-1} \rrbracket \subseteq \mathcal F^{\ell +m}\mathfrak{h}_{\leq \ell }$. Finally, we have $\llbracket \pi _{\ell } ,\pi _{0}^{\leq \ell } \rrbracket \in\partial _{\leq \ell -1} \pi _{\ell } +\mathcal{F}^{\ell +m}\mathfrak{h}_{\leq \ell }$, establishing the claim.

Let now $A_{\ell } \in \mathfrak{h}_{\leq \ell -1}$ with $ \mathrm{fd}( A_{\ell }) =\ell +m$ such that $\left\llbracket \pi _{0}^{\leq \ell } +\sum _{i=1}^{\ell } \pi _{i} ,\pi _{0}^{\leq \ell } +\sum _{j=1}^{\ell } \pi _{j} \right\rrbracket \in A_{\ell } +\mathcal{F}^{\ell +m+1}\mathfrak{h}_{\leq \ell }$. The Jacobi identity for the Schouten bracket and Lemma \ref{lem:filtration} allow us to conclude that $\partial _{\leq \ell } A_{\ell } =0$. The argument goes as follows:
\begin{align*}
0=\left\llbracket \pi _{0}^{\leq \ell } +\sum _{i=1}^{\ell } \pi _{i} ,\left\llbracket \pi _{0}^{\leq \ell } +\sum _{i=1}^{\ell } \pi _{i} ,\pi _{0}^{\leq \ell } +\sum _{j=1}^{\ell } \pi _{j} \right\rrbracket \right\rrbracket \in \left\llbracket \pi _{0}^{\leq \ell } +\sum _{i=1}^{\ell } \pi _{i} ,A_{\ell } +\mathcal{F}^{\ell +m+1}\mathfrak{h}_{\leq \ell } \right\rrbracket .
\end{align*}
But for $i\geq 1$ we have $\llbracket \pi _{i} ,A_{\ell } \rrbracket \in \llbracket \mathcal{F}^{i+m-1}\mathfrak{h}_{\leq i-1} ,\mathcal{F}^{\ell +m}\mathfrak{h}_{\leq \ell } \rrbracket \subseteq \mathcal{F}^{\ell +i+2m-2-( i-1)}\mathfrak{h}_{\leq \ell } =\mathcal{F}^{\ell +2m-1}\mathfrak{h}_{\leq \ell }$ and $\llbracket \pi _{0}^{\leq \ell } ,A_{\ell } \rrbracket \in \partial _{\leq \ell } (A_{\ell } )+\mathcal{F}^{\ell +m+1}\mathfrak{h}_{\leq \ell }$. We observe that, by construction, $\partial _{\leq \ell } A_{\ell } =\partial _{\leq \ell -1} A_{\ell }$.
We choose $\pi _{\ell +1}$ such that $\partial _{\leq \ell } \pi _{\ell +1} =-A_{\ell } /2$.
\end{proof}

The Maurer-Cartan element of Sections \ref{sec:MC} and \ref{sec:conn} enters the recursion of Lemma \ref{lem:main}. Let us embed 
$\wedge \Der( S) \hookrightarrow \mathfrak{g} ,\ \partial /\partial x_{i} \mapsto \xi ^{i} =\xi _{( 0)}^{i}$ and not notationally distinguish an element of $\wedge \Der( S)$ from its image under the embedding. The corresponding
\begin{align}
\label{eq:MCpi}
    \sum _{\mu ,\nu }( \dN Z-[ Z,Z])_{\mu }^{\nu } x_{\nu }^{( 1)} \xi _{( 1)}^{\mu } \in \mathcal{F}^{2m-1}\mathfrak{g}
\end{align}
is $\partial$-closed by Theorem \ref{thm:MC}. The respective source term will be part of $\pi_{m+1}$.

We use the higher derived brackets (see Theorem \ref{thm:higher}) of Theodore Voronov (see also \cite{VoronovHigherDerived,KhudaverdianVoronov}) when $m=2$. More precisely, we put $\Gamma:=\Z$, $\mathfrak{G}:=\Pi\mathfrak g$, $\mathfrak{A}:=\Pi R$, $\varepsilon:=\Pi\epsilon$ where $\epsilon:\mathfrak g\to R$ is the augmentation and $\Delta:=\Pi\pi$, where $\pi=\sum_{\ell=0}^\infty \pi_\ell$. This is what had been done in \cite{higherKoszul} in order to produce a $P_\infty$-algebra structure on $R$. Now, for even $m>2$ we are facing the following problem:
\begin{align*}
1\leq \ell \leq m&:\ \pi _{\ell } \in \mathfrak{g}^{m}\\
m+1\leq \ell \leq 2m-1&:\ \pi _{\ell } \in \mathfrak{g}^{m} \oplus \mathfrak{g}^{2m-2}\\
2m\leq \ell \leq 3m-2&:\ \pi _{\ell } \in \mathfrak{g}^{m} \oplus \mathfrak{g}^{2m-2} \oplus \mathfrak{g}^{3m-4}\\
3m-1\leq \ell \leq 4m-3&:\ \pi _{\ell } \in \mathfrak{g}^{m} \oplus \mathfrak{g}^{2m-2} \oplus \mathfrak{g}^{3m-4} \oplus \mathfrak{g}^{4m-6}\mbox{ etc.}
\end{align*}
This is because every time $ \llbracket \pi _{1} ,\pi _{q} \rrbracket $ , $q=m+1,2m,3m-1,\dots$, has to be considered in the iteration a new cohomological
degree potentially opens up. We can write the statements more succinctly as
\begin{align}
\label{eq:spread}
    \pi _{\ell } \in \bigoplus _{r=1}^{\lfloor ( \ell -m-1) /( m-1) \rfloor +2}\mathfrak{g}^{rm-2( r-1)}
\end{align} 
for $\ell\ge 1$.
In particular $\pi= \sum _{\ell=0}^\infty \pi _{\ell }$ is not in $ \mathfrak{g}$ in general.

To remedy the situation we put $\Gamma=\Z_2$ when $m>2$ and introduce another filtration
\begin{align*}
F^{l}\mathfrak{g}:=\{X\in \mathfrak{g} \ |\ |X|\geq l\}, \ l\in\Z.
\end{align*}
With this $\mathfrak{g}=\mathfrak{g}_{\overline{0}}\oplus\mathfrak{g}_{\overline{1}}$ becomes a filtered $ \mathbb{Z}_{2}$-graded Gerstenhaber algebra in the sense that $\llbracket F^{k}\mathfrak{g} ,F^{l}\mathfrak{g} \rrbracket \subseteq F^{k+l+1}\mathfrak{g}$, $F^{k}\mathfrak{g} \cdot F^{l}\mathfrak{g} \subseteq F^{k+l}\mathfrak{g}$
and the usual axioms (see axioms (1)--(4) of Section \ref{sec:MC}) hold. The $F$-adic completion $ \widehat{\mathfrak{g}}$ is a $ \mathbb{Z}_{2}$-graded Gerstenhaber algebra as well and now $ \pi =\sum _{\ell} \pi _{\ell} \in  \widehat{\mathfrak{g}}_{\overline{0}}$.
Note that $ R=\widehat{R} \subset \widehat{\mathfrak{g}}$ since it is as a $ \mathbb{Z}$-graded algebra concentrated in non-positive degree. The augmentation $\epsilon :\mathfrak{g}\rightarrow R$ is filtered and gives rise to the projection $ \widehat{\epsilon } :\widehat{\mathfrak{g}}\rightarrow \widehat{R} =R$. The shifted
$ \mathfrak{G} :=\Pi \mathfrak{g}$ is a filtered Lie superalgebra $ \left[ F^{k}\mathfrak{G} ,F^{l}\mathfrak{G}\right] \subseteq F^{k+l}\mathfrak{G}$  with respect to the bracket $[\:,\:]=\Pi\circ \llbracket\:,\:\rrbracket\circ (\Pi^{-1}\otimes \Pi^{-1})$. Its completion is $ \widehat{\mathfrak{G}} =\Pi \widehat{\mathfrak{g}}$. Moreover,
$ \Pi R=\Pi \widehat{R}$ is an abelian subalgebra of $ \widehat{\mathfrak{G}}$ and we have the projection $ \varepsilon=\Pi\widehat{\epsilon} : \widehat{\mathfrak{G}}\rightarrow \Pi R$ satisfying
$ \varepsilon [ X,Y] =\varepsilon [ \varepsilon X,Y] +\varepsilon [ X,\varepsilon Y]$. Obviously, $ [ \Pi \pi ,\Pi \pi ] =0$.

Applying Theorem \ref{thm:higher} we find for $x_1,\dots, x_j\in \Pi R$ the operations $\left(l_j\right)_{j\ge 1}$
\begin{align}\label{eq:derived brackets}
l_j(x_1,x_2,\dots,x_j):=-\varepsilon\left(\left[\dots[[\Pi\pi,x_1],x_2],\dots,x_j\right]\right)
\end{align}
define a $\Gamma$-graded $\Pi L_\infty$-algebra on $\Pi R$.
Using the décalage  \eqref{eq:decalage} this induces a $\Gamma$-graded $P_\infty$-algebra structure $\left(\{\:,\dots,\:\}_j\right)_{j\ge 1}$,
\begin{align}
\nonumber
\{\phi_1,\dots,\phi_j\}_j&:=(-1)^{\sum_{i=1}^j(j-i)|\phi_i|}\Pi^{-1}(l_j(\Pi \phi_1,\dots,\Pi \phi_j)) \\
&=-(-1)^{\sum_{i=1}^j(j-i)|\phi_i|}\widehat{\epsilon}\left(\left\llbracket\dots\llbracket\llbracket \pi, \phi_1\rrbracket,\phi_2 \rrbracket,\dots,\phi _j\right\rrbracket\right),
\end{align}
on $R$, where $\widehat{\epsilon}:\widehat{\mathfrak g}\to R$ is the augmentation. If $m=2$ the hats can be omitted in the formulas above. 
The map $R\to A$ is actually a $\Gamma$-graded $L_\infty$-quasiisomorphism. 
That $\{\phi_1,\dots,\phi_j\}_j$ satisfy a Leibniz rule in each argument follows from the respective property of the Schouten bracket.
To complete the  proof Theorem \ref{thm:Pinfty}, note that in case $f_1,\dots,f_k$ are Casimir we have simply $\pi=\pi_{0}+\pi_1$. Hence $R$ becomes a dg Nambu-Poisson algebra.

The brackets of the $\Gamma$-graded $L_\infty$-algebroid is given by the formula $[\dR\phi_1,\dots,\dR\phi_l]_l=\dR\{\phi_1,\dots,\phi_l\}_l$. The  higher anchor $\rho_l(\dR\phi_1,\dots,\dR\phi_{l-1},a)$ is the image of $\{\phi_1,\dots,\phi_{l-1},g\}_{l}$, $a:=g+I$, under the map $R\to A$. Here the $\phi_i$ can be assumed to be linear in the coordinates $x_i^{(l)}$.
To prove Corollary \ref{cor:Linftyalgeboid} note that the necessary identities for the $\Gamma$-graded $L_\infty$-algebroid follow from the Jacobi identities of the $\Gamma$-graded $P_\infty$-structure. It remains to prove that all the higher anchors 
\begin{align*}
 a\mapsto\rho_l(\dR\phi_1,\dots,\dR\phi_{l-1},a)   
\end{align*}
vanish as soon as for one of the $\phi_i$, which are assumed to be linear in the $x_i^{(l)}$, is in the kernel of $\kappa:R\to S$. Let $\{\ ,\dots,\ \}_{j,l}$ be the part of bracket $\{\ ,\dots,\ \}_l$ that comes from $\pi_j$ and let us write $\rho_l(\dR\phi_1,\dots,\dR\phi_{l-1},a)_{j}=\{\phi_1,\dots,\phi_{l-1},g\}_{j,l}+I$, $a:=g+I$, keeping in mind that the $\phi_i$ are linear coordinates.
By the procedure of Lemma \ref{lem:main} we have for $j\geq 2$ each $\pi_j$ is a source term for $\partial$.
This means that $\{\phi_1,\dots,\phi_{l-1},g\}_{j,l}\in \ker(\kappa)$  whenever $j\geq 2$.
This obviously cannot be in $A$. Hence the only term that contributes to the anchor is $\pi_1$: 
\begin{align*}
    \rho_l(\dR\phi_1,\dots,\dR\phi_{l-1}, )=\begin{cases}
     X_{\phi_1,\dots,\phi_{l-1}}\mbox{ if }\phi_1,\dots,\phi_{l-1}\in A \mbox{ and }m=l,\\
     0 \mbox{ otherwise},
    \end{cases}
\end{align*}
which is $m$-ary.

To prove Corollary \ref{cor:formbr}  we put $W$ to be the $A$-submodule of $\mathbb L_{A|\kk}$ of elements of cohomological degree $0$. The restriction of $[\ ,\dots ,\ ]_m$ to $W$ is induced from the Nambu-Poisson tensor $\pi_1$. 
Moreover, in cohomological degree zero $\Bild[\ ]_1$ is an ideal. As this is the image $\overline{I\dR S+S\dR I}$ of $I\dR S+S\dR I\in\Omega_{R|\kk}$ in $\mathbb L_{A|\kk}$ we check
\begin{align*}
    [\overline{I\dR S+S\dR I},\overline{S\dR S},\dots ,\overline{S\dR S}]_m\subseteq \overline{S\{I,S,\dots,S\dR S+S\dR\{I,S,\dots,S\}}\subseteq \overline{I\dR S+S\dR I}.
\end{align*}
Hence the map $\mathbb L_{A|\kk}\to\Omega_{A|\kk}$ sends $[\ ,\dots ,\ ]_m$ to the form bracket on $\Omega_{A|\kk}$. In order to see that the anchor descends to $\Omega_{A|\kk}$ we verify
\begin{align*}
    \rho_m(\Bild[\ ]_1,\mathbb L_{A|\kk},\dots ,\mathbb L_{A|\kk},A)\subseteq I\{S,\dots,S\}+S\{I,S,\dots,S\}=0\in A=S/I.
\end{align*}
Now apply Proposition \ref{prop:minimal}.

\section{Complete and locally complete intersections}
\label{sec:complete}

If $ f_{1} ,\dotsc ,f_{k}$ forms a complete intersection, then $ R$ is the Koszul complex and the kernel $ \mathfrak{m}$ of $ R\rightarrow S$ is an ideal with the property $ \mathfrak{m}^{k+1} =0$. From this nilpotency one might expect that one can prove a priori that in the summation $\pi= \sum _{\ell } \pi _{\ell }$ only finitely many terms are non-vanishing (see \cite[Subsection 7.2]{higherKoszul} for the Poisson case). If $ m >2$ this seems not to be the case.
In contrast to the case $ m=2$, for $ m >2$ we have only $ \pi _{\ell } \in \bigoplus _{r=1}^{\lfloor ( \ell -m-1) /( m-1) \rfloor +2}\mathfrak{g}^{rm-2( r-1)}$ instead of $ \pi _{\ell } \in \mathfrak{g}^{2}$. But taking $ \lambda \in \mathcal{F}^{\ell +m-1}\mathfrak{g}^{rm-2( r-1)}$ we cannot prove a priori that $ \lambda $ vanishes for $ \ell $ large enough. The estimate we can get from the degree counting is $ \ell \geq k+( r-1) m+2\Longrightarrow \lambda =0$, which depends on $ r$ (the Poisson case $m=2$ corresponds to $r=1$). For $ \pi _{\ell }$ this would mean $ \ell \geq k+( \lfloor ( \ell -m-1) /( m-1) \rfloor +1) m+2\Longrightarrow \pi _{\ell } =0$ (see Eqn. \eqref{eq:spread}). But this inequality never holds. Even when the locus of the Nambu-Poisson ideal $I=(f_{1} ,\dotsc ,f_{k})$ is smooth we do not see an argument why the summation should be finite a priori. A complete intersection generated by Casimirs always provides a dg Nambu algebra structure on $R$ given by $\pi=\pi_0+\pi_1$.
Moreover, if for some reason the expression in \eqref{eq:MCpi} vanishes on the nose, then there is hope that the summation $ \sum _{\ell } \pi _{\ell }$ may turn out to be finite.

In the case of a complete intersection the cotangent complex is concentrated in cohomological degrees $0$ and $-1$, i.e., $\mathbb L_{A|\kk}=\mathbb L_{0}\oplus \mathbb L_{-1}$. For the $\Gamma$-graded $L_\infty$-algebroid structure on $\mathbb L_{A|\kk}$ certain parts of $\pi$ go to zero. Namely, all the terms in $\mathfrak m^2\widehat{\mathfrak{g}}$ do not contribute to the brackets $[\ ,\dots,\ ]_l$. Still, there seems to be no a priori reason why there should be only finitely many non-zero brackets. In the Poisson case there are no brackets of arity three or higher due to the restriction that come from the $\Z$-grading. 

In the case of a local complete intersection the resolvent is generated by variables of cohomological degrees $0,-1,-2$. This means that the ideal $\mathfrak m$ is not nilpotent. It is therefore impossible to make statements about the finiteness of the sum $\sum_\ell\pi_\ell$ just by counting the degrees even in the Poisson case $m=2$. We emphasize however that there is a projective resolvent of $A$ that is  generated by variables of cohomological degrees $0,-1$. In other words, it is typically possible to present the locus of $I$ by the vanishing of a section in a vector bundle in such a way that the Koszul complex of that bundle section is a resolution of $A=S/I$. Now the natural choice for $\pi_1$ instead of the Poisson tensor is the Rothstein Poisson tensor \cite{Rothstein,HerbigPhD}. It differs from the Poisson tensor by curvature terms in higher filtration degree, and the homological perturbation theory of Lemma \ref{lem:main} undergoes only minor modifications. The upshot being that in the case of a Poisson local complete intersection one can have a realization of the cotangent complex that is actually a dg Lie algebroid, and a resolvent whose $P_\infty$-structure is presented by a finite sum $\pi=\sum_\ell\pi_\ell$.

\section{Outlook}
\label{sec:outlook}

The natural question arises if something similar to our paper can be done when the base ring is the algebra $\mathcal C^\infty(\mathbb{R}^n)$ of smooth functions on $\R^n$. Under mild assumptions (e.g., $f_1,\dots, f_k$ should be coherent) the answer is yes, but one has to work in the framework of Fréchet algebras \cite{HerbigPhD}. Another possible generalization concerns the case when $\operatorname{char}(\kk)=p$. Here one has to replace the resolvent by the acyclic closure (see \cite{AvramovInfFree}). It is also possible to work with acyclic closures that are generated by projective modules instead of free ones. We mention that it is believed that the $P_\infty$-structure on the resolvent is unique up $P_\infty$-quasiisomorphism,
a statement that needs to be solidified. The authors intend to work out the details in the near future. At the moment it is mostly unclear how a deformation theory of Nambu-Poisson singularities should look like and whether the higher brackets play a role. One may also elaborate the cases when $\kk^n$ is replaced by a smooth scheme or $\R^n$ is replaced by a smooth manifold. 

In principle, there must be an algebraic structure on the André-Quillen homology
\begin{align*}
\operatorname{D}_{\cdot }( A|\boldsymbol{k} ,A) =\bigoplus _{n\geq 0}\operatorname{D}_{n}( A|\boldsymbol{k} ,A), \qquad \operatorname{D}_{n}( A|\boldsymbol{k} ,A) =H_{\partial }^{-n}(\mathbb{L}_{A|\boldsymbol{k}})
\end{align*} that is induced from the $\Gamma$-graded $L_\infty$-algebroid on $\mathbb L_{A|\kk}$. There is a catch here however. Every $L_\infty$-algebra $L$ decomposes non-canonically as a direct sum $L=L'\oplus L''$ of a minimal $L_\infty$-algebra $L'$ and a contractible one $L''$; see, e.g., \cite{Manchon}. The proof of \cite{Manchon} depends on the coalgebraic formulation for $L_\infty$-algebras. Such a coalgebraic formulation for $L_\infty$-algebroids will be elaborated in the forthcoming \cite{PoissonHomology} in the case when the  $A$-module $L$ is finitely generated projective in each degree. Yet for $L_\infty$-algebroids such as the cotangent complex
the André-Quillen homologies are typically intricate $A$-modules. To our knowledge, even when the  $A$-module $L$ is finitely generated projective in each degree such split a lá \cite{Manchon} in a minimal $L_\infty$-algebroid and a complementary contractible $L_\infty$-algebra $L''$ has not been worked out.

\section{Empirical Data}
We present here a couple of examples. They were elaborated by modifying the Mathematica code of \cite{higherKoszul}. In the examples no terms of arity $>m$ occur up to the depth we have been able to calculate. The catch is that in all examples that we looked at the Maurer-Cartan element turns out to be zero on the nose. It is not unlikely that all examples of invariants of finite subgroups of $\SL_3(\C)$ have a zero Maurer-Cartan element, as this is the case for the Kleinian singularities. In the case of Example \ref{ex:abelian} it takes some efforts to verify vanishing of the Maurer-Cartan element; for lack of space we cannot provide the details here. For Example \ref{ex:E} we were not able to do a similar analysis as the expressions are very bulky. For diagonal brackets Maurer-Cartan element appears to vanish too, we will try to prove this on another occasion. We hope that for 
invariant theory examples constructed from Nambu-Poisson Lie groups we might get examples of singular Nambu-Poisson algebras with nonzero Maurer-Cartan element and plan to investigate this elsewhere (this idea we got from a discussion with Ulrich Krähmer). In the Poisson case there are plenty of examples with non-vanishing Maurer-Cartan element \cite{higherKoszul}.

\subsection{A complete intersection generated by two monomials in dimension $4$}\label{sec:emp}
Regard the ideal in $\bs k[x_1,x_2,x_3,x_4]$ generated by the two quadratic  monomials $f_1 = x_1x_2$ and $f_2 = x_3 x_4$ with diagonal Nambu-Poisson bracket $\{x_1,x_2,x_3,x_4\}=x_1x_2x_3x_4$. Here we are working with the $\pi_0^{\leq 7}$ from \cite[Subsection 7.1]{higherKoszul}.
\begin{align*}
  \pi_1&=x^{(0)}_1 x^{(0)}_2 x^{(0)}_3 x^{(0)}_4 \xi^{1}_{(0)}\xi^{2} _{(0)}\xi^{3} _{(0)}\xi ^{4} _{(0)} \\
 \pi_2&=-x^{(0)}_1 x^{(0)}_2 x^{(0)}_3 x^{(1)}_2 \xi^{1}_{(0)} \xi^{2}_{(0)} \xi^{3}_{(0)} \xi_{(1)}^2+ x^{(0)}_1x^{(0)}_2x^{(0)}_4 x^{(1)}_2 \xi^{1}_{(0)} \xi^{2}_{(0)} \xi^{4}_{(0)} \xi_{(1)}^2 - x^{(0)}_1 x^{(0)}_3 x^{(0)}_4 x^{(1)}_1 \xi^{1}_{(0)} \xi^{3}_{(0)}\xi^{4}_{(0)}\xi_{(1)}^1\\
& + x^{(0)}_2 x^{(0)}_3 x^{(0)}_4 x^{(1)}_1 \xi^{2}_{(0)}\xi^{3}_{(0)}\xi^{4}_{(0)}\xi_{(1)}^1,\\
\pi_3&=-x^{(0)}_1x^{(0)}_3 x^{(1)}_1 x^{(1)}_2\xi^{1}_{(0)} \xi^{3}_{(0)}\xi_{(1)}^1 \xi_{(1)}^2 + x^{(0)}_1 x^{(0)}_4 x^{(1)}_1 x^{(1)}_2\xi^{1}_{(0)} \xi^{4}_{(0)} \xi_{(1)}^1 \xi_{(1)}^2 + x^{(0)}_2x^{(0)}_3 x^{(1)}_1 x^{(1)}_2 \xi^2_{(0)} \xi^3_{(0)} \xi_{(1)}^1\xi_{(1)}^2 \\
&- x^{(0)}_2x^{(0)}_4x^{(1)}_1 x^{(1)}_2\xi^{2} _{(0)}\xi^{4} _{(0)}\xi_{(1)}^1\xi_{(1)}^2,\\
\pi_i&=0\mbox{ for all }i\geq 4.
\end{align*}
 \subsection{A non-complete intersection generated by two monomials in dimension $4$} We consider the ideal in $\bs k[x_1,x_2,x_3,x_4]$ generated by the two quadratic  monomials $f_1 = x_1^2$ and $f_2 = x_1 x_2$ with diagonal Nambu-Poisson bracket $\{x_1,x_2,x_3,x_4\}=x_1x_2x_3x_4$. Here we are working with the $\pi_0^{\leq 7}$ from \cite[Subsection 7.1]{higherKoszul}.
 We were able to determine $\pi_i$ for $i\leq 5$:
\begin{align*}
  \pi_1&=x_1^{(0)}x_2^{(0)}x_3^{(0)}x_4^{(0)}\xi_{(0)}^1\xi _{(0)}^2\xi _{(0)}^3\xi _{(0)}^4 ,\\
 \pi_2&=-x_1^{(0)}x_3^{(0)}x_4^{(0)}x_2^{(1)}\xi
   _{(0)}^1\xi _{(0)}^3\xi _{(0)}^4\xi _{(1)}^2+2x_2^{(0)}x_3^{(0)}x_4^{(0)}x_1^{(1)}\xi
   _{(0)}^2\xi _{(0)}^3\xi _{(0)}^4\xi
   _{(1)}^1
   +x_2^{(0)}x_3^{(0)}x_4^{(0)}x_2^{(1)}\xi _{(0)}^2\xi _{(0)}^3\xi _{(0)}^4\xi _{(1)}^2, \\
\pi_3&=x_1^{(0)}x_3^{(0)}x_4^{(0)}x_1^{(2)}\xi _{(0)}^1\xi
   _{(0)}^3\xi _{(0)}^4\xi _{(2)}^1+2
   x_1^{(0)}x_3^{(0)}x_4^{(0)}x_1^{(2)}\xi _{(0)}^3\xi
   _{(0)}^4\xi _{(1)}^1\xi _{(1)}^2
   -2x_2^{(0)}x_3^{(0)}x_4^{(0)}x_1^{(2)}\xi _{(0)}^2\xi
   _{(0)}^3\xi _{(0)}^4\xi _{(2)}^1,\\
\pi_4&=-x_1^{(0)}x_3^{(0)}x_4^{(0)}x_1^{(3)}\xi _{(0)}^1\xi
   _{(0)}^3\xi _{(0)}^4\xi _{(3)}^1+3
   x_2^{(0)}x_3^{(0)}x_4^{(0)}x_1^{(3)}\xi _{(0)}^2\xi
   _{(0)}^3\xi _{(0)}^4\xi
   _{(3)}^1+x_2^{(0)}x_3^{(0)}x_4^{(0)}x_1^{(3)}\xi
   _{(0)}^3\xi _{(0)}^4\xi _{(1)}^2\xi _{(2)}^1,\\
\pi_5&=2 x_1^{(0)}x_3^{(0)}x_4^{(0)}x_1^{(4)}\xi _{(0)}^1\xi
   _{(0)}^3\xi _{(0)}^4\xi _{(4)}^1-3
   x_1^{(0)}x_3^{(0)}x_4^{(0)}x_1^{(4)}\xi _{(0)}^3\xi
   _{(0)}^4\xi _{(1)}^2\xi
   _{(3)}^1+x_1^{(0)}x_3^{(0)}x_4^{(0)}x_2^{(4)}\xi
   _{(0)}^1\xi _{(0)}^3\xi _{(0)}^4\xi _{(4)}^2,\\
   &+2x_1^{(0)}x_3^{(0)}x_4^{(0)}x_2^{(4)}\xi _{(0)}^3\xi
   _{(0)}^4\xi _{(1)}^1\xi _{(3)}^1-3
   x_2^{(0)}x_3^{(0)}x_4^{(0)}x_1^{(4)}\xi _{(0)}^2\xi
   _{(0)}^3\xi _{(0)}^4\xi _{(4)}^1-4
   x_2^{(0)}x_3^{(0)}x_4^{(0)}x_2^{(4)}\xi _{(0)}^2\xi
   _{(0)}^3\xi _{(0)}^4\xi _{(4)}^2.
\end{align*}
\subsection{Angular momentum type equations in dimension $7$} We consider the ideal in $\bs k[x_1,x_2,\dots,x_7]$ generated by the three components of  
\begin{align*}
    \begin{pmatrix} f_1\\f_2\\f_3\end{pmatrix}:=\begin{pmatrix} x_1\\x_2\\x_3\end{pmatrix}\times \begin{pmatrix} x_4\\x_5\\x_6\end{pmatrix},
\end{align*}
i.e., $f_1 = x_2 x_6-x_3x_5$, $f_2 = x_3 x_4-x_1x_6$, $f_3 = x_1 x_5-x_2x_4$, with Nambu-Poisson bracket the determinantal bracket of arity $4$ given by
\begin{align*}
    \{a_1,a_2,a_3,a_4\}&=
    \begin{vmatrix} 
    0 & -x_6 & x_5 & \dfrac{\partial a_1}{\partial x_1} & \dfrac{\partial a_2}{\partial x_1} & \dfrac{\partial a_3}{\partial x_1} & \dfrac{\partial a_4}{\partial x_1}\\
    x_6 & 0 & -x_4 & \dfrac{\partial a_1}{\partial x_2} & \dfrac{\partial a_2}{\partial x_2} & \dfrac{\partial a_3}{\partial x_2} & \dfrac{\partial a_4}{\partial x_2}\\
    -x_5 & x_4 & 0 & \dfrac{\partial a_1}{\partial x_3} & \dfrac{\partial a_2}{\partial x_3} & \dfrac{\partial a_3}{\partial x_3} & \dfrac{\partial a_4}{\partial x_3}\\
    0 & x_3 & -x_2 & \dfrac{\partial a_1}{\partial x_4} & \dfrac{\partial a_2}{\partial x_4} & \dfrac{\partial a_3}{\partial x_4} & \dfrac{\partial a_4}{\partial x_4}\\
    -x_3 & 0 & x_1 & \dfrac{\partial a_1}{\partial x_5} & \dfrac{\partial a_2}{\partial x_5} & \dfrac{\partial a_3}{\partial x_5} & \dfrac{\partial a_4}{\partial x_5}\\
    x_2 & -x_1 & 0 & \dfrac{\partial a_1}{\partial x_6} & \dfrac{\partial a_2}{\partial x_6} & \dfrac{\partial a_3}{\partial x_6} & \dfrac{\partial a_4}{\partial x_6}\\
    0 & 0 & 0 & \dfrac{\partial a_1}{\partial x_7} & \dfrac{\partial a_2}{\partial x_7} & \dfrac{\partial a_3}{\partial x_7} & \dfrac{\partial a_4}{\partial x_7}
    \end{vmatrix}={\partial(f_1,f_2,f_3,a_1,a_2,a_3,a_4)\over \partial(x_1,x_2,x_3,x_4,x_5,x_6,x_7)}.
\end{align*}
The following formulas for $\pi_0^{\leq 1},...,\pi_0^{\leq 4}$:
\begin{align*}
    \pi_0^{\leq 1}&=x_1^{(0)}x_5^{(0)}\xi _{(1)}^3-x_1^{(0)}x_6^{(0)}\xi
   _{(1)}^2-x_2^{(0)}x_4^{(0)}\xi _{(1)}^3+x_2^{(0)}x_6^{(0)}\xi
   _{(1)}^1+x_3^{(0)}x_4^{(0)}\xi _{(1)}^2-x_3^{(0)}x_5^{(0)}\xi
   _{(1)}^1\\
    \pi_0^{\leq 2}&=x_1^{(0)}x_1^{(1)}\xi _{(2)}^2+x_2^{(0)}x_2^{(1)}\xi
   _{(2)}^2+x_3^{(0)}x_3^{(1)}\xi _{(2)}^2+x_4^{(0)}x_1^{(1)}\xi
   _{(2)}^1+x_5^{(0)}x_2^{(1)}\xi _{(2)}^1+x_6^{(0)}x_3^{(1)}\xi
   _{(2)}^1\\
    \pi_0^{\leq 3}&=x_1^{(0)}x_1^{(2)}\xi _{(3)}^3+x_1^{(1)}x_2^{(1)}\xi
   _{(3)}^1-x_1^{(1)}x_3^{(1)}\xi _{(3)}^2+x_2^{(0)}x_1^{(2)}\xi
   _{(3)}^2+x_2^{(1)}x_3^{(1)}\xi _{(3)}^3+x_3^{(0)}x_1^{(2)}\xi
   _{(3)}^1-x_4^{(0)}x_2^{(2)}\xi _{(3)}^3\\
   &-x_5^{(0)}x_2^{(2)}\xi
   _{(3)}^2-x_6^{(0)}x_2^{(2)}\xi _{(3)}^1\\
    \pi_0^{\leq 4}&= x_1^{(0)}x_1^{(3)}\xi _{(4)}^5+x_1^{(0)}x_2^{(3)}\xi
   _{(4)}^6+x_1^{(1)}x_1^{(2)}\xi _{(4)}^1+x_1^{(1)}x_2^{(2)}\xi
   _{(4)}^2+x_2^{(0)}x_1^{(3)}\xi _{(4)}^2-x_2^{(0)}x_3^{(3)}\xi
   _{(4)}^6-x_2^{(1)}x_1^{(2)}\xi _{(4)}^3\\
   &-x_2^{(1)}x_2^{(2)}\xi
   _{(4)}^5-x_3^{(0)}x_2^{(3)}\xi _{(4)}^2-x_3^{(0)}x_3^{(3)}\xi
   _{(4)}^5+x_3^{(1)}x_1^{(2)}\xi _{(4)}^4+x_3^{(1)}x_2^{(2)}\xi
   _{(4)}^6+x_4^{(0)}x_1^{(3)}\xi _{(4)}^3+x_4^{(0)}x_2^{(3)}\xi
   _{(4)}^4\\
   &+x_5^{(0)}x_1^{(3)}\xi _{(4)}^1-x_5^{(0)}x_3^{(3)}\xi
   _{(4)}^4-x_6^{(0)}x_2^{(3)}\xi _{(4)}^1-x_6^{(0)}x_3^{(3)}\xi
   _{(4)}^3.
\end{align*}
we determined using the package \texttt{dgalgebras} of Macaulay2 \cite{M2}. We find
 \begin{align*}
  \pi_1&=-x_1^{(0)}x_1^{(0)}x_5^{(0)}\xi _{(0)}^1\xi _{(0)}^2\xi
   _{(0)}^4\xi _{(0)}^7-x_1^{(0)}x_1^{(0)}x_6^{(0)}\xi
   _{(0)}^1\xi _{(0)}^3\xi _{(0)}^4\xi _{(0)}^7+x_1^{(0)}x_2^{(0)}x_4^{(0)}\xi _{(0)}^1\xi
   _{(0)}^2\xi _{(0)}^4\xi
   _{(0)}^7\\
   &-x_1^{(0)}x_2^{(0)}x_5^{(0)}\xi _{(0)}^1\xi_{(0)}^2\xi_{(0)}^5\xi_{(0)}^7-x_1^{(0)}x_2^{(0)}x_6^{(0)}\xi _{(0)}^1\xi_{(0)}^3\xi _{(0)}^5\xi_{(0)}^7-x_1^{(0)}x_2^{(0)}x_6^{(0)}\xi _{(0)}^2\xi_{(0)}^3\xi _{(0)}^4\xi
   _{(0)}^7\\
   &+x_1^{(0)}x_3^{(0)}x_4^{(0)}\xi _{(0)}^1\xi_{(0)}^3\xi _{(0)}^4\xi_{(0)}^7-x_1^{(0)}x_3^{(0)}x_5^{(0)}\xi _{(0)}^1\xi _{(0)}^2\xi _{(0)}^6\xi_{(0)}^7+x_1^{(0)}x_3^{(0)}x_5^{(0)}\xi _{(0)}^2\xi_{(0)}^3\xi _{(0)}^4\xi _{(0)}^7\\
   &-x_1^{(0)}x_3^{(0)}x_6^{(0)}\xi _{(0)}^1\xi_{(0)}^3\xi _{(0)}^6\xi_{(0)}^7+x_1^{(0)}x_4^{(0)}x_5^{(0)}\xi _{(0)}^1\xi
   _{(0)}^4\xi _{(0)}^5\xi_{(0)}^7+x_1^{(0)}x_4^{(0)}x_6^{(0)}\xi _{(0)}^1\xi_{(0)}^4\xi _{(0)}^6\xi_{(0)}^7\\
   &+x_1^{(0)}x_5^{(0)}x_5^{(0)}\xi _{(0)}^2\xi _{(0)}^4\xi _{(0)}^5\xi_{(0)}^7+x_1^{(0)}x_5^{(0)}x_6^{(0)}\xi _{(0)}^2\xi   _{(0)}^4\xi _{(0)}^6\xi _{(0)}^7+x_1^{(0)}x_5^{(0)}x_6^{(0)}\xi _{(0)}^3\xi_{(0)}^4\xi _{(0)}^5\xi
   _{(0)}^7\\
   &+x_1^{(0)}x_6^{(0)}x_6^{(0)}\xi _{(0)}^3\xi_{(0)}^4\xi _{(0)}^6\xi_{(0)}^7+x_2^{(0)}x_2^{(0)}x_4^{(0)}\xi _{(0)}^1\xi
   _{(0)}^2\xi _{(0)}^5\xi_{(0)}^7-x_2^{(0)}x_2^{(0)}x_6^{(0)}\xi _{(0)}^2\xi_{(0)}^3\xi _{(0)}^5\xi_{(0)}^7\\
   &+x_2^{(0)}x_3^{(0)}x_4^{(0)}\xi _{(0)}^1\xi_{(0)}^2\xi _{(0)}^6\xi_{(0)}^7+x_2^{(0)}x_3^{(0)}x_4^{(0)}\xi _{(0)}^1\xi   _{(0)}^3\xi _{(0)}^5\xi_{(0)}^7+x_2^{(0)}x_3^{(0)}x_5^{(0)}\xi _{(0)}^2\xi _{(0)}^3\xi _{(0)}^5\xi_{(0)}^7\\
   &-x_2^{(0)}x_3^{(0)}x_6^{(0)}\xi _{(0)}^2\xi_{(0)}^3\xi _{(0)}^6\xi_{(0)}^7-x_2^{(0)}x_4^{(0)}x_4^{(0)}\xi _{(0)}^1\xi
   _{(0)}^4\xi _{(0)}^5\xi_{(0)}^7-x_2^{(0)}x_4^{(0)}x_5^{(0)}\xi _{(0)}^2\xi_{(0)}^4\xi _{(0)}^5\xi_{(0)}^7\\
   &+x_2^{(0)}x_4^{(0)}x_6^{(0)}\xi _{(0)}^1\xi_{(0)}^5\xi _{(0)}^6\xi _{(0)}^7-x_2^{(0)}x_4^{(0)}x_6^{(0)}\xi _{(0)}^3\xi
   _{(0)}^4\xi _{(0)}^5\xi_{(0)}^7+x_2^{(0)}x_5^{(0)}x_6^{(0)}\xi _{(0)}^2\xi _{(0)}^5\xi _{(0)}^6\xi _{(0)}^7\\
   &+x_2^{(0)}x_6^{(0)}x_6^{(0)}\xi _{(0)}^3\xi _{(0)}^5\xi _{(0)}^6\xi _{(0)}^7+x_3^{(0)}x_3^{(0)}x_4^{(0)}\xi _{(0)}^1\xi  _{(0)}^3\xi _{(0)}^6\xi _{(0)}^7+x_3^{(0)}x_3^{(0)}x_5^{(0)}\xi _{(0)}^2\xi
   _{(0)}^3\xi _{(0)}^6\xi _{(0)}^7\\
   &-x_3^{(0)}x_4^{(0)}x_4^{(0)}\xi _{(0)}^1\xi _{(0)}^4\xi _{(0)}^6\xi _{(0)}^7-x_3^{(0)}x_4^{(0)}x_5^{(0)}\xi _{(0)}^1\xi
   _{(0)}^5\xi _{(0)}^6\xi_{(0)}^7-x_3^{(0)}x_4^{(0)}x_5^{(0)}\xi _{(0)}^2\xi _{(0)}^4\xi _{(0)}^6\xi_{(0)}^7\\
   &-x_3^{(0)}x_4^{(0)}x_6^{(0)}\xi _{(0)}^3\xi_{(0)}^4\xi _{(0)}^6\xi_{(0)}^7-x_3^{(0)}x_5^{(0)}x_5^{(0)}\xi_{(0)}^2\xi_{(0)}^5\xi_{(0)}^6\xi_{(0)}^7-x_3^{(0)}x_5^{(0)}x_6^{(0)}\xi _{(0)}^3\xi _{(0)}^5\xi _{(0)}^6\xi _{(0)}^7,\\
        \pi_2&=0, \qquad \pi_3=0,\\
\pi_4&=-x_1^{(0)}x_1^{(0)}x_1^{(2)}\xi _{(0)}^1\xi _{(0)}^4\xi
   _{(0)}^7\xi _{(2)}^2-x_1^{(0)}x_2^{(0)}x_1^{(2)}\xi _{(0)}^1\xi _{(0)}^5\xi _{(0)}^7\xi_{(2)}^2-x_1^{(0)}x_2^{(0)}x_1^{(2)}\xi _{(0)}^2\xi _{(0)}^4\xi _{(0)}^7\xi _{(2)}^2\\
   &-x_1^{(0)}x_3^{(0)}x_1^{(2)}\xi _{(0)}^1\xi_{(0)}^6xi _{(0)}^7\xi_{(2)}^2-x_1^{(0)}x_3^{(0)}x_1^{(2)}\xi _{(0)}^3\xi_{(0)}^4\xi _{(0)}^7\xi_{(2)}^2-x_1^{(0)}x_4^{(0)}x_1^{(2)}\xi _{(0)}^1\xi _{(0)}^4\xi _{(0)}^7\xi _{(2)}^1\\
   &+x_1^{(0)}x_4^{(0)}x_2^{(2)}\xi _{(0)}^1\xi_{(0)}^4\xi _{(0)}^7\xi_{(2)}^2-x_1^{(0)}x_5^{(0)}x_1^{(2)}\xi _{(0)}^2\xi
   _{(0)}^4\xi _{(0)}^7\xi _{(2)}^1-x_1^{(0)}x_5^{(0)}x_1^{(2)}\xi _{(0)}^4\xi _{(0)}^5\xi _{(0)}^7\xi_{(2)}^2\\
   &+x_1^{(0)}x_5^{(0)}x_2^{(2)}\xi _{(0)}^1\xi _{(0)}^2\xi _{(0)}^7\xi_{(2)}^1+x_1^{(0)}x_5^{(0)}x_2^{(2)}\xi _{(0)}^2\xi _{(0)}^4\xi _{(0)}^7\xi_{(2)}^2-x_1^{(0)}x_6^{(0)}x_1^{(2)}\xi _{(0)}^3\xi _{(0)}^4\xi _{(0)}^7\xi _{(2)}^1\\
   &-x_1^{(0)}x_6^{(0)}x_1^{(2)}\xi _{(0)}^4\xi_{(0)}^6\xi _{(0)}^7\xi_{(2)}^2+x_1^{(0)}x_6^{(0)}x_2^{(2)}\xi _{(0)}^1\xi _{(0)}^3\xi _{(0)}^7\xi_{(2)}^1+x_1^{(0)}x_6^{(0)}x_2^{(2)}\xi _{(0)}^3\xi_{(0)}^4\xi _{(0)}^7\xi _{(2)}^2\\
   &-x_2^{(0)}x_2^{(0)}x_1^{(2})\xi _{(0)}^2\xi_{(0)}^5\xi _{(0)}^7\xi_{(2)}^2-x_2^{(0)}x_3^{(0)}x_1^{(2)}\xi _{(0)}^2\xi _{(0)}^6\xi _{(0)}^7\xi _{(2)}^2\\
   &-x_2^{(0)}x_4^{(0)}x_1^{(2)}\xi _{(0)}^1\xi _{(0)}^5\xi _{(0)}^7\xi_{(2)}^1+x_2^{(0)}x_4^{(0)}x_1^{(2)}\xi _{(0)}^4\xi_{(0)}^5\xi _{(0)}^7\xi _{(2)}^2-x_2^{(0)}x_4^{(0)}x_2^{(2)}\xi _{(0)}^1\xi _{(0)}^2\xi _{(0)}^7\xi _{(2)}^1
   \end{align*}
   \begin{align*}
   &+x_2^{(0)}x_4^{(0)}x_2^{(2)}\xi _{(0)}^1\xi _{(0)}^5\xi _{(0)}^7\xi_{(2)}^2-x_2^{(0)}x_5^{(0)}x_1^{(2)}\xi _{(0)}^2\xi _{(0)}^5\xi _{(0)}^7\xi _{(2)}^1+x_2^{(0)}x_5^{(0)}x_2^{(2)}\xi _{(0)}^2\xi_{(0)}^5\xi _{(0)}^7\xi_{(2)}^2\\
   &-x_2^{(0)}x_6^{(0)}x_1^{(2)}\xi _{(0)}^3\xi_{(0)}^5\xi _{(0)}^7\xi _{(2)}^1-x_2^{(0)}x_6^{(0)}x_1^{(2)}\xi _{(0)}^5\xi _{(0)}^6\xi _{(0)}^7\xi _{(2)}^2+x_2^{(0)}x_6^{(0)}x_2^{(2)}\xi _{(0)}^2\xi_{(0)}^3\xi _{(0)}^7\xi _{(2)}^1\\
   &+x_2^{(0)}x_6^{(0)}x_2^{(2)}\xi _{(0)}^3\xi _{(0)}^5\xi _{(0)}^7\xi_{(2)}^2-x_3^{(0)}x_3^{(0)}x_1^{(2)}\xi _{(0)}^3\xi
   _{(0)}^6\xi _{(0)}^7\xi _{(2)}^2+x_3^{(0)}x_4^{(0)}x_1^{(2)}\xi _{(0)}^4\xi _{(0)}^6\xi _{(0)}^7\xi _{(2)}^2\\
   &-x_3^{(0)}x_4^{(0)}x_2^{(2)}\xi _{(0)}^1\xi _{(0)}^3\xi _{(0)}^7\xi_{(2)}^1+x_3^{(0)}x_4^{(0)}x_2^{(2)}\xi _{(0)}^1\xi _{(0)}^6\xi _{(0)}^7\xi_{(2)}^2-x_3^{(0)}x_5^{(0)}x_1^{(2)}\xi _{(0)}^2\xi _{(0)}^6\xi _{(0)}^7\xi _{(2)}^1\\
   &+x_3^{(0)}x_5^{(0)}x_1^{(2)}\xi _{(0)}^5\xi_{(0)}^6\xi _{(0)}^7\xi_{(2)}^2-x_3^{(0)}x_5^{(0)}x_2^{(2)}\xi _{(0)}^2\xi_{(0)}^3\xi _{(0)}^7\xi_{(2)}^1+x_3^{(0)}x_5^{(0)}x_2^{(2)}\xi _{(0)}^2\xi_{(0)}^6\xi _{(0)}^7\xi_{(2)}^2\\
   &-x_3^{(0)}x_6^{(0)}x_1^{(2)}\xi _{(0)}^3\xi _{(0)}^6\xi _{(0)}^7\xi _{(2)}^1+x_3^{(0)}x_6^{(0)}x_2^{(2)}\xi _{(0)}^3\xi_{(0)}^6\xi _{(0)}^7\xi_{(2)}^2+x_4^{(0)}x_1^{(1)}x_2^{(1)}\xi _{(0)}^1\xi_{(0)}^6\xi _{(0)}^7\xi_{(2)}^1\\
   &+x_4^{(0)}x_4^{(0)}x_2^{(2)}\xi _{(0)}^1\xi_{(0)}^4\xi _{(0)}^7\xi_{(2)}^1+x_4^{(0)}x_5^{(0)}x_2^{(2)}\xi _{(0)}^1\xi _{(0)}^5\xi _{(0)}^7\xi_{(2)}^1+x_4^{(0)}x_5^{(0)}x_2^{(2)}\xi _{(0)}^2\xi_{(0)}^4\xi _{(0)}^7\xi _{(2)}^1\\
   &+x_4^{(0)}x_6^{(0)}x_2^{(2)}\xi _{(0)}^3\xi _{(0)}^4\xi _{(0)}^7\xi_{(2)}^1+x_5^{(0)}x_5^{(0)}x_2^{(2)}\xi _{(0)}^2\xi
   _{(0)}^5\xi _{(0)}^7\xi_{(2)}^1+x_5^{(0)}x_6^{(0)}x_2^{(2)}\xi _{(0)}^2\xi
   _{(0)}^6\xi _{(0)}^7\xi_{(2)}^1\\
   &+x_5^{(0)}x_6^{(0)}x_2^{(2)}\xi _{(0)}^3\xi _{(0)}^5\xi _{(0)}^7\xi _{(2)}^1+x_6^{(0)}x_6^{(0)}x_2^{(2)}\xi _{(0)}^3\xi
   _{(0)}^6\xi _{(0)}^7\xi _{(2)}^1.
   \end{align*}

\bibliographystyle{amsplain}
\bibliography{Nambu.bib}
\end{document}